\documentclass[review]{elsarticle}
\usepackage{lineno}
\modulolinenumbers[5]

\makeatletter
\def\ps@pprintTitle{%
 \let\@oddhead\@empty
 \let\@evenhead\@empty
\def\@oddfoot{}%
 \let\@evenfoot\@oddfoot}
\makeatother
\usepackage{ucs}
\usepackage[utf8x]{inputenc}
\usepackage{amsmath}
\usepackage{amsfonts}
\usepackage[english]{babel}
\usepackage{fontenc}
\usepackage{lmodern}

\usepackage{amssymb}
\usepackage{amsthm}
\newtheorem{theorem}{Theorem}
\newtheorem{lemma}[theorem]{Lemma}

\newtheorem{assumption}[theorem]{Assumption}
\newtheorem{remark}[theorem]{Remark}

\DeclareMathOperator{\spann}{span}
\DeclareMathOperator{\divergence}{div}

\newcommand{\R}{\mathbb{R}}

\usepackage{tikz}
\usetikzlibrary{calc}
\usepackage{graphicx} 
\usepackage{hyperref}

\begin{document}

\begin{frontmatter}

\title{Isogeometric mortar methods}

\author[IMATI,IUSS]{Ericka Brivadis}	
\ead{ericka.brivadis@iusspavia.it}
\author[IMATI]{Annalisa Buffa} 
\ead{annalisa@imati.cnr.it}
\author[TUM]{Barbara Wohlmuth}
\ead{wohlmuth@ma.tum.de}
\author[TUM]{Linus Wunderlich\corref{mycorrespondingauthor}}
\cortext[mycorrespondingauthor]{Corresponding author}
\ead{ linus.wunderlich@ma.tum.de}

\address[IMATI]{Istituto di Matematica Applicata e Tecnologie Informatiche del CNR, Via Ferrata~1, 27100~Pavia, Italy } 
\address[IUSS]{Istituto Universitario di Studi Superiori Pavia, Palazzo del Broletto, Piazza della Vittoria~15, 27100~Pavia, Italy } 
\address[TUM]{M2 - Zentrum Mathematik, Technische Universit\"at M\"unchen, Boltzmannstra\ss{}e~3, 85748~Garching, Germany }

\begin{abstract}
The application of mortar methods in the framework of isogeometric analysis is investigated theoretically as well as numerically. 
For the Lagrange multiplier two choices of uniformly stable spaces are presented, both of them are spline spaces but of a different degree. 
In one case, we consider an equal order pairing for which a cross point modification based on a local degree reduction is required. In the other case, the degree of the dual space is reduced by two compared to the primal. This pairing is proven to be inf-sup stable without any necessary cross point modification.
Several numerical examples confirm the theoretical results and illustrate additional aspects. 

\end{abstract}

\end{frontmatter}

\linenumbers

\section{Introduction}
\label{chapter1}
The name isogeometric analysis was introduced in 2005 by  Hughes et al. in~\cite{hughes:05}. Nowadays it includes a family of methods, normally called isogeometric methods, that use B-Splines and non-uniform rational B-Splines (NURBS) as basis functions to construct numerical approximations of partial differential equations (PDEs). Originally, isogeometric analysis follows the isoparametric paradigm, i.e., the geometry is represented by functions which are used to approximate the PDE. In~\cite{beirao:14}, it was shown that this concept can be relaxed, also allowing NURBS for the parametrization and B-Splines defined on the same mesh for the approximation of the PDE.

With isogeometric methods, the computational domain is generally split into patches. Within this framework, techniques to couple the numerical solution on different patches are required. To retain the flexibility of the meshes at the interfaces, weak coupling methods are favorable in contrast to strong point-wise couplings. Thus it is interesting to consider mortar methods, which offer a flexible approach to domain decomposition, originally applied in spectral and finite element methods.  
Mortar methods have been successfully  investigated in the finite element context for over two decades,~\cite{bernardi:93, bernardi:94,ben_belgacem:97,ben_belgacem:99}, for a mathematical overview, see~\cite{wohlmuth:01}. Further applications of the mortar methods include contact problems,~\cite{wriggers:02, ben_belgacem:99b, krause:02, wohlmuth:11, wall:12}, and interface problems, e.g., in multi-physics applications,~\cite{kaltenbacher:10}. 
  
The isogeometric analysis, \cite{hughes:09, hughes:09b},  is currently a very active research area. It is attractive for a large variety of applications and 
there already exist a fair amount of mathematically sound results, recently collected in~\cite{beirao:14}. Besides variational approaches, the global smoothness of splines also allows the use of collocation methods, see~\cite{hughes:10}.

In several articles, the coupling of multipatch geometries has been investigated,~\cite{kleiss:12, nguyen:13,nguyen:14, schillinger:14, langer:14}, and successful applications of the mortar method are shown in~\cite{hesch:12,dornisch:14, bletzinger:14}. Additionally the use of mortar methods in contact simulations, where isogeometric methods have some advantages over finite element methods, was considered in~\cite{delorenzis:11, delorenzis:12, temizer:12, kim:12,hesch:14,temizer:14}.

The important point of an isogeometric mortar method is the choice of the Lagrange multiplier. From the classical mortar theory, two abstract requirements for the Lagrange multiplier space are given. One is the sufficient approximation order, the other is the requirement of an inf-sup stability. 
For a primal space of splines of degree p, we investigate three different degrees for the Lagrange multiplier: $p$, $p-1$ and $p-2$. Each choice is from some point of view natural but has quite different characteristic features.

This article is structured as follows. In Section~\ref{chapter2}, we recall basic properties of isogeometric methods.  The isogeometric mortar methods is then defined in Section~\ref{chapter3}. In Section~\ref{chapter4}, we complete the definition of our mortar methods by explicitly detailing three different types of Lagrange multipliers.  The theoretical results are investigated numerically in Section~\ref{chapter5}, where also additional aspects are considered.

\section{B-Splines and NURBS basics}\label{sec:isogeometric_method}
\label{chapter2}
In this section, we give a brief overview on the isogeometric functions and introduce some notations and concepts which are used throughout the paper. For more details, we refer to the classical literature~\cite{hughes:09,piegl:97,bazilevs:06,Schumaker:07}. Firstly, we introduce B-Splines in the one-dimensional case and recall some of their basic properties. Secondly, we extend these definitions to the multi-dimensional case and introduce NURBS and then NURBS parametrizations.

\subsection{Univariate B-Splines}\label{sec:iga_p1}
Let us denote by $p$ the degree of the univariate B-Splines and by $\Xi $ an open uni\-variate knot vector, where the first and last entries are repeated $(p+1)$-times, i.e.,
\begin{equation*}
	\Xi = \{ 0 = \xi_1 = \ldots = \xi_{p+1} < \xi_{p+2} \leq \ldots \leq \xi_{n} < \xi_{n+1} = \ldots= \xi_{n+p+1} = 1\}.
\end{equation*}
Let us define $Z = \{\zeta_1,\, \zeta_2,\, \ldots,\, \zeta_{E}\}$ as the knot vector without any repetition, also called breakpoint vector.
For each breakpoint $\zeta_j$ of $Z$, we define its multiplicity $m_j$ as its number of repetitions in $\Xi$. The elements in $Z$ form a partition of the parametric interval $(0, 1)$, i.e., a mesh. 

We denote by $\widehat{B}_i^p(\zeta)$, $i=1,\ldots, n$, the collection of B-Splines defined on $\Xi$ and by $S^p(\Xi)=\spann\{\widehat{B}_i^p(\zeta), \,i=1,\, \ldots,\, n\}$ the corresponding spline space.

We recall hereafter some important properties of the univariate B-Splines. 
Each $\widehat{B}_i^p$ is a piecewise positive polynomial of degree $p$ and has a local support, i.e.,
$\widehat{B}_i^p$ is non-zero only on at most $p+1$ elements and $\operatorname{supp} \widehat{B}_i^p = [\xi_i, \xi_{i+p+1}]$. Consequently on [$\zeta_i,\, \zeta_{i+1}$] at most $p+1$ basis functions have non-zero values. The inter-element continuity is defined by the breakpoint multiplicity. More precisely, we have that the basis functions are $C^{p-m_j}$ at each $\zeta_j \in Z$.

Assuming that $S^p(\Xi) \subset C^0(0,1)$ (i.e., $m_j\leq p, \,\,j=1,\, \ldots,\, E$), and let $\Xi'=\{ \xi_2, \, \ldots, \, \xi_{n+p}\}$, then the derivation operator $\partial_\zeta: S^p(\Xi) \rightarrow S^{p-1}(\Xi')$ is linear and surjective, see~\cite{beirao:14, Schumaker:07}. 

For spline spaces, different refinement strategies are available.  Further knots can be inserted ($h$-refinement), the degree can be elevated ($p$-refinement) and a combination of both is possible ($k$-refinement). We refer to~\cite{hughes:09, piegl:97} for some algorithmic details on the  refinement procedures. In the following, we only consider $h$-refinement, keeping the degree fixed during refinement.

\subsection{Multivariate B-Splines and NURBS}\label{sec:iga_p2}
Multivariate B-Splines are defined based on a tensor product of univariate B-Splines.
Let $d$ be the space dimension. For any direction $\delta=1,\,\ldots,\, d$, we introduce  $p_\delta$ the degree of the univariate B-Splines, $n_{\delta}$ the number of univariate B-Spline functions, $\Xi_\delta$ the univariate  open knot vector and $Z_\delta$ the univariate breakpoint vector. We then define the multivariate knot vector by $ \mathbf{\Xi} =(\Xi_1 \times \Xi_2 \times \ldots \times \Xi_d)$ and the multivariate breakpoint vector by $ \textbf{Z} =(Z_1 \times Z_2 \times \ldots \times Z_d)$.  For simplicity of notation, we are not defining the degree vector but instead we assume in the following that the degree is the same in all parametric directions and denote it by p.

 $ \textbf{Z} $ forms a partition of the parametric domain  $\widehat{\Omega}=(0,1)^d$ and $\widehat{\mathcal{M}}$ defines the set of elements
\begin{align*}
	\widehat{\mathcal{M}}=\{\mathbf{Q}_{\bf j} = \widehat{\tau}_{1, j_1} \times \ldots \times\, \widehat{\tau}_{d, j_d}, ~ \widehat{\tau}_{\delta, j_\delta}=[\zeta_{\delta, j_\delta}, \zeta_{\delta, j_{\delta}+1}] , ~ 1\leq j_\delta \leq E_\delta -1 \}.
\end{align*}

We introduce a set of multi-indices $\mathbf{I} =\{\mathbf{i}=(i_1, \, \ldots, \, i_d): 1\leq i_\delta \leq n_\delta\}$ and define multivariate B-Spline functions for each multi-index $\mathbf{i}$ by tensorization from the univariate B-Spline:
\begin{equation*}
	\widehat{B}_{\mathbf{i}}^p({\boldsymbol \zeta})=\widehat{B}_{i_1}^{p}(\zeta_1) \ldots \widehat{B}_{i_d}^{p}(\zeta_d), \quad \mathbf{i} \in \mathbf{I}.
\end{equation*}
Let us then define the multivariate spline space in the parametric domain by 
\begin{equation*}
	S^p(\mathbf{\Xi})= \otimes_{\delta = 1}^{d} S^p(\Xi_\delta)=\spann\{\widehat{B}_{\mathbf{i}}^p({\boldsymbol \zeta}),  \mathbf{i} \in \mathbf{I}\}.
\end{equation*}

Multivariate NURBS are rational functions of multivariate B-Spline functions. Given a set of positive weights $\{\omega_{\mathbf{i}}, \, \mathbf{i} \in \mathbf{I}\}$, we define the weight function  $\widehat{DW}({\boldsymbol \zeta})=\sum_{ \mathbf{i} \in \mathbf{I}} \omega_{\mathbf{i}} \,\widehat{B}_{\mathbf{i}}^p({\boldsymbol \zeta})$, and then the NURBS functions as 
\begin{equation*}
	\widehat{N}_{\mathbf{i}}^p({\boldsymbol \zeta})=\frac{\omega_{\mathbf{i}} \,\,\widehat{B}_{\mathbf{i}}^p({\boldsymbol \zeta})}{\widehat{DW}({\boldsymbol \zeta})},
\end{equation*}
and in general they are not a tensor product of univariate NURBS functions. Note that B-Splines can be regarded as NURBS with the weights equal to 1, i.e., $\widehat{DW}({\boldsymbol \zeta}) = 1$. Hence whenever there is no ambiguity, we also refer to them as NURBS.

\subsection{Isogeometric parametrization}\label{sec:iga_p3}
NURBS are widely used in the computer aided geometrical design (CAGD), since they are capable to describe various  geometries either exactly (this includes conic sections) or very accurately. Given a set of control points $\mathbf{C}_{\mathbf{i}} \in \mathbb{R}^d$, $\mathbf{i} \in \mathbf{I}$, we can define a parametrization of a NURBS surface ($d=2$) or solid ($d=3$) as a linear combination of NURBS and  control points
\begin{equation*}
	\mathbf{F({\boldsymbol \zeta})}=\displaystyle \sum_{\mathbf{i} \in \mathbf{I}} \mathbf{C_i} \,\widehat{N}_{\mathbf{i}}^p({\boldsymbol \zeta }). 
\end{equation*}
The NURBS geometry is defined as  the image of $\mathbf F$, which is also called geometric mapping, i.e., $\Omega = \mathbf{F}(\widehat \Omega)$.
We define a physical mesh $\mathcal{M}$ as the image  of the parametric mesh $\widehat{\mathcal{M}}$ through $\mathbf{F}$, and denote by $\mathbf{O}$ its elements,
\begin{equation*}
	\mathcal{M}=\{ \mathbf{O} \subset \Omega: \mathbf{O}=\mathbf F(\mathbf{Q}), ~ {\mathbf{Q}} \in \widehat{\mathcal{M}}  \}.
\end{equation*}
Let us assume the following regularity of $\mathbf{F}$.
\begin{assumption}\label{as:regularity_mapping}
The parametrization $\mathbf{F}$ is a bi-Lipschitz homeomorphism. Moreover, $\mathbf{F}_{|\overline{\mathbf{Q}}}$ is in $C^\infty(\overline{\mathbf{Q}})$ for all element of the parametric mesh, and $\mathbf{F}^{-1}_{|\overline{\mathbf{O}}}$ is in $C^\infty(\overline{\mathbf{O}})$ for all element of the physical mesh.
\end{assumption}

Let us define the mesh-size for any parametric element as $h_{\mathbf{Q_j}}=\rm{diam}(\mathbf{Q_j})$ and analogously for any  physical element as $h_{\mathbf{O}_\mathbf{j}}$  and let us note that Assumption~\ref{as:regularity_mapping} ensures that $h_{\mathbf{Q}_\mathbf{j}} \approx h_{\mathbf{O}_\mathbf{j}}$. Thus, no distinction is required and we use the simple notation $h_{\mathbf j}$ for the mesh size. We denote the maximal mesh-size by $h = \max_{\mathbf j} h_{\mathbf j}$.
The mesh size of the underlying univariate partition $\Xi_\delta$, $\delta=1,\,\ldots,\, d$, is denoted by $h_{\delta,j}, \, j=1,\,\ldots,\,E_{\delta}-1$. Let us do the following assumption regarding the mesh uniformity. 
\begin{assumption}\label{as:uniform_mesh}
The partition defined by the breakpoints is globally quasi-uniform, i.e., there exists a constant $\theta$ such that the univariate element size ratio is uniformly bounded: $ h_{\delta, i}/h_{\delta', j} \leq \theta$, with $\delta, \delta'=1,\,\ldots,\, d$ and $i=1,\,\ldots,\,E_{\delta}-1$, $j=1,\,\ldots,\,E_{\delta'}-1$.
\end{assumption}
Note that Assumption~\ref{as:uniform_mesh} excludes the case of anisotropic meshes which are used for, e.g., boundary layers and of graded meshes which are used in case of singularities. However, this assumption is made here only to reduce the technicality of the proofs. We anticipate that a more detailed analysis may show the same results under milder assumptions on the mesh (as the local quasi-uniformity). 

\section{Isogeometric mortar methods }\label{sec:mortar_method}
\label{chapter3}
In this section, we first state the problem and define the geometry setting, then the functional framework and finally the approximation spaces and their required properties to be optimal.

Let $\Omega \subset \R^d$, $d = 2,3$, be a bounded domain, $\alpha, \beta \in L^{\infty}(\Omega)$, $\alpha > \alpha_0 > 0$ and $\beta \geq 0$. We consider the following  second order elliptic boundary value problem with homogeneous Dirichlet conditions
\begin{subequations}\label{eq:strong_formulation}
\begin{align} 
		-\divergence (\alpha\nabla u) + \beta u &= f \quad \text{ in } \Omega,\\
		u &= 0 \quad \text{ on } \partial\Omega_D = \partial\Omega.
\end{align}
\end{subequations}
We assume $\alpha, \beta$ to be sufficiently smooth, but allow jumps in special locations, which are specified later. 
\subsection{Description of the computational domain}
Let a decomposition of the domain $\Omega$ into $K$ non-overlapping domains $\Omega_k$ be given:
\[
\overline{\Omega} = \bigcup_{k=1}^K \overline{\Omega}_k, \text{ and } \Omega_i \cap \Omega_j = \emptyset, i \neq j.
\]
For $1\leq k,\, l \leq K $, $k\neq l$, we define the interface as the interior of the intersection of the boundaries, i.e., $\overline{\gamma}_{kl} = \partial {\Omega}_k \cap \partial {\Omega}_{l}$, where ${\gamma}_{kl}$ is open. Let the non-empty interfaces be enumerated by $\gamma_l$, $l = 1,\,\ldots,\, L$,  and define the skeleton $\Gamma = \bigcup_{l=1}^L \gamma_l$ as the union of all interfaces. For each interface, one of the adjacent subdomains is chosen as the master side and one as the slave side, this choice is arbitrary but fixed. 
We denote the index of the former by $m(l)$, the index of the latter one by ${s(l)}$, and thus $\overline{\gamma}_l=\partial \Omega_{m(l)}\cap \partial \Omega_{s(l)}$. Note that one subdomain can at the same time be classified as a master domain for one interface and as a slave domain for another interface, see Figure~\ref{mortar:mortar_setting}. On the interface $\gamma_l$, we define the outward normal $\bf  n_l$ of the master side $\partial \Omega_{m(l)}$ and denote  by $\displaystyle{\frac{\partial u}{\partial\bf n_l}}$ the normal derivative on $\gamma_l$ from the master side.

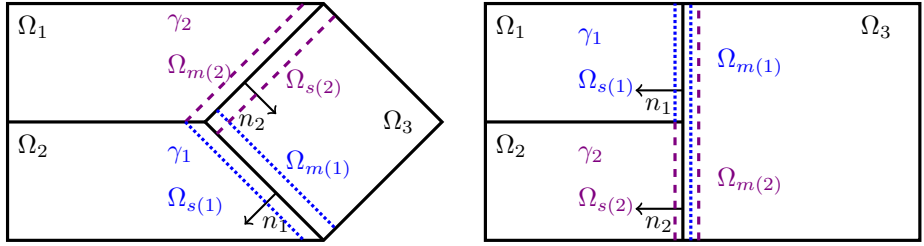
\begin{figure}[htbp!] 
	\begin{tikzpicture}[scale=1.05]

\tikzstyle{style1} = [black, very thick]
\tikzstyle{style2} = [violet, very thick, dashed]
\tikzstyle{style3} = [blue, very thick, dash pattern=on \pgflinewidth off 1pt]

\def\a{5.5};
\def\b{2.5};
\def\c{4};
\def\d{2.5};
\def\h{3};
\def\hh{1.5};
\def\f{0.1};
\def\s{0.1};
\def\ss{0.2};

\coordinate (0) at (0, 0);
\coordinate (1) at (\c, 0);
\coordinate (2) at (\a, \hh);
\coordinate (3) at (\c, \h);
\coordinate (4) at (0, \h);
\coordinate (5) at (0, \hh);
\coordinate (6) at (\b, \hh);

\coordinate (7) at ($(6)!0.1!(5)$);
\coordinate (8) at ($(6)!0.1!(3)$);
\coordinate (9) at ($(6)!0.1!(1)$);

\coordinate (10) at ($(7) + (3) - (6)$);
\coordinate (11) at ($(7) + (1) - (6)$);
\coordinate (12) at ($(8) + (1) - (6)$);
\coordinate (13) at ($(9) + (3) - (6)$);

\def\ln{0.4};
\coordinate(14)at (\c-0.6,0.6);
\coordinate(15)at(\c-0.6-\ln,0.6-\ln);
\draw[thick,->]  (14) -- (15);
\coordinate(18)at(\c-\ln-0.2,0.2);
\node[style1] at ($(18)$) {$n_1$};

\coordinate(16)at (\b+0.5,\hh+0.5);
\coordinate(17)at(\b+0.5+\ln,\hh+0.5-\ln);
\draw[thick,->]  (16) -- (17);
\coordinate(19)at(\b+0.5+0.1,\hh+0.5-\ln-0.1);
\node[style1] at ($(19)$) {$n_2$};

\draw[style1] (0) -- (1) -- (2) -- (3) -- (4) -- cycle;
\draw[style1]  (5) -- (6) -- (3);
\draw[style1]  (6) -- (1);

\draw[style2] (7) -- (10);
\draw[style3] (7) -- (11);
\draw[style3] (8) -- (12);
\draw[style2] (9) -- (13);

\node[right, style1] at ($(0)!0.9!($(0)+(0,\h)$)$) {$\Omega_1$};
\node[right, style1] at ($(0)!0.4!($(0)+(0,\h)$)$) {$\Omega_2$};
\node[style1] at ($(5)!0.9!(2)$) {$\Omega_3$};

\node[style2, text width=1cm] at ($(6) + (0, 0.31*\h)$) {$\gamma_2$\\$\Omega_{m(2)}$};
\node[style3, text width=1cm]  at ($(6) - (0, 0.25*\h)$) {$\gamma_1$\\$\Omega_{s(1)}$};

\node[style3, text width=1cm] at ($(1)+(0, 0.32*\h)$) {$\Omega_{m(1)}$};
\node[style2, text width=1cm] at ($(1)+(0, 0.66*\h)$) {$\Omega_{s(2)}$};

\begin{scope}[xshift=1.1*\a cm]
\coordinate (0) at (0, 0);
\coordinate (1) at (\a, \h);
\coordinate (2) at (\d, 0);
\coordinate (3) at (\d, \h);
\coordinate (4) at (0, \hh);
\coordinate (5) at (\d, \hh);
\coordinate (6) at ($(2) - (\s, 0)$);
\coordinate (7) at ($(2) + (\s, 0)$);
\coordinate (8) at ($(2) + (\ss, 0)$);
\coordinate (9) at ($(5) - (\s, 0)$);
\coordinate (10) at ($(5) - (\s, 0)$);
\coordinate (11) at ($(3) - (\s, 0)$);
\coordinate (12) at ($(3) + (\s, 0)$);
\coordinate (13) at ($(3) + (\ss, 0)$);
\coordinate (14) at ($(0)!0.3!($(0)+(\a, 0)$)$);
\coordinate (15) at ($(0)!0.62!($(0)+(\a, 0)$)$);

\def\lnb{0.6};
\coordinate(16)at (\d,\hh+0.4);
\coordinate(17)at(\d-\lnb,\hh+0.4);
\coordinate(18)at(\d-\lnb/2,\hh+0.2);
\draw[thick,->]  (16) -- (17);
\node[style1] at ($(18)$) {$n_1$};

\coordinate(19)at (\d,0.4);
\coordinate(20)at(\d-\lnb,0.4);
\coordinate(21)at(\d-\lnb/2,0.2);
\draw[thick,->]  (19) -- (20);
\node[style1] at ($(21)$) {$n_2$};

\draw[style1]  (0) rectangle (1);
\draw[style1]  (2) -- (3);
\draw[style1]  (4) -- (5);
\draw[style2] (6) -- (9);
\draw[style3] (10) -- (11);
\draw[style3] (7) -- (12);
\draw[style2] (8) -- (13);

\node[right, style1] at ($(0)!0.9!($(0)+(0,\h)$)$) {$\Omega_1$};
\node[right, style1] at ($(0)!0.4!($(0)+(0,\h)$)$) {$\Omega_2$};
\node[style1] at ($(0)!0.9!(1)$) {$\Omega_3$};

\node[style3, text width=1cm, ] at ($(14)+(0, 0.75*\h)$) {$\gamma_1$\\$\Omega_{s(1)}$};
\node[style2, text width=1cm, ] at ($(14)+(0, 0.25*\h)$) {$\gamma_2$\\$\Omega_{s(2)}$};

\node[style3, text width=1cm, ] at ($(15)+(0, 0.75*\h)$) {$\Omega_{m(1)}$};
\node[style2, text width=1cm, ] at ($(15)+(0, 0.25*\h)$) {$\Omega_{m(2)}$};

\end{scope}

\end{tikzpicture}
	\caption{Geometrical conforming case (left) and  slave conforming case (right).}
	\label{mortar:mortar_setting} 
\end{figure}

Each subdomain $\Omega_k$ is given as the image of the parametric space $\widehat \Omega = (0,1)^d$ by one single NURBS parametrization $\mathbf{F}_k: \widehat{\Omega} \rightarrow \Omega_k$, see Section~\ref{sec:iga_p3}, which satisfies the Assumption~\ref{as:regularity_mapping}. The $h$-refinement procedure, see Sections~\ref{sec:iga_p2} and \ref{sec:iga_p3}, yields a family of meshes denoted $\mathcal{M}_{k,h}$, each mesh being a refinement of the initial one, where we require Assumption~\ref{as:uniform_mesh}. Under  these assumptions, the family of meshes is shape regular.

We furthermore assume that for each interface, the pull-back with respect to the slave domain is a whole face of the unit $d$-cube in the parametric space. Under these assumptions, we are not necessarily in a geometrically conforming situation, but we call it a \emph{slave conforming} situation, see the right setting in Figure~\ref{mortar:mortar_setting}.
If we also assume that the pull-back with respect to the master domain is a whole face of the unit $d$-cube, we are in a fully geometrically conforming situation, see the left picture of Figure~\ref{mortar:mortar_setting}.

\subsection{The variational problem}
In the following, we recall main functional analysis properties to introduce our abstract framework and then set the variational problem. 

We use standard Lebesgue and Sobolev spaces on a bounded Lipschitz domain $D\subset \mathbb{R}^{d-1}$ or $D\subset\mathbb{R}^{d}$.  $L^2(D)$ denotes the Lebesgue space of square integrable functions, endowed with the norm
$\|f\|_{L^2(D)}=(\int_{D} \left| f \right|^2 dx)^{1/2}$. 
For $l \in \mathbb{N}$,  $H^l(D)$ denotes the Sobolev space of functions $f \in L^2(D)$ such that their weak derivatives up to the order $l$ are also in $L^2(D)$. 
 For fractional indices $s > 0$, $H^s(D)$ denotes the fractional Sobolev spaces as defined in~\cite{grisvard:11}. Let us mention that $H^{1/2}(\partial D)$ is the trace space of $H^1(D)$.
 
 The Sobolev space of order one with vanishing trace is $H_0^1(D) = \{v\in H^1(D), tr(v) = 0\}$. 
Working on subsets of the boundary  $\gamma \subset \partial D$, special care has to be taken about the values on the boundary  of $\gamma$. We define by $H^{1/2}_{00}(\gamma) \subset H^{1/2}(\gamma)$  the space of all functions that can be trivially extended on $\partial D \setminus \gamma$ by zero to an element of $H^{1/2}(\partial D)$. The dual space of $H^{1/2}_{00}(\gamma)$ is denoted $H^{-1/2}(\gamma)$. Note that on closed surfaces, i.e., $\gamma = \partial D$, it holds $H^{1/2}(\gamma) =  H^{1/2}_{00}(\gamma)$. Furthermore, in the following we omit the trace operator, whenever there is no ambiguity.

 For each $\Omega_k$, we introduce the space $H^1_*(\Omega_k)=\{ v_k \in H^1(\Omega_k),  v_{k |_{\partial \Omega \cap \partial \Omega_k}}=0\}$. And in order to set a global functional framework on $\Omega$, we consider the broken Sobolev spaces $V= \Pi_{k=1}^K H^1_*(\Omega_k)$, endowed with the broken norm $ \| v \|_{V}^2 = \sum_{k=1}^K \| v \|_{H^1(\Omega_k)}^2$, and $M= \Pi_{l=1}^L H^{-1/2}(\gamma_l)$.

The standard weak formulation of~(\ref{eq:strong_formulation}) reads as follows: Find $u \in H_0^1(\Omega)$ such that
\begin{equation}\label{eq:standard_weak_formulation}
	\int_\Omega \alpha \nabla u \cdot \nabla v + \beta \, u\, v \, \mathrm{d}x = \int_\Omega f \,v \, \mathrm{d}x, \quad   v \in H_0^1(\Omega).
\end{equation}
It is well-known that under the assumptions on $\alpha$ and $\beta$, the variational problem~(\ref{eq:standard_weak_formulation}) is uniquely solvable.

From now on, we assume that jumps of $\alpha$ and $\beta$ are solely located at the skeleton, and we define the  linear and bilinear forms $a\colon V \times V \rightarrow \mathbb{R}$ and  
 $f\colon V \rightarrow \mathbb{R}$, such that  
 \[a(u,v) = \sum_{k=1}^K \int_{\Omega_k} 
\alpha \nabla u \cdot  \nabla v + c\, u\, v ~ \mathrm{d}x, \quad
f(v) =\sum_{k=1}^K \int_{\Omega_k} f v ~ \mathrm{d}x.
\]

\subsection{Isogeometric mortar discretization}
In the following, we set our approximations spaces.
Let us introduce $V_{k,h}$ the approximation space on $\Omega_k$, by $V_{k,h}=\{v_k=\widehat{v}_k \circ \mathbf{F}_k^{-1}, \widehat{v}_k \in N^{p_k}(\mathbf{\Xi}_k) \}$ defined on the knot vector $\mathbf{\Xi}_k$ of degree $p_k$. Denote $h_k$ the mesh size of $V_{k,h}$ but note that we use the maximal mesh size $h = \max_k h_k$ as the mesh parameter. We recall that NURBS spaces are known to have optimal approximation properties as stated in the following lemma, see, e.g,~\cite{bazilevs:06,Schumaker:07,sangalli:12}.
\begin{lemma}\label{lem:Spline_approx}
Given a quasi-uniform mesh and let $r, s$ be such that $0\leq r \leq s \leq  p_k+1$.
Then, there exists a constant C depending only on $ p_k$, $\theta_k$, $\mathbf{F}_k$ and $\widehat{DW}_k $, such that 
for any $v  \in H^s(\Omega_k)$  there exists an approximation $v_h \in   N^{p_k}(\mathbf{\Xi}_k)$,  such that 
		\begin{equation*}
		 \|v -v_h\|_{H^r{(\Omega_k)}}\leq C h^{s-r} \| v \|_{H^s(\Omega_k)} .
	\end{equation*}
\end{lemma}
 
 On $\Omega$, we define the product space $V_h = \Pi_{k=1}^K V_{k,h} \subset V$, which forms a $H^1(\Omega)$-non-conforming space which is discontinuous over the interfaces.

The mortar method is based on a weak enforcement of continuity across the interfaces $\gamma_l$ in broken Sobolev spaces. Let a space of discrete Lagrange multipliers $M_{l,h}\subset L^2(\gamma_l)$ on each interface $\gamma_l$ built on the slave mesh be given. On the skeleton $\Gamma$, we define the discrete product Lagrange multiplier space $M_h$ as $M_h = \Pi_{l=1}^L M_{l,h}$. Choices of different spaces will be discussed in the next section.
Furthermore, we define the discrete trace space with additional zero boundary conditions by $W_{l,h} = \{ v_{|\gamma_l}, v \in V_{{s(l)},h}\} \cap H^1_{0}(\gamma_l)$.

One possibility for a mortar method is to specify the discrete weak formulation as a saddle point problem:
Find $ (u_h, \lambda_h) \in V_h \times M_h,$  such that 
\begin{subequations}\label{eq:discrete_spp}
\begin{align}
		a(u_h, v_h)+ b(v_h, \lambda_h) &= f(v_h), \quad   v_h \in V_h,\\
		b(u_h, \mu_h) &= 0, \quad  \mu_h \in M_h,
\end{align}
\end{subequations}
where $b(v,\mu) = \sum_{l=1}^L \int_{\gamma_l} \mu [v]_l ~\mathrm{d}\sigma$ and  $[\cdot]_l$ denotes the jump from the master to the slave side over $\gamma_l$.

We note, that the Lagrange multiplier $\lambda_h$ gives an approximation of the normal flux across the skeleton.

It is well known from the theory of mixed and mortar methods, that the following  abstract requirements guarantee the method to be well-posed and  of optimal order, see~\cite{ben_belgacem:99, brezzi:13}. In the following, we will denote by $0 < C < \infty$ a generic constant that is independent of the mesh sizes but possibly depends on $p_k$.

The first assumption is a uniform inf-sup stability for the discrete trace spaces. Although the primal variable of the saddle point problem  is in a broken $H^1$ space, the inf-sup stability can be formulated as a $L^2$ stability over each interface. This implies the $H^{1/2}_{00}-H^{-1/2}$ stability, which can be used in the geometrically conforming situation for $d=2$ and in weighted $L^2$ norms, which can be used for the other cases, see~\cite{braess:99}.
\begin{assumption}\label{as:mortar_intro:inf_sup} For $l=1,\ldots,L$ and any $\mu_l \in M_{l,h}$ it holds
\begin{align*} 
	\sup_{w_l \in W_{l,h}} \frac{ \int_{\gamma_l} w_l \, \mu_l ~\mathrm{d}\sigma}{\|w_l \|_{L^2(\gamma_l)}} \geq C \| \mu_l \|_{L^2(\gamma_l)}.
\end{align*}
\end{assumption}

The second assumption is the approximation order of the dual space. Since for the dual space weaker norms are used, the approximation order of $M_{l,h}$ with respect to the $L^2$ norm can be smaller than the one of $W_{l,h}$.
\begin{assumption}\label{as:mortar_intro:approximation_order} For $l=1,\,\ldots,\,L$  there exists a fixed  $\eta(l)$, such that for any $\lambda \in H^{\eta(l)}(\gamma_l)$ it holds
	\begin{align*}
		\inf_{\mu_l \in M_{l,h}} \| \lambda - \mu_l \|_{L^2(\gamma_l)} \leq C h^{\eta(l)} \|\lambda\|_{H^{\eta(l)}(\gamma_l)}.  
	\end{align*}
\end{assumption}

We now give the following a-priori estimates in the broken $V$ and $M$ norms, which can be shown by standard techniques, see~\cite{bernardi:93, ben_belgacem:97}. 
\begin{theorem} \label{thm:convergence_rates}
Given Assumptions~\ref{as:mortar_intro:inf_sup} and~\ref{as:mortar_intro:approximation_order}, the following convergence is given for the primal solution of~(\ref{eq:discrete_spp}). For $u\in H^{\sigma+1}(\Omega)$, $1/2<\sigma \leq \min_{k,l} (p_k,\eta(l)+1/2)$ it holds
\[
	\frac{1}{h^2} \|u-u_h\|_{L^2(\Omega)}^2 + \| u - u_h \|_{V}^2  \leq C \sum_{k=1}^K h^{2\sigma}_k \|u \|^2_{H^{\sigma+1}(\Omega_k)}.
\]
We can also give an estimate for the dual solution which approximates the normal flux:
\[
\sum_{l=1}^L \| \alpha \frac{\partial u}{\partial \mathbf{n}_l} - \lambda_h \|_{H^{-1/2}(\gamma_l)}^2   \leq C \sum_{k=1}^K h^{2\sigma}_k \|u \|_{H^{\sigma+1}(\Omega_k)}^2.
\]
\end{theorem}
In the geometrically non-conforming case, as well as for $d=3$, the ratio of the mesh sizes on the master and the slave side enters in the a-priori estimate, see~\cite{lamichhane:06}. But due to our global quasi-uniformity assumption, see Assumption~\ref{as:uniform_mesh}, this ratio does not play a role. 

We note that if $\eta(l)= p_{s(l)}-1/2$ can be chosen, optimality of the mortar method holds. Moreover, the dual estimate could still be improved under additional regularity assumptions, see~\cite{wohlmuth:12}.

\section{Possible choices of Lagrange multiplier spaces}
\label{chapter4}
For a given interface $\gamma_l$, we aim at providing multiplier spaces that satisfy the inf-sup stability of Assumption~\ref{as:mortar_intro:inf_sup}. In our setting, i.e., a geometrically slave conforming situation, see Figure~\ref{mortar:mortar_setting}, $\gamma_l$ is a whole face of $\Omega_{s(l)}$, which is defined as $\mathbf{F}_{s(l)}(\widehat{\Omega})$ and without loss of generality we suppose that $\gamma_l=\mathbf{F}_{s(l)}(\widehat{\gamma}\times\{0\})$, $\widehat{\gamma}=(0,1)^{d-1}$. As  we consider each interface $\gamma_l$ separately, to shorten the notations we will omit the index $l$ in the following.  

Given a Lagrange multiplier space on the parametric space $\widehat M$, we set the Lagrange multiplier space $M=\{\mu = \widehat{\mu} \circ \mathbf{F}_s^{-1}, \widehat{\mu} \in \widehat M \}$. 
By change of variable, the integral in Assumption~\ref{as:mortar_intro:inf_sup} can be transformed into a weighted integral on the parametric space. Denoting $\widehat w=  (w\circ\mathbf{F}_s)\,\widehat{DW} \in S^p(\widehat \gamma)$ for $w\in W$ and $\widehat \mu = \mu \circ \mathbf{F}_s\in \widehat M$ for $\mu \in M$, the integral becomes
\begin{align}\label{change_of_variable}
\int_{\gamma} w \, \mu \, \mathrm{d}\sigma &=\int_{\widehat{\gamma}} (w\circ \mathbf{F}_{s}) \, (\mu \circ \mathbf{F}_s) \det(\nabla_{\hat \gamma}\mathbf{F}_s) \, \mathrm{d}x\notag\\&= \int_{\widehat{\gamma}} \widehat{w} \, \widehat{\mu} ~ (\widehat{DW})^{-1}\, \det(\nabla_{\hat \gamma}\mathbf{F}_s)\, \mathrm{d}x,
\end{align}
where $\nabla_{\hat \gamma}$ denotes the surface gradient on $\widehat \gamma$.
Due to the Assumption~\ref{as:regularity_mapping} and the uniform positivity of NURBS weights, we can firstly concentrate on the following problem. 
Given $\widehat{\gamma}=(0,1)^{d-1}$, a degree $p$ and knot vectors $\Xi_\delta$ with $\delta=1,\dots,d-1$, we denote by $S^p(\widehat{\gamma})$ the corresponding spline space and $S^p_0(\widehat{\gamma})=S^p(\widehat{\gamma})\cap H^1_0(\widehat{\gamma})$, and study the following inf-sup stability
\begin{equation}\label{eq:para_inf_sup}
\sup_{\widehat{w} \in S^p_0(\hat \gamma)} \frac{ \int_{\widehat{\gamma}} \widehat{w} \, \widehat{\mu} ~\mathrm{d}x}{\|\widehat{w}  \|_{L^2(\hat{\gamma})}} \geq C \| \widehat{\mu}  \|_{L^2(\hat{\gamma})}.
\end{equation}
for any $\widehat \mu \in \widehat M$ for three choices of Lagrange multipliers space $\widehat{M}$. Then, in the case (\ref{eq:para_inf_sup}) is satisfied, we show that the desired inf-sup stability, i.e., Assumption~\ref{as:mortar_intro:inf_sup}, is satisfied. 

In the following remark we briefly discuss the construction of a dual biorthogonal basis with functions having the same support as the primal basis functions. Due to possible difficulties concerning the approximation order, this approach is not considered in the following of this article.
\begin{remark}
By a local orthogonalization procedure, a biorthogonal Lagrange multiplier basis $\{\psi_i\}_{i=1}^{n}$ fulfilling $\operatorname{supp} \widehat{B}_i^p = \operatorname{supp} \psi_i$ and
\[
\int_{\widehat{\gamma}} \widehat{B}_i^p({x}) \psi_j({x}) \mathrm{d}x = c_i \delta_{ij},
\]
for a suitable scaling $c_i$, can be constructed. In Figure~\ref{fig:dual_basis} a primal quadratic basis function and its corresponding biorthogonal basis are depicted.
 
This yields computational advantages, since the coupling degree of freedom can be locally eliminated. However, in the higher order finite element case, it was shown that the construction of a biorthogonal basis with the desired approximation results is not a trivial task, see~\cite{lamichhane:07}.
\begin{figure} 
\begin{center}
	\includegraphics[width=0.8\textwidth]{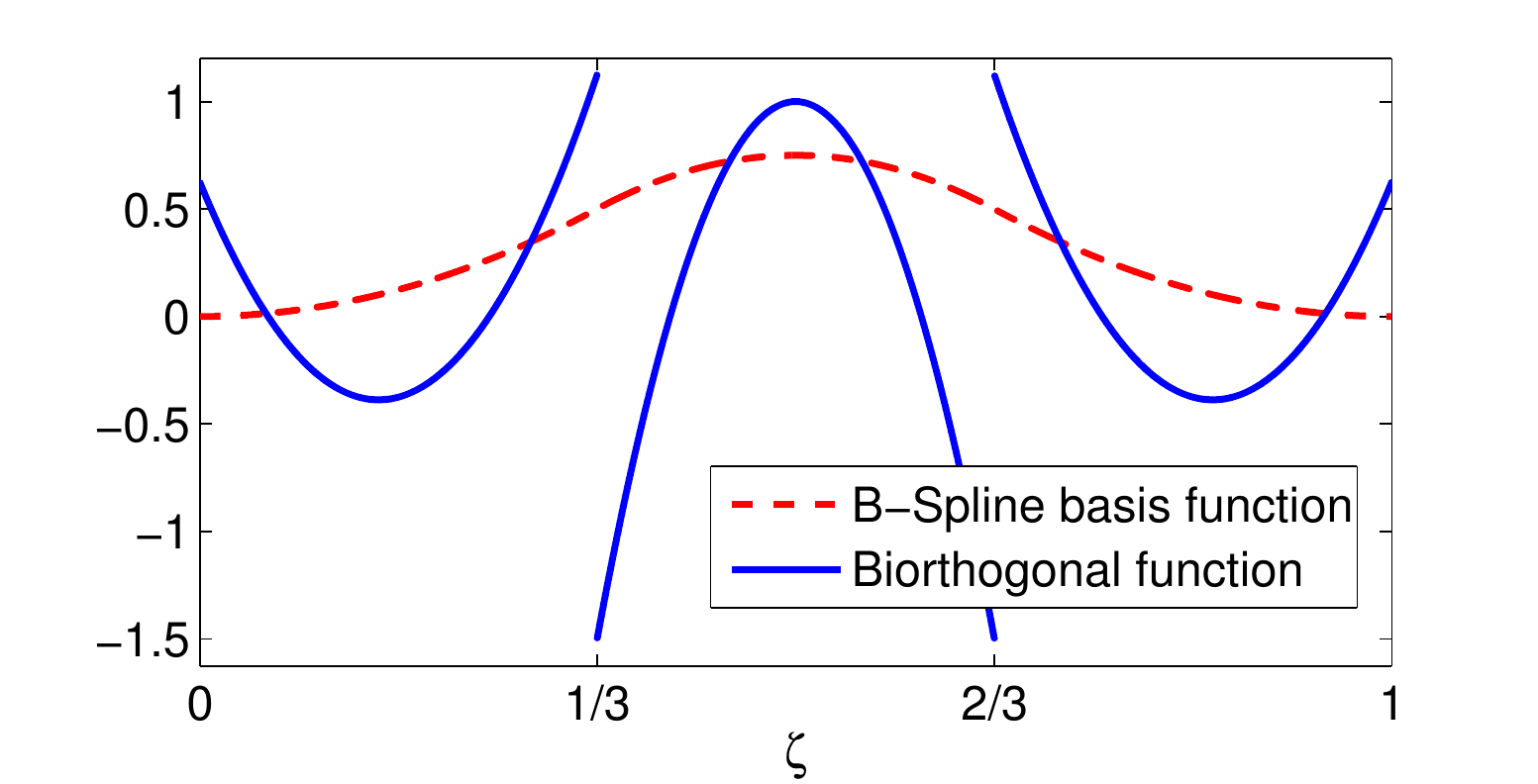}
\end{center}
\caption{A quadratic basis function and its corresponding dual basis on a uniform mesh. The quadratic function corresponds to the local knots $0, 1/3, 2/3, 1$.} 
\label{fig:dual_basis}
\end{figure}
\end{remark}

In the following, we give the details of this inf-sup study, and then we conclude the underlying approximation properties of these isogeometric mortar methods.

\subsection{Choice 1: unstable pairing  $p/p-1$}\label{sec:p_pm1}
Theorem~\ref{thm:convergence_rates} states that order $p=\min_k p_k$ a priori bounds can only be obtained if $\eta(l)$ can be set equal to $p-1/2$. This observation motivates our choice to use a spline space of order $p-1$ as dual space. Then $\eta(l)$ in Assumption~\ref{as:mortar_intro:approximation_order} can be set to $p$ and provided that the uniform inf-sup stability, Assumption~\ref{as:mortar_intro:inf_sup}, can be proved, a convergence rate equal to $p$ might be reached.

Denote by $\widehat{M}^{1} = \spann_{i=1,\,\ldots,\, n^{(1)}}\,\{\widehat{B}_{i}^{p-1} \}$  the spline space of order $p-1$ built on the knot vector(s) $\Xi_\delta'$ with $\delta=1,\dots,d-1$ obtained from the restriction of $\mathbf{\Xi}$ to the corresponding direction(s) removing in the underlying univariate knot vector the first and the last knots. The superscript $1$ refers to the degree difference between the primal and the dual space.

As we will see this choice unfortunately lacks the uniform inf-sup condition~(\ref{eq:para_inf_sup}) and thus also Assumption~\ref{as:mortar_intro:inf_sup}. 
Indeed, a checkerboard mode which yields an $h$-dependent inf-sup constant can be constructed.

\begin{figure} 
\begin{center}
	\includegraphics[width=\textwidth]{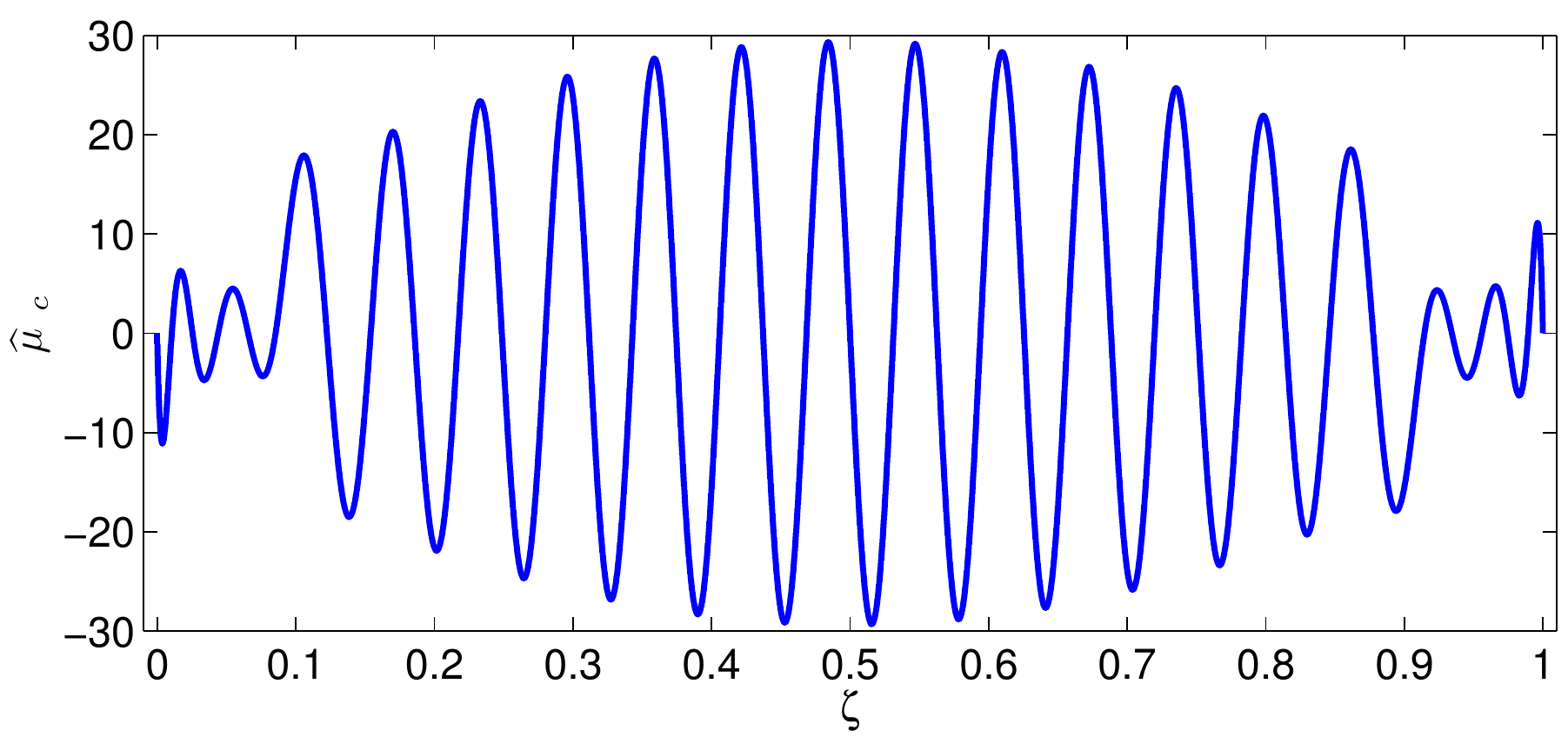}
\end{center}
\caption{Checkerboard mode for $d=2$ and $p=6$.} 
\label{fig:checkerboard_plot}
\end{figure}
Let us consider B-Splines on a uniform knot vector $\Xi = \{0,\,\ldots,\, 0,\, h,\, 2h,\, \ldots,\,\allowbreak 1,\, \ldots,\, \allowbreak 1\}$ for $h = 2^{-j}$, where $j$ is the number of uniform refinements.
Let us now construct a multiplier $\widehat{\mu}_c \in \widehat{M}^{1}$, which yields an $h$-dependent inf-sup constant.
The choice \[ \widehat{\mu}_c = \sum_{i=1}^{n^{(1)}} \widehat{\mu}_i \widehat{B}_i^{p-1}, \quad \widehat{\mu}_i = (-1)^i (i-1) (n^{(1)}-i),\] is shown in Figure~\ref{fig:checkerboard_plot}.
 For the bivariate case, a tensor product using $\widehat{\mu}_c$ in each direction is chosen. The numerical stability constants were computed by a direct evaluation of the supremum
 \[
 \sup_{\widehat w \in S^p(\hat\gamma)} \frac{\int_{\widehat \gamma} \widehat w \, \widehat{\mu}_c \mathrm{d}x}{\|\widehat w\|_{L^2(\hat \gamma)}},
 \]
 and dividing the result by $\| \widehat{\mu}_c\|_{L^2(\hat \gamma)}$.
The results are  shown in Figure~\ref{fig:checkerboard_constant} for $d=2$ and $d=3$, where an $h$-dependency of order $\mathcal{O}(h^{d-1})$ can be observed.  Note that on the same mesh, the stability constant is larger for higher degrees, but the asymptotic rate of the $h$-dependency is the same.

\begin{figure}
\begin{center}
\includegraphics[width=0.49\textwidth]{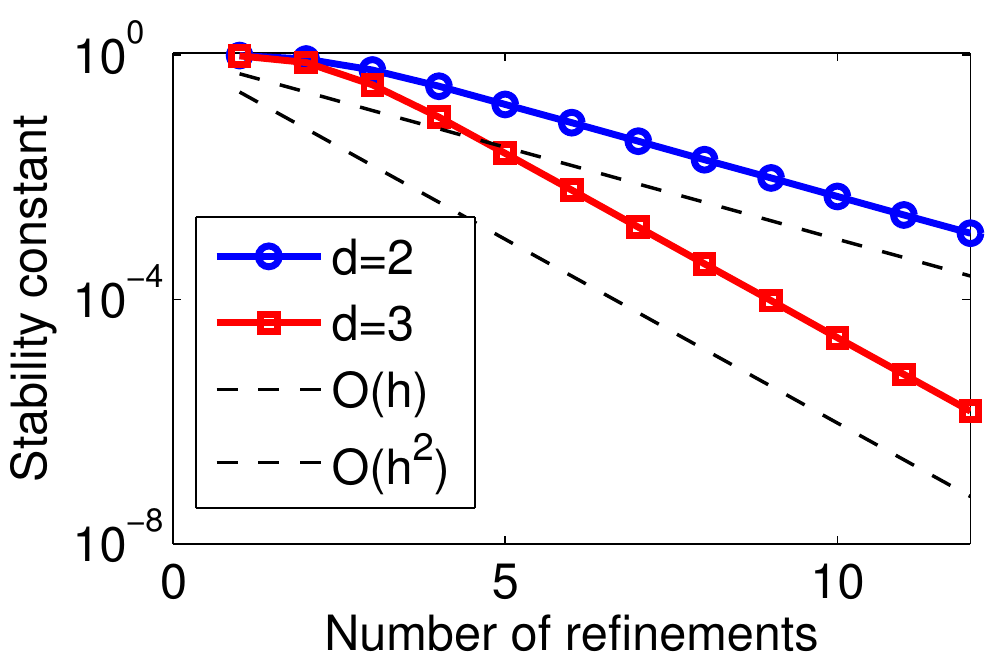}
\includegraphics[width=0.49\textwidth]{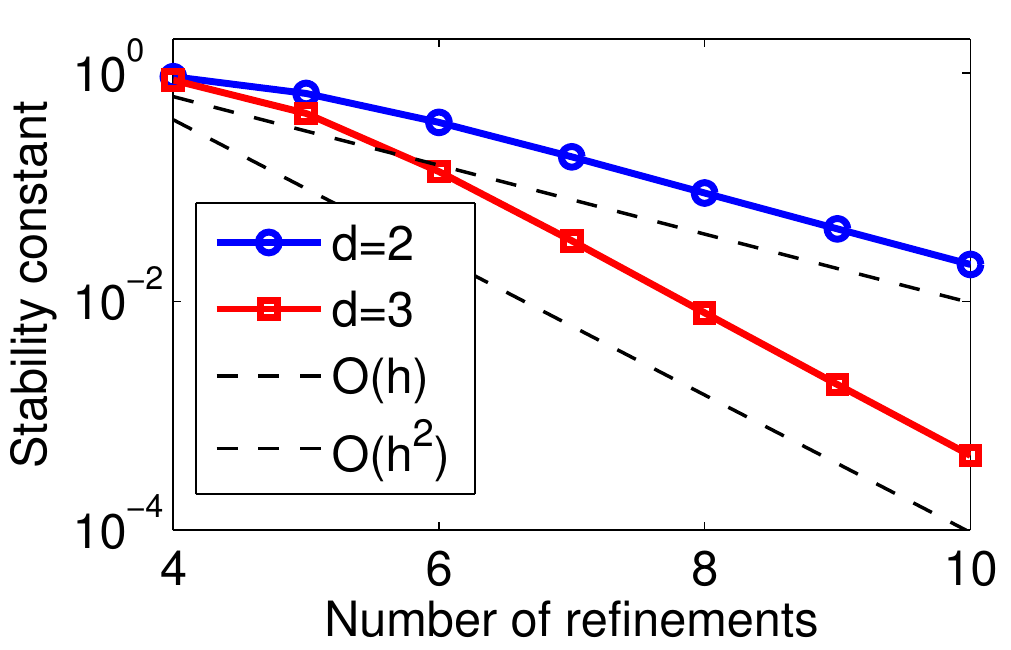}
\end{center}
\caption{$h$-dependency of the $L^2$ inf-sup constant for dimension $d=2$ and $d=3$. Left: $p= 2$, Right: $p=10$. }
\label{fig:checkerboard_constant}
\end{figure}

\begin{remark}
 Numerical experiments show, that the inf-sup constant can be recovered by the use of a staggered grid, which is similar to the behavior known from the finite element method. Another possibility is to use a coarse dual mesh for the Lagrange multipliers.
 \end{remark}

\subsection{Choice 2: stable pairing $p/p-2$}\label{sec:p_pm2}
Having an unstable pairing means roughly speaking that the chosen Lagrange multiplier space is too rich. An easy way to overcome this is by using a smaller space which motivates our second choice.
If the spline space $S^p(\widehat{\gamma})$ is at least $C^1$, then it is also possible to construct a spline space of degree $p-2$ on the knot vector(s) $\Xi_\delta''$ with $\delta=1,\dots,d-1$ obtained from the restriction of $\mathbf{\Xi}$ to the corresponding direction(s) removing in the underlying univariate knot vector the first and the last two knots. We denote this space by $\widehat{M}^{2} = \spann_{i=1,\,\ldots,\, n^{(2)}}\{\widehat{B}_{i}^{p-2} \}$, where the superscript $2$ refers to the degree difference between the primal and the dual space. Clearly, this choice will never provide an optimal convergence rate because even if the stability is true, in general the theoretical convergence rate will not exceed $p-1/2$. In what follows, we prove that $\widehat{M}^2$ verifies the inf-sup stability (\ref{eq:para_inf_sup}). 

The proof is based on an identification of both spaces using derivatives and integrals as well as on an auxiliary stability result for the degree $p-1$. Let us first introduce some preliminary notation.

To shorten our notation, we denote by $S^q$ with $q= p-2, \,p-1,$ and $p$ the spline spaces of degree $q$ constructed on $\mathbf{\Xi''}, \mathbf{\Xi'}$ and $\mathbf{\Xi}$, respectively.
Furthermore let us define the spline space with zero mean value $S^{p-1}_{\rm zmv} = \{\widehat{s}\in S^{p-1}: \int_0^1 \widehat{s} ~\mathrm{d}x = 0\}$ for $d=2$ and 
\[
S^{p-1}_{\rm zmv} =\left\{\widehat{s}\in S^{p-1}: \int_0^1 \widehat{s}(x,\bar y) \mathrm{d}x = 0 = \int_0^1 \widehat{s}(\bar x, y) \mathrm{d}y, \quad \bar x, \bar y \in [0,1] \right\},
\]
for $d=3$.
While for $d=2$, we consider a single derivative $D=\partial_x$ as the derivative operator, for $d=3$, due to the tensor product structure, we also consider the mixed derivative $D=\partial_{xy}$. Associated with the mixed derivative, we consider the tensor product Sobolev space
\[
H^{1,1}(\widehat \gamma) = H^1(0,1)\otimes H^1(0,1)= \{\widehat{v} \in L^2(\widehat \gamma): \partial_{x}^i \partial_y^j \widehat{v} \in L^2(\widehat \gamma), i,j \in \{ 0,1\} \},
\]
endowed with the norm
$
\| \widehat{v} \|_{H^{1,1}(\hat \gamma)}^2 = 
\| \widehat{v} \|_{H^{1}(\hat \gamma)}^2 + 
\|\partial_{xy} \widehat{v} \|_{L^2(\hat \gamma)}^2
$.
To simplify the notation, we will denote in the following $Z = H^1(\widehat \gamma)$ for $d=2$ and $Z=H^{1,1}(\widehat \gamma)$ for $d=3$. Let $Z'$ denotes the dual space of $Z$.

The following lemma shows that the given derivative operator maps bijectively the spaces $S^p_0$, $S^{p-1}_{\rm zmv}$ and $S^{p-2}$ into each other. 
\begin{lemma}\label{lem:spline_derivatives}
The operators $D: S^p_0\rightarrow S^{p-1}_{{\rm zmv}}$ and $D: S^{p-1}_{{\rm zmv}}\rightarrow S^{p-2}$ are bijections. Moreover for any $v\in Z \cap H_0^1(\widehat \gamma)$, it holds $\|v\|_{L^2(\hat \gamma)} \leq C \| D v\|_{Z'}$. 
\end{lemma}
\begin{proof}
Based on~\cite[Theorem 5.9]{Schumaker:07} the derivative of a spline of degree $p$ is a spline of degree $p-1$, see also Section~\ref{sec:iga_p1}. The injectivity follows from the additional constraints of the spline space. To show the surjectivity, we construct an element of the pre-image space. The coercivity of the derivative can be seen by an explicit computation using  partial integration.

Case $d=2$.
Given $\widehat{s}^{p-2} \in S^{p-2}$, we define $\widehat{s}^{p-1}(x) = \int_0^x \widehat{s}^{p-2}(\xi)\mathrm{d}\xi - m$, where $m\in \mathbb{R}$ is chosen such that $\int_0^1 \widehat{s}^{p-1} \mathrm{d}x =  0$.  Obviously $\widehat{s}^{p-1} \in S^{p-1}_{\rm zmv}$. For any $\widehat{s}^{p-1} \in S^{p-1}_{\rm zmv}$ we may define $\widehat{s}^p(x) = \int_0^x \widehat{s}^{p-1}(\xi)\mathrm{d}\xi $ and it holds $\widehat{s}^p \in S^p_0$.

To show the coercivity, consider any $\widehat{w}\in L^2(0,1)$. We can find $\widehat{z}\in H^1_{\rm zmv}(0,1) = \{\widehat{z}\in Z: \int_0^1 \widehat{z}\, \mathrm{d}x = 0\}$, such that $\partial_x \widehat{z} = \widehat{w}$ and then 
\begin{align*}
\|\widehat{v}\|_{L^2(\hat\gamma)} &= \sup_{\widehat{w}\in L^2(\hat\gamma)}\frac{\int_0^1 \widehat{v}\, \widehat{w}~\mathrm{d}x}{\|\widehat{w}\|_{L^2(\hat\gamma)}} = \sup_{\widehat{z}\in H^1_{\rm zmv}(\hat\gamma)} \frac{\int_0^1 \widehat{v} ~ \partial_x \widehat{z}~\mathrm{d}x}{\left|\widehat{z}\right|_{Z}} \\&
\leq 
C \sup_{\widehat{z}\in H^1_{\rm zmv}(\hat\gamma)} \frac{\int_0^1   \widehat{z}~\partial_x  \widehat{v} ~\mathrm{d}x}{\|\widehat{z}\|_{Z}} \leq  C\|  \partial_x  \widehat{v} \|_{Z'},
\end{align*}
where $C$ is the inverse of the Poincar\'e constant, i.e., $\|\widehat{z}\|_{Z} \leq {C}^{-1} \left| \widehat{z} \right|_{Z}$ for $\widehat{z}\in H^1_{\rm zmv}(\widehat\gamma)$. 

Case $d=3$.
Given  $\widehat{s}^{p-2} \in S^{p-2}$, we construct the spline
$\widehat{s}^{p-1}(x) = \int_0^x \int_0^y \widehat{s}^{p-2}(\xi, \eta) ~\mathrm{d}\eta ~\mathrm{d}\xi - \widehat{f}^{p-1}(x) - \widehat{g}^{p-1}(y) - m$, where $m\in \R$ and $\widehat{f}^{p-1},\widehat{g}^{p-1}$ are univariate splines of degree $p-1$ with zero mean value. These unknowns can be chosen such that $\widehat{s}^{p-1} \in S^{p-1}_{\rm zmv}$. As for the univariate case, given  $\widehat{s}^{p-1} \in S^{p-1}_{\rm zmv}$ we consider $\widehat{s}^p(x,y) = \int_0^x \int_0^y \widehat{s}^{p-1}(\xi,\eta)~ \mathrm{d}\eta~ \mathrm{d}\xi $ and it holds $\widehat{s}^p \in S^p_0$.

For the proof of the coercivity, partial integration needs to be performed twice. The integration will be shown in more details in the proof of Theorem~\ref{thm:inf_sup_p-2}.
\end{proof}

To apply the bijectivity of the derivative in the proof of the inf-sup condition, we can no longer work with the $L^2$ norm, but need to consider the $Z'$ and $Z$ norm.
 The following lemma states an auxiliary  stability result in these norms.

\begin{lemma}\label{lem:inf_sup_symmetry}
For any $\widehat{g}^{p-1} \in S^{p-1}_{\rm zmv}$, it holds
\begin{equation*}
 \sup_{\widehat{f}^{p-1}\in S^{p-1}_{\rm zmv} } \frac{\int_{\widehat \gamma} \widehat{g}^{p-1}\widehat{f}^{p-1}~\mathrm{d}x}{\|\widehat{f}^{p-1} \|_{Z'}}  \geq C \|\widehat{g}^{p-1}\|_{Z}.
\end{equation*}
\end{lemma}
\begin{proof}
The equal order pairing $Z-Z'$ inf-sup condition of $ S^{p-1}$ is first considered by introducing the Fortin operator $\Pi: L^2 \rightarrow S^{p-1}$   and proving its $Z$ stability. Then we show that the inf-sup condition remains satisfied for the constrained space $S^{p-1}_{\rm zmv}$. Since the infinum over a sub-space is an upper-bound of the infinum over a space, the critical part is the restriction of the primal space. 

Case $d=2$.
Standard techniques show that the Fortin operator associated with $S^{p-1}$, which is the $L^2$ projection, is uniformly $Z$ stable, see, e.g.,~\cite[Lemma 1.8]{lamichhane:06}. Thus the $Z-Z'$ inf-sup condition holds on $S^{p-1}$, i.e., for $\widehat{q}^{p-1} \in S^{p-1}$ it holds,
\begin{equation}\label{inf_sup_Z_Z_dual_standard}
\sup_{\widehat{r}^{p-1} \in S^{p-1} } \frac{\int_{\widehat \gamma} \widehat{r}^{p-1}\widehat{q}^{p-1} ~\mathrm{d}x}{ \|\widehat{r}^{p-1}\|_{Z}  } \geq C\|\widehat{q}^{p-1} \|_{Z'}.
\end{equation}
Next, we show that the restriction to $S^{p-1}_{\rm zmv}$ retains this stability.

Let us consider $\widehat{f}^{p-1}\in S^{p-1}_{\rm zmv}$, since the inf-sup condition remains satisfied for $\widehat{q}^{p-1} \in S^{p-1}$ and $\widehat{f}^{p-1}\in S^{p-1}_{\rm zmv}$. Let  us define $\widehat{g}^{p-1}\in S^{p-1}_{\rm zmv}$ such that $\widehat{g}^{p-1}(x) = \widehat{q}^{p-1}(x) - \int_{\widehat \gamma} \widehat{q}^{p-1}(\xi) \mathrm{d}\xi  \in S^{p-1}_{\rm zmv}$ and note that for $\widehat{f}^{p-1}\in S^{p-1}_{\rm zmv}$ 
\[
\int_{\widehat \gamma} \widehat{f}^{p-1}  \widehat{q}^{p-1} \mathrm{d}x=\int_{\widehat \gamma} \widehat{f}^{p-1} \widehat{g}^{p-1} \mathrm{d}x  
\]
and $\| \widehat{g}^{p-1}\|_{Z} \leq \| \widehat{q}^{p-1} \|_{Z}$.
This shows
\begin{equation*}
\inf_{ \widehat{f}^{p-1}\in S^{p-1}_{\rm zmv} } \sup_{{\widehat{g}}^{p-1} \in S^{p-1}_{\rm zmv} } \frac{\int_{\widehat \gamma} \widehat{g}^{p-1}\widehat{f}^{p-1} ~\mathrm{d}x}{\|\widehat{f}^{p-1} \|_{Z'} \|\widehat{g}^{p-1}\|_{Z}  } \geq C > 0.
\end{equation*}
Now using~\cite[Proposition 3.4.3]{brezzi:13}, we interchange the spaces of the infimum and the supremum which yields the result.

Case $d=3$. Although we follow the same structure as in the case $d=2$, there are some essential differences. We note that $Z=H^{1,1}(\widehat \gamma)$ is no longer a standard Sobolev space, and thus the $Z$ stability of the Fortin operator cannot  be shown as in the case $d=2$. Instead, we make use of a tensor product of the univariate Fortin operators. See~\cite{sangalli:12} for another application of a tensor product of projection operators. 

We first show, that the tensor product of univariate $L^2$ projections is the multivariate $L^2$ projection, i.e., the Fortin operator. Then we show that the $H^1$ stability of the univariate projections yield the $Z$ stability of their tensor product. 
We define
 $\overline \Pi_i: L^2(0,1) \rightarrow S^{p-1}(\Xi_i)$ as the $L^2$ projection into the univariate spline space.  Their tensor product $\Pi = \overline \Pi_1 \otimes \overline \Pi_2$ is defined as described in the following.   We first extend the projections to $\widehat \gamma$ by $\Pi_1: L^2(\widehat \gamma) \rightarrow L^2(\widehat \gamma)$ and $\Pi_2: L^2(\widehat \gamma) \rightarrow L^2(\widehat \gamma)$, such that
\[
[\Pi_1 \widehat{f}](\xi, \eta) = [\overline \Pi_1 \bar f_\eta](\xi), \quad 
[\Pi_2 \widehat{f}](\xi, \eta) = [\overline \Pi_2 \bar{f}_\xi](\eta).
\]
Here  $\bar f_\eta$ denote the univariate function depending on $\xi$, where the coordinate $\eta$ plays the role of a parameter. $\bar f_\xi$ is defined analogously and it holds $\widehat{f}(\xi, \eta) = \bar f_\xi(\eta) = \bar f_\eta(\xi)$. Now the tensor product of the projections can be defined as $\Pi = \overline \Pi_1 \otimes \overline \Pi_2: L^2(\widehat \gamma) \rightarrow S^{p-1}$ by
$
 \overline \Pi_1 \otimes \overline \Pi_2 = \Pi_1 \circ \Pi_2 = \Pi_2 \circ \Pi_1.
$ 

Applying the univariate projection property of $\overline \Pi_i$, a direct calculation shows that $\Pi$ is the $L^2$ projection onto $S^{p-1}$. Let $\widehat{B}_{i,1}, \widehat{B}_{j,2}$ denote the univariate basis functions in the two parametric directions, then we get
\begin{equation*}
\int_{\widehat \gamma} (\Pi \widehat{v})(x,y)\widehat{B}_{i,1}(x) \widehat{B}_{j,2}(y) ~\mathrm{d}x~\mathrm{d}y  \quad=
\int_{\widehat \gamma} \widehat{v}(x,y)  \widehat{B}_{i,1}(x) \widehat{B}_{j,2}(y)~\mathrm{d}x~\mathrm{d}y.
\end{equation*}

For a fixed $\bar x, \bar y \in (0,1)$ denote $I_{\bar y} = \{(x,\bar y)\in (0,1)^2\}$ and $I_{\bar x} = \{(\bar x, y) \in (0,1)^2\}$. 
For the calculation, we need the two steps resulting from the univariate stability of the unidirectional projectors in $L^2(I_k)$ and $H^1(I_k)$ for $k=\overline{x}$ or $\overline{y}$:

First, for any $\bar y \in (0,1)$, we have
\begin{align*}
\| \partial_{xy} \Pi_1 \widehat{w} \|_{L^2(I_{\bar y} ) }& = \| \partial_{x} \Pi_1 (\partial_y  \widehat{w}) \|_{L^2(I_{\bar y} ) } = \left|  \Pi_1 (\partial_y  \widehat{w}) \right|_{H^1(I_{\bar y} ) } \leq C \|  \partial_y  \widehat{w} \|_{H^1(I_{\bar y} ) } \\&= 
C \|\partial_y  \widehat{w} \|_{L^2(I_{\bar y} ) } + C \|\partial_{xy}  \widehat{w} \|_{L^2(I_{\bar y} ) }
\end{align*}
We will use this result for $ \widehat{w} = \Pi_2 \widehat{v}$. Of course the analogue result for $\Pi_2$ and any $\bar x \in (0,1)$ also holds.

Hence, we see
\begin{align*}
\|\partial_{xy} \Pi  \widehat{v} \|_{L^2(\hat \gamma)}^2 &= \int_{y\in I^2} \| \partial_{xy} \Pi  \widehat{v}  \|_{L^2(I_y)}^2 ~\mathrm{d}y  \\&
\leq \int_{y\in I^2} \| \partial_{y} \Pi_2  \widehat{v}  \|_{L^2(I_y)}^2 ~\mathrm{d}y  + \int_{y\in I^2} \| \partial_{xy} \Pi_2  \widehat{v}  \|_{L^2(I_y)}^2 ~\mathrm{d}y \\&= 
\int_{x\in I^1} \| \partial_{y} \Pi_2  \widehat{v}  \|_{L^2(I_x)}^2 ~\mathrm{d}x  + \int_{x\in I^1} \| \partial_{xy} \Pi_2  \widehat{v}  \|_{L^2(I_x)}^2 ~\mathrm{d}x \\
&\leq C \|  \widehat{v} \|_{Z}^2,
\end{align*}
i.e., the operator is $Z$ stable.

The $Z-Z'$ stability of $S^{p-1}_{\rm zmv}$ can be concluded similarly to the univariate case starting from the $Z-Z'$ inf-sup condition for $\widehat{q}^{p-1}$ and $\widehat{r}^{p-1} \in S^{p-1}$, see~(\ref{inf_sup_Z_Z_dual_standard}). We can consider $\widehat{f}^{p-1}\in S^{p-1}_{\rm zmv}$, since the inf-sup condition remains valid for $\widehat{q}^{p-1} \in S^{p-1}$ and $\widehat{f}^{p-1}\in S^{p-1}_{\rm zmv}$. Now we  define $\widehat{g}^{p-1}\in S^{p-1}_{\rm zmv}$ such that $\widehat{g}^{p-1}(x,y) = \widehat{q}^{p-1}(x,y) - \widehat{s}^1_0(x) - \widehat{s}^2_0(y) - c \in  S^{p-1}_{\rm zmv}$ with $\widehat{s}^1_0 \in S^{p-1}(\Xi_1), \widehat{s}^2_0 \in S^{p-1}(\Xi_2)$ and $c \in \R$, and note that for $\widehat{f}^{p-1}\in S^{p-1}_{\rm zmv}$ it holds
\[
\int_{\widehat \gamma} \widehat{f}^{p-1}  \widehat{q}^{p-1} \mathrm{d}x=\int_{\widehat \gamma} \widehat{f}^{p-1} \widehat{g}^{p-1} \mathrm{d}x.  
\]
 Now, the $Z-Z'$ stability  can be concluded by noting that
 $\| \widehat{g}^{p-1}\|_{Z} \leq \| \widehat{q}^{p-1} \|_{Z} $. 
 The proof ends the same way as the case $d=2$ using~\cite[Proposition 3.4.3]{brezzi:13}.
\end{proof}

It remains to combine these preliminary results to prove the main theorem of this section. We use the bijectivity between the spline spaces of different degrees, stated in Lemma~\ref{lem:spline_derivatives}, and partial integration to estimate the inf-sup term by the equal order $p-1$ stability which was estimated in Lemma~\ref{lem:inf_sup_symmetry}.
\begin{theorem}\label{thm:inf_sup_p-2}
Let $p \geq 2$ and the knot vectors $\Xi_\delta, \, \delta=1,\,\ldots,\, d-1$, be such that $S^p(\widehat{\gamma}) \subset C^1(\widehat{\gamma})$. The dual space $\widehat{M}^{2}$ verifies
\begin{align*}
\sup_{\widehat{w} \in S^p_0} \frac{\int_{\widehat \gamma}\widehat{ \mu}\, \widehat{w}~\mathrm{d}x }{\|\widehat{w} \|_{L^2(\hat \gamma)}} \geq C \| \widehat{\mu} \|_{L^2(\hat \gamma)}, \quad \widehat{\mu} \in \widehat{M}^2
\end{align*}
with a constant $C$ independent of the mesh size, but possibly dependent on p.
\end{theorem}
\begin{proof}
As before, the cases $d=2$ and $d=3$ are considered separately. We perform partial integration, noting that in the bivariate case, a tensor product structure is exploited. 

Given any $\widehat{\mu}^{p-2} \in S^{p-2}$, we may introduce $\widehat{g}^{p-1}\in S^{p-1}_{\rm zmv}$, such that $\partial_x \widehat{g}^{p-1} = \widehat{\mu}^{p-2}$ as constructed in Lemma~\ref{lem:spline_derivatives}.

For the case $d=2$, partial integration yields
\begin{align*}
\sup_{\widehat{w}^p \in S_0^p} \frac{\int_{\widehat \gamma} \widehat{w}^p \, \widehat{\mu}^{p-2} \mathrm{d}x}{\|\widehat{w}^p\|_{L^2(\hat \gamma)} } &=
\sup_{\widehat{w}^p \in S_0^p} \frac{\int_{\widehat \gamma} \widehat{w}^p ~\partial_x  \widehat{g}^{p-1} \mathrm{d}x}{\|\widehat{w}^p\|_{L^2(\hat \gamma)} } =
\sup_{\widehat{w}^p \in S_0^p} \frac{\int_{\widehat \gamma}   \widehat{g}^{p-1} ~\partial_x \widehat{w}^p \mathrm{d}x}{\|\widehat{w}^p\|_{L^2(\hat \gamma)} }.
\end{align*}
Now, let us denote $ \widehat{f}^{p-1} = \partial_x  \widehat{w}^p \in S^{p-1}_{\rm zmv}$ and use the coercivity of the derivative as stated in Lemma~\ref{lem:spline_derivatives}. Since $\partial_x$ is bijective from $S^p_0$ onto $S^{p-1}_{\rm zmv}$, we have
\begin{align*}
\sup_{ \widehat{w}^p \in S_0^p} \frac{\int_{\widehat \gamma}  \widehat{g}^{p-1}~ \partial_x  \widehat{w}^p ~\mathrm{d}x}{\| \widehat{w}^p\|_{L^2(\hat \gamma)} } &\geq
\sup_{ \widehat{w}^p \in S_0^p} C \frac{\int_{\widehat \gamma}  \widehat{g}^{p-1} ~ \partial_x  \widehat{w}^p ~\mathrm{d}x}{\|\partial_x  \widehat{w}^p\|_{Z'} } \\&=
\sup_{ \widehat{f}^{p-1} \in S_{\rm zmv}^{p-1}} C \frac{\int_{\widehat \gamma}  \widehat{f}^{p-1} \, \widehat{g}^{p-1} ~\mathrm{d}x}{\| \widehat{f}^{p-1}\|_{Z'} }.
\end{align*}
Now, we make use of the $Z'-Z$ stability on the equal order pairing, as stated in Lemma~\ref{lem:inf_sup_symmetry}.
Since $\partial_x \widehat{g}^{p-1} = \widehat{\mu}^{p-2}$, we have
\begin{align*}
\sup_{\widehat{f}^{p-1} \in S_{\rm zmv}^{p-1}} &C \frac{\int_{\widehat \gamma} \widehat{f}^{p-1} \, \widehat{g}^{p-1}~\mathrm{d}x}{\|\widehat{f}^{p-1}\|_{Z'} } \geq 
C \|  \widehat{g}^{p-1} \|_{Z}   \geq C \left| \widehat{g}^{p-1}\right|_{Z} = C \|\widehat{\mu}^{p-2}\|_{L^2(\hat \gamma)},
\end{align*}
which yields the stated inf-sup condition.

The proof for the case $d=3$ is analogue, but special care must be taken due to the tensor product structure. In this case, the suitable differential operator is the mixed derivative $\partial_{xy}$, so the partial integration has to be performed twice. Since most parts of the proof were shown in the previous lemmas, proving the analogue partial integration formula is the only remaining part. Given $\widehat{f}^{p-2} \in S^{p-2}_0$, define $\widehat{g}^{p-1} \in S^{p-1}_{\rm zmv}$ such that $\partial_{xy} \widehat{g}^{p-1} = \widehat{\mu}^{p-2}$. We apply Gau\ss{} theorem twice and note that in both cases the boundary term vanishes
\begin{align*}
\int_{\widehat \gamma} \widehat{g}^{p-1} \, \partial_i \widehat{w}^{p}~\mathrm{d}x = \int_{\partial {\widehat \gamma}} \widehat{w}^p\, \widehat{g}^{p-1} n_i~\mathrm{d}\sigma  - \int_{\widehat \gamma} \widehat{w}^p \, \partial_i \widehat{g}^{p-1}~\mathrm{d}x,
\end{align*}
where $n_i$ is the $i$-th component of the outward unit normal on $\partial \widehat \gamma$, i.e., $n_i \in \{0,\pm 1\}$. 

Using the zero trace of $w^p \in H^{1,1}_0(\widehat \gamma)$, the first step
\begin{align*}
\int_{{\widehat \gamma}} \widehat{w}^p\, \widehat{\mu}^{p-2} \mathrm{d}x &= \int_{{\widehat \gamma}} \widehat{w}^p \partial_{xy} \widehat{g}^{p-1}\mathrm{d}x \notag\\&= 
-\int_{{\widehat \gamma}} \partial_x \widehat{w}^p\, \partial_{y} \widehat{g}^{p-1}\mathrm{d}x +  { \int_{\partial {\widehat \gamma}} \widehat{w}^p \,\partial_{y} \widehat{g}^{p-1} n_1 \mathrm{d}\sigma }\notag\\& = -\int_{{\widehat \gamma}} \partial_x \widehat{w}^p\, \partial_{y} \widehat{g}^{p-1}  \mathrm{d}x 
\end{align*}
follows.

For the second step, we use that on the part of $\partial {\widehat \gamma}$ parallel to the $x$-axis, it holds $\partial_x \widehat{w}^p = 0$. On the orthogonal part (parallel to the $y$-axis), it holds $n_2  = 0$. 
\begin{align*}
-\int_{{\widehat \gamma}} \partial_x  \widehat{w}^p\, \partial_{y} \widehat{g}^{p-1}\mathrm{d}x \notag
&=
\int_{\widehat \gamma} \partial_{xy}  \widehat{w}^p\, \widehat{g}^{p-1} \mathrm{d}x - { \int_{\partial {\widehat \gamma}} \partial_x  \widehat{w}^p\, \widehat{g}^{p-1} n_2 \mathrm{d}\sigma } \\&= \int_{\widehat \gamma} \partial_{xy}  \widehat{w}^p\, \widehat{g}^{p-1} \mathrm{d}x. 
\end{align*}
We define $\widehat{f}^{p-1} = \partial_{xy}  \widehat{w}^p \in S^{p-1}_{\rm zmv}$ and continue analogously to the univariate case. Note, that this proof is not restricted to the bivariate case, but can be applied to tensor products of arbitrary dimensions.
\end{proof}

While we considered an inf-sup condition in the parametric space (\ref{eq:para_inf_sup}), the inf-sup condition, Assumption~\ref{as:mortar_intro:inf_sup}, needs to be fulfilled in the physical domain. Now we prove from Theroem~\ref{thm:inf_sup_p-2} the inf-sup stability in the physical space.

\begin{theorem}\label{thm:inf_sup_physical} 
Let (\ref{eq:para_inf_sup}) holds and
let $M_{l,h}^2=\{ \mu = \widehat \mu \circ \mathbf{F}_{s(l)}^{-1}, \widehat \mu \in S^{p-2}(\widehat \gamma)\}$, and $W_{l,h}=\{w = ((\widehat w/\widehat{DW}) \circ \mathbf{F}_{s(l)}^{-1}), \widehat w \in S^{p}_0(\widehat \gamma)\}$ be respectively the Lagrange multiplier space and the primal trace space given in the physical domain. Then, for $h$ sufficiently small, the pairing $W_{l,h} - M_{l,h}^2$ fulfills a uniform inf-sup condition, i.e., for each $\mu \in M_{l,h}^2$, it holds
\begin{align*}
\sup_{{w} \in W_{l,h}} \frac{\int_{\gamma} \mu\, w~\mathrm{d}\sigma }{\|{w} \|_{L^2(\gamma)}} \geq C \| {\mu} \|_{L^2(\gamma)}.
\end{align*}
\end{theorem}
\begin{proof}
After a change of variable, the integral over the physical boundary can be expressed as a weighted integral over the parametric space. The proof is based on a super-approximation of the product of the dual variable with the weight. In contrast to the previous proofs, we do not need to distinguish between the cases $d=2$ and $3$.

We recall the transformation of the integral onto the parametric space (\ref{change_of_variable})
 \begin{align*}
\int_{\gamma} \mu \,w~\mathrm{d}\sigma =
\int_{\widehat \gamma} \widehat \mu \,\widehat w \,\rho  ~\mathrm{d}x,
\end{align*}
where $ \rho =  (\widehat{DW})^{-1} \,\left| \det \nabla_\gamma \mathbf{F}_{s(l)} \right|$  is uniformly bounded by above and below, fulfills $ \rho \in C^{p-2}(\hat \gamma)$ and is $h$-independent. We also note the norm equivalence
 \begin{align}\label{eq:weighted_norm_equivalence}
 C^{-1} \| \widehat v\|_{L^2(\hat \gamma)} \leq \|  \rho \widehat v\|_{L^2(\hat \gamma)} \leq C \| \widehat v\|_{L^2(\hat \gamma)}.
 \end{align}

Let $\Pi: L^2(\widehat \gamma) \rightarrow S^{p-2}(\widehat \gamma)$ denote any local projection with best approximation properties, e.g,~\cite[Equation 37]{bazilevs:06}, the following super-approximation holds
 \begin{align} \label{eq:approximation_mu_rho}
\|\widehat \mu  \rho - \Pi(\widehat \mu  \rho)\|_{L^2(\hat \gamma)} \leq C h \|\widehat \mu \|_{L^2(\hat \gamma)}.
\end{align}The proof of the super-approximation given in~\cite[Theorem 2.3.1]{wahlbin:95} can be easily extended to the isogeometric setting using the standard approximation results for splines, see~\cite{bazilevs:06}.

Then, for $\mu = \widehat \mu \circ \mathbf{F}_{s(l)}^{-1}$, we choose $\widehat w_{\hat \mu \rho} \in S^p_0(\widehat \gamma)$, such that \[\frac{\int_{\widehat \gamma} {\widehat{w}_{\hat \mu \rho}}  \Pi( \widehat \mu \rho) ~\mathrm{d}x }{\|\widehat{w}_{\hat \mu \rho} \|_{L^2(\gamma)}} \geq C \| \Pi( \widehat \mu \rho) \|_{L^2(\hat \gamma)}.\] 
We replace in the inf-sup integral the term $\widehat \mu  \rho$ by its projection, use the super-approximation and the norm equivalence (\ref{eq:weighted_norm_equivalence}) to obtain:
\begin{align*}
\sup_{{w} \in W_{l,h}} \frac{\int_{\gamma} \mu \,w~\mathrm{d}x }{\|{w} \|_{L^2(\gamma)}} &\geq  C~
\sup_{{\widehat w} \in S^p_0(\hat \gamma)} \frac{\int_{\widehat \gamma} \widehat w \,\widehat \mu \rho ~\mathrm{d}x }{\|\widehat{w} \|_{L^2(\hat \gamma)}} \\&=
 C~ \frac{\int_{\widehat \gamma} \widehat w_{\hat \mu \rho} \,\Pi( \widehat \mu \rho) ~\mathrm{d}x }{\|\widehat{w}_{\hat \mu \rho} \|_{L^2(\hat \gamma)}} +  C\frac{ \int_{\widehat \gamma} \widehat w_{\hat \mu \rho} (\widehat \mu \rho- \Pi( \widehat \mu \rho) )~\mathrm{d}x }{\|\widehat{w}_{\hat \mu \rho} \|_{L^2(\hat\gamma)}} \\
  &\geq C \| \Pi( \widehat \mu \rho) \|_{L^2(\hat \gamma)} - C\|\widehat \mu \rho- \Pi( \widehat \mu \rho) \|_{L^2(\hat \gamma)} \\&\geq  C \| \Pi( \widehat \mu \rho) \|_{L^2(\hat \gamma)} - C' h \|\widehat \mu \rho\|_{L^2(\hat \gamma)}.
\end{align*}
Now, we use the approximation result~(\ref{eq:approximation_mu_rho}) and the norm equivalence (\ref{eq:weighted_norm_equivalence}) to bound $\| \Pi( \widehat \mu \rho) \|_{L^2(\hat \gamma)}$:
\begin{align*}
\|  \widehat \mu   \|_{L^2(\hat \gamma)} \leq \| \Pi( \widehat \mu \rho) \|_{L^2(\hat \gamma)} + \| \Pi( \widehat \mu \rho) - \widehat \mu \rho\|_{L^2(\hat \gamma)} \leq \| \Pi( \widehat \mu \rho) \|_{L^2(\hat \gamma)} + C'' h \|  \widehat \mu   \|_{L^2(\hat \gamma)},
\end{align*}
which shows $\| \Pi( \widehat \mu \rho) \|_{L^2(\hat \gamma)}\geq C\|  \widehat \mu  \|_{L^2(\hat \gamma)}$ for sufficiently small $h$.  Then standard norm equivalences show the inf-sup condition in the physical domain.
\end{proof}

\begin{remark}
An analogue proof shows the stability of a pairing of order $p$ and $p-2k\geq 0$ for $k\in \mathbb{N}$. However, for $k>1$ the dual approximation order in the $L^2$ norm $p-2k$ is very low and will reduce the convergence order drastically, i.e., to $p-2k+3/2$. 
Since for Signorini and contact problems, the regularity of the solution is usually bounded by $H^{5/2-\varepsilon}(\Omega)$, see, e.g.~\cite{moussaoui:92}, low dual degrees might be reasonably used in these cases. 
\end{remark}

\subsection{Choice 3: stable p/p pairing with boundary modification}\label{sec:p_p}
The first two choices had been motivated by Assumptions~\ref{as:mortar_intro:inf_sup}  and~\ref{as:mortar_intro:approximation_order}. While the choice 1 does not yield uniformly stable pairings, the choice 2 does not guarantee optimal order $p$ convergence. Thus consider the natural equal order pairing in more details. In the finite element context, it is well-known that the simple choice of taking the space of Lagrange multiplier as the space of traces from the slave side yields to troubles at the so-called cross points for $d=2$ and wirebaskets for $d=3$, i.e., $(\bigcup_{l\neq j} \partial \gamma_l \cap \partial \gamma_j) \cup (\bigcup_{l} \partial \gamma_l\cap \partial \Omega_D)$. As a remedy, in the finite element method a modification is performed, see~\cite{ben_belgacem:97, wohlmuth:01}. We adapt this strategy to isogeometric analysis, thus a modification of the dual spaces is performed to ensure at the same time accuracy, see Assumption~\ref{as:mortar_intro:approximation_order}, and stability, see Assumption~\ref{as:mortar_intro:inf_sup}. This modification results in a reduction of dimension of the dual space such that a counting argument for the dimensions still holds. Roughly speaking there are two possibilities: in the first case, the mesh for the Lagrange multiplier is coarsened locally in the neighborhood of the cross point (wirebasket), and in the second case the degree is reduced in the neighborhood of the cross point (wirebasket). Here we only consider the second possibility.
 
Let us start the construction for the univariate case ($d=2$), since the construction for the bivariate case ($d=3$) can be done as a tensor product. Given an open knot vector and the corresponding B-Spline functions $\widehat{B}_i^p$. We define the modified basis $\widetilde B_i^p$, $i=2,\,\ldots,\,n-1$ as follows
\begin{align*}	
\widetilde{B}_i^p(\zeta)  =\begin{cases}
\widehat{B}_i^p(\zeta) + \alpha_i \widehat{B}_1^p(\zeta),\quad &i \in \{2,\,\ldots,\,p+1\}, \\
\widehat{B}_i^p(\zeta), \quad& i\in \{p+2,\,\ldots,\, n-p-1\},\\
\widehat{B}_i^p(\zeta) + \beta_i \widehat{B}_n^p(\zeta), \quad &i \in \{n-p, n-1\}.
\end{cases}
\end{align*}
The coefficients $ \alpha_i$ and $\beta_i$ are chosen such that the basis function is a piecewise polynomial of degree $p-1$ on the corresponding element while retaining the inter-element continuity on $\widehat{\gamma}$, i.e., as 
\begin{align*}
\alpha_i = - \widehat{B}_i^{p \, {(p)}}(\zeta) / \widehat{B}_1^{p \, {(p)}}(\zeta), \quad \zeta \in (0,\zeta_{2}),\\
\beta_i = - \widehat{B}_i^{p \, {(p)}}(\zeta) / \widehat{B}_n^{p\,{(p)}}(\zeta), \quad \zeta \in  (\zeta_{E-1}, 1).
\end{align*}
An example for degree $p=3$ is shown in Figure~\ref{fig:cp_mod_basis_p3}.
Note that $\widehat{B}_i^p$ is a polynomial of degree $p$ on one single element, so the coefficients are well-defined and constant. Since derivatives of B-Spline functions are a combination of lower order B-Spline functions, a recursive algorithm for the evaluation exists, see~\cite[Section 2.1.2.2]{hughes:09}. Using the recursive formula it can easily be seen that the coefficients are uniformly bounded under the assumption of quasi-uniform meshes. 
We  define the space of Lagrange multipliers of the same order as the primal basis, as $\widehat{M}^0 = \spann_{2,\,\ldots,\, n-1} \{ \widetilde{B}_{i}^{p} \}$.
The construction guarantees that the resulting basis forms a partition of unity.
 
\begin{figure}[htbp]
\begin{center}
\includegraphics[width=0.7\textwidth]{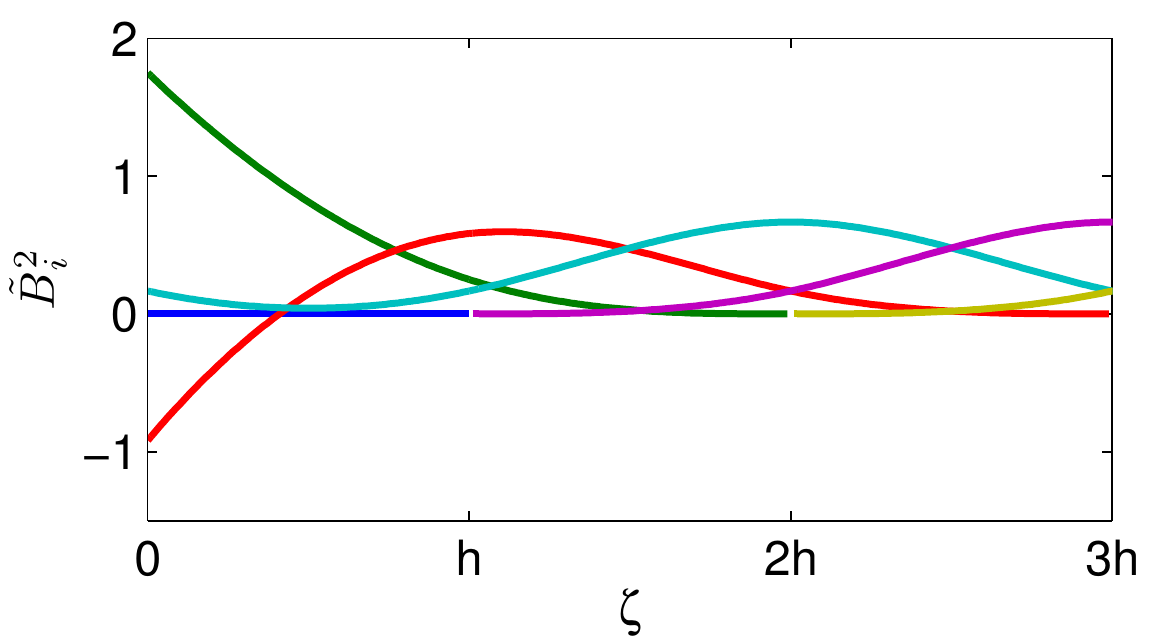}
 \end{center}
 \caption{Boundary modification of a spline of degree $3$ for $d=2$, left modification. }
 \label{fig:cp_mod_basis_p3}
 \end{figure}

 \begin{theorem} \label{th:cm_approx}
 Assumption~\ref{as:mortar_intro:approximation_order} holds for the dual space $\widehat{M}^0$.
 \end{theorem}
\begin{proof}
Since the space of global polynomials of degree $p-1$ is contained in the dual space $\widehat{M}^0$, we can directly argue as in~\cite[Section 3]{bazilevs:06}.
\end{proof}

\subsection{Stability for the three choices} \label{sec:physical_domain}
Finally hereafter, we summarize the results for the three pairings considered:  
\begin{itemize}
\item the pairing $p/p-1$  satisfies the necessary convergence order $p$ in the $L^2$ norm, Assumption~\ref{as:mortar_intro:approximation_order}, but it does not fulfill Assumption~\ref{as:mortar_intro:inf_sup}. As a result, Theorem~\ref{thm:convergence_rates} cannot be applied and no optimal convergence can be expected.
\item the pairing $p/p-2$ fulfills Assumption~\ref{as:mortar_intro:inf_sup} and Assumption~\ref{as:mortar_intro:approximation_order}, hence this choice yields an order $p-1/2$ convergence by Theorem~\ref{thm:convergence_rates}.

\item the pairing $p/p$ cannot satisfy Assumption~\ref{as:mortar_intro:inf_sup} without a crosspoint modification. We propose a modification based on a local degree reduction at the boundary of the interface and show the uniform inf-sup stability numerically. And obviously it ensures Assumption~\ref{as:mortar_intro:approximation_order}, hence Theorem~\ref{thm:convergence_rates} guarantees  an optimal convergence order $p$. 
\end{itemize}

\section{Numerical results}
\label{chapter5}
In this section, we apply the proposed mortar method to five examples, in order to validate its optimality and enlighten some additional practical aspects. All our numerical results were obtained on a Matlab code, using GeoPDEs,~\cite{geopdes:11}. 
Previous to the examples, we numerically evaluate the inf-sup constants for the considered spaces, and also for further choices of even lower degree. The first example is a multi-patch NURBS geometry with a curved interface, for which the computed $L^2$ and broken $V$ rates are optimal. 
The second example is a re-entrant corner, where we investigate, whether the presence of a singularity disturbs the proposed mortar method. Since the results are as expected, it can be said that the singularity does not have a large influence on the proposed coupling. An interface problem with jumping coefficients is considered as a third example, since for these problems domain decomposition methods are very attractive.
Although NURBS are capable of exactly representing many geometries, it is not always possible to have a matching interface between subdomains. For this reason in the fourth example, we introduce an additional variational crime by a geometry approximation. It can be seen, that the proposed method is robust with respect to a non-matching interface. The last example is a problem of linear elasticity and it is shown that the mortar method behaves as well as for scalar problems.

%
%
\subsection{A numerical evaluation of the inf-sup condition} \label{numerics:subsection_chappelle} 
We consider one  subdomain $\Omega_k$ resulting from the identity mapping of the unit square and assume that its mesh is uniformly refined. We identify  elements in $M_{l,h}$ and $W_{l,h}$ with its algebraic vector representation. Then the inf-sup condition, on one interface $\gamma_l$, reads
\begin{equation}\label{eq:numerics:inf-sup_test}     
         \inf_{\mu \in \mathbb{R}^{n'} }  \sup_{v\in \mathbb{R}^n} 
\frac {{\mu^\top} {G}  \,{v} }
         {\left( {\mu^\top}{S} {\mu} \right)^{1/2} \left( {v^\top} {T} {v} \right)^{1/2} }
         \geq C > 0,
\end{equation}   
where $n' = \dim M_{l,h}$ and $n = \dim W_{l,h}$ and $G, S, T$ denote the $L^2$ inner product matrices. 
Here we use the technique of Chapelle and Bathe,~\cite{chapelle:93}, to verify our theoretical results on the inf-sup stability.
The  proof of this approach can be found in~\cite[Chapter 3]{brezzi:13}.

The $h$-dependency of the inf-sup condition was studied first for primal spaces without any Dirichlet boundary condition and with homogeneous conditions. Precisely, primal spaces are either $\{v_{|\gamma_l}, v \in V_{{s(l)},h}\}$ or $\{v_{|\gamma_l}, v \in V_{{s(l)},h}\} \cap H^1_{0}(\gamma_l)=W_{l,h}$, and dual spaces are $\{\mu = \widehat{\mu} \circ \mathbf{F}_{s(l)}^{-1}, \widehat{\mu} \in \widehat S^p\}$ or $\{\mu = \widehat{\mu} \circ \mathbf{F}_{s(l)}^{-1}, \widehat{\mu} \in \spann_{2,\,\ldots,\, n-1} \{ \widetilde{B}_{i}^{p} \}$ for same degree pairings as it is necessary to consider a boundary modification.

\begin{figure}[htbp]   
	\includegraphics[width=0.5\textwidth]{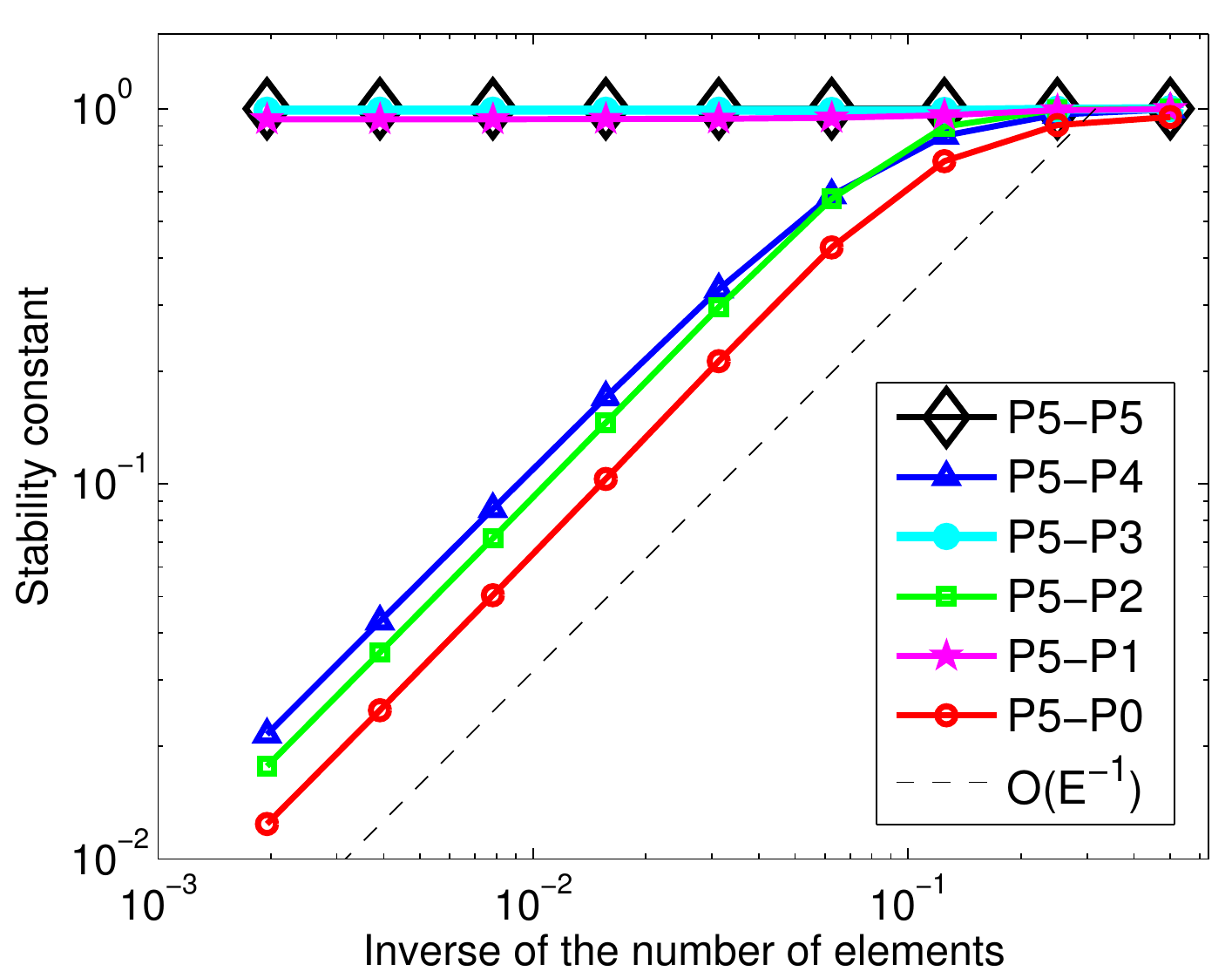}\,
	\includegraphics[width=0.515\textwidth]{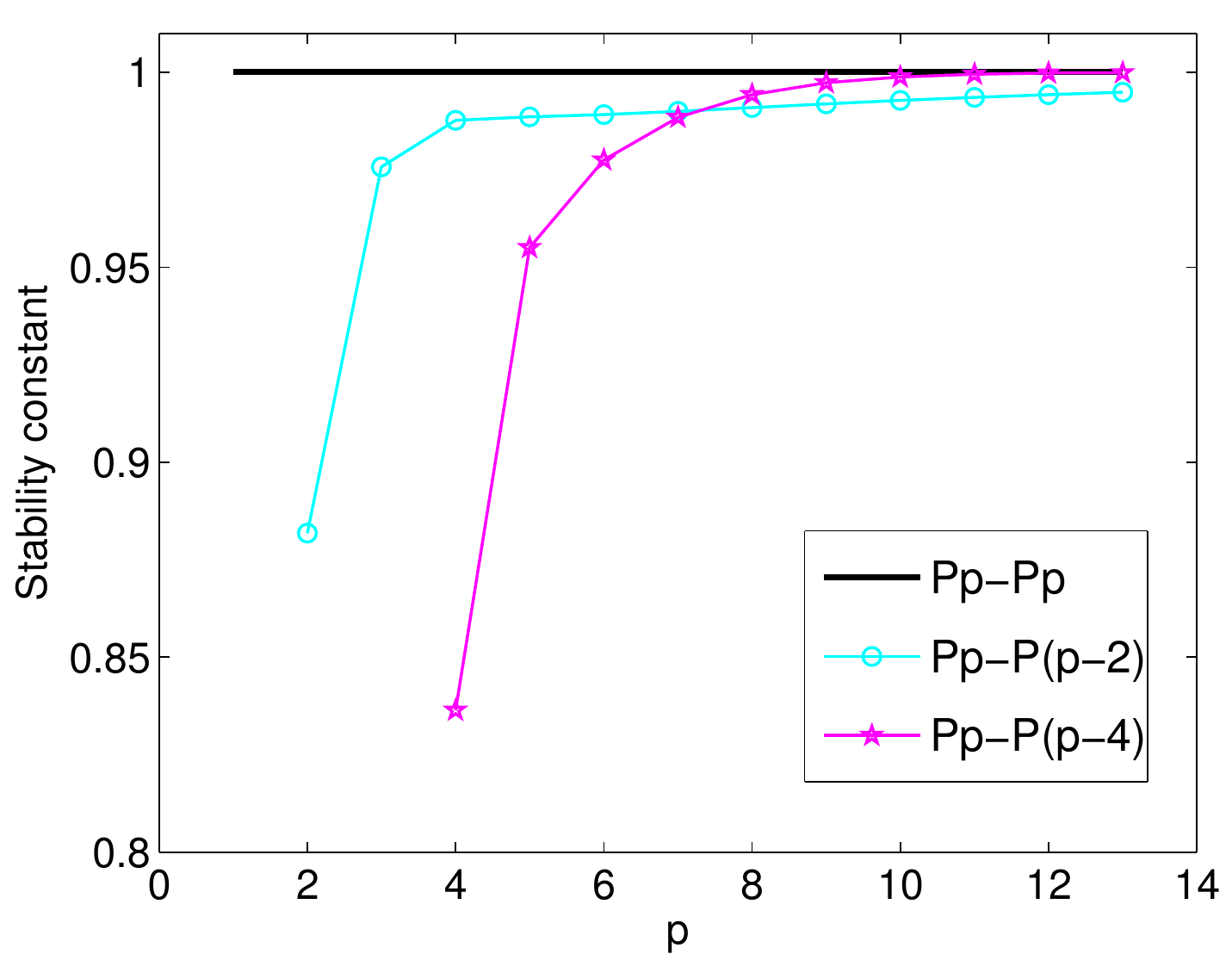} 
	\includegraphics[width=0.5\textwidth]{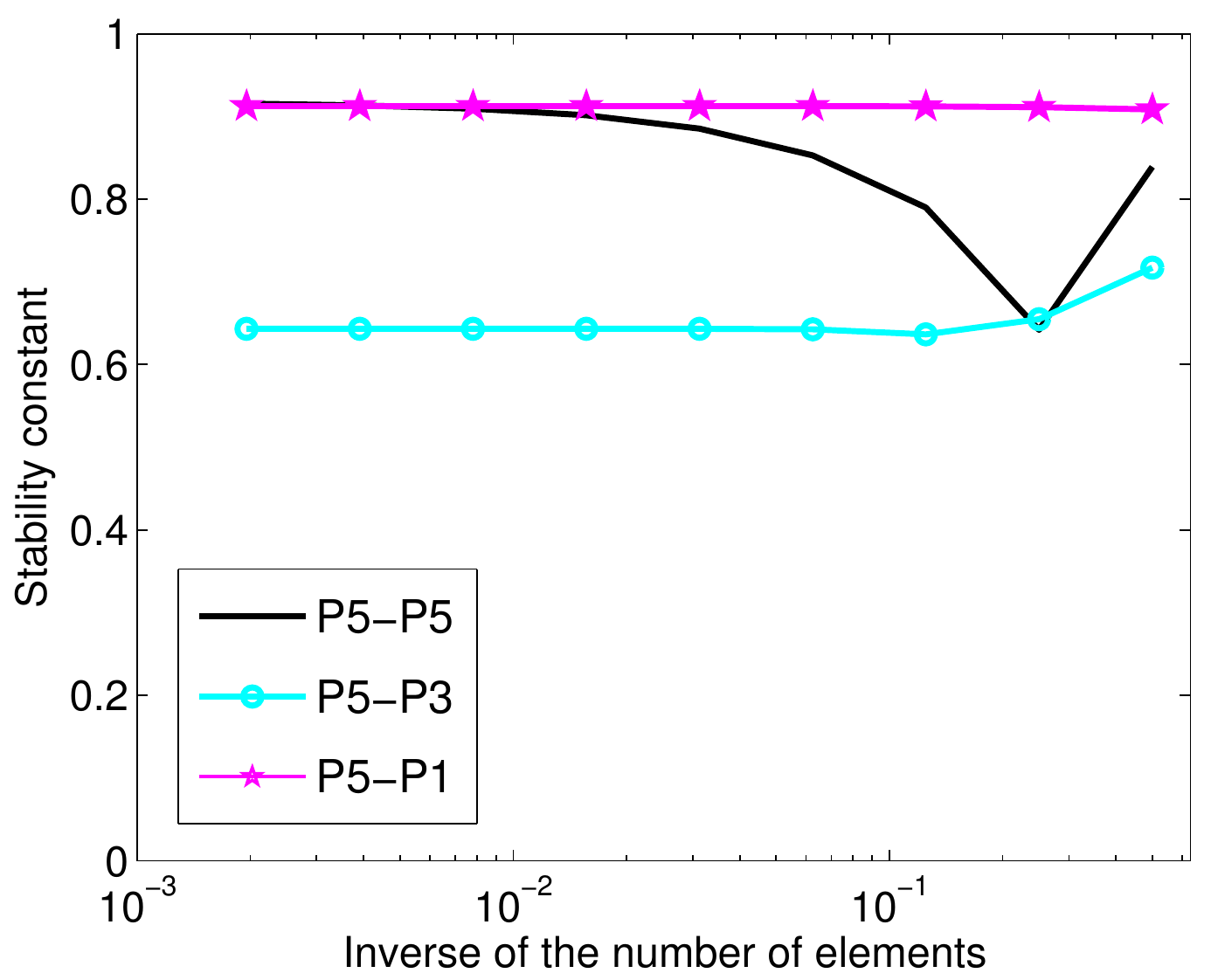}\,
	\includegraphics[width=0.515\textwidth]{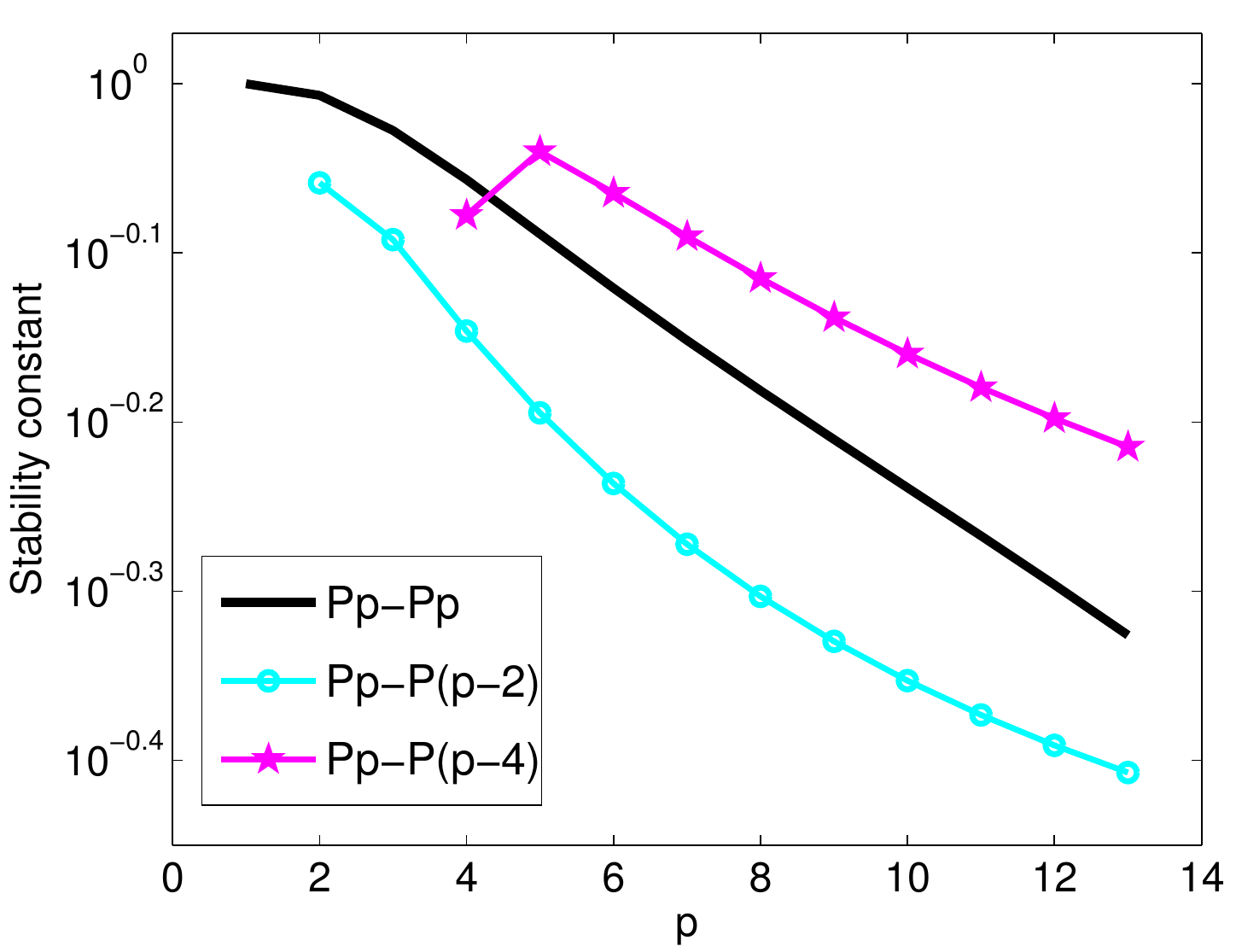}
	\caption{Problem of Subsection~\ref{numerics:subsection_chappelle} - Left: $h$-dependency for pairing $P5/Pp$ $(p=0,\,\ldots,\,5)$. Right: $p$-dependency. Top: primal spaces without boundary condition. Bottom: primal spaces with homogeneous boundary conditions. }
	\label{fig:numerics:inf-sup_test_results}
\end{figure}

This study leads us to the following conclusion: the inf-sup condition is satisfied for couples of the same parity, see Figure~\ref{fig:numerics:inf-sup_test_results} for the pairings of primal degree $p=5$. Moreover regarding the $p$-dependence, a reasonable behavior has been observed for primal space without boundary condition, whereas an exponential behavior has been found for primal space with boundary conditions, see Figure~\ref{fig:numerics:inf-sup_test_results}. 

Comparing  the  three  stable  pairings  of  the  top  right  picture  of  Figure~\ref{fig:numerics:inf-sup_test_results},  we  note  that,  although  the  dual  dimension  decreases,  the  stability  constant  gets  smaller  with  a  lower  dual  degree.  Once  more,  this  shows  that  the  inf-sup  condition  is  not  only  a  matter  of  dimensions  of  the  spaces,  especially  for  splines  for  which  the  spaces  are  not  nested  in  general.  
We also note, that considering homogeneous Dirichlet conditions, the stability constant for the case $P5/P3$ is less than for the other cases. However, the difference is quite small and should not lead to any remarkable effect.
%
%
%
%
\subsection{A scalar problem on a multi-patch NURBS domain} \label{numerics:subsection_annulus}
Let us consider the standard Poisson equation $-\Delta u =f $, solved on the domain $\Omega = \{(r,\varphi)$, $0.2<r<2$, $0<\varphi<\pi/2\}$ which is given in polar coordinates. The domain is decomposed into two patches, which are presented in Figure~\ref{fig:numerics:annulus_setting_cv_curve}. The internal load and the boundary conditions have been manufactured to have the solution $u(x,y)=\sin(\pi x) \sin(\pi y)$, given in Cartesian coordinates. To test the same degree pairing, we consider a case such that no boundary modification is required. This can be granted by setting Neumann boundary conditions on $\partial \Omega_N = \{(r, \varphi)$, $0.2<r<2$, $\varphi\in \{0, \pi/2\}\}$ and Dirichlet boundary conditions on $\partial \Omega \backslash \partial \Omega_N$.
\begin{figure}[htbp]  
\centering
	\includegraphics[width=0.4\textwidth]{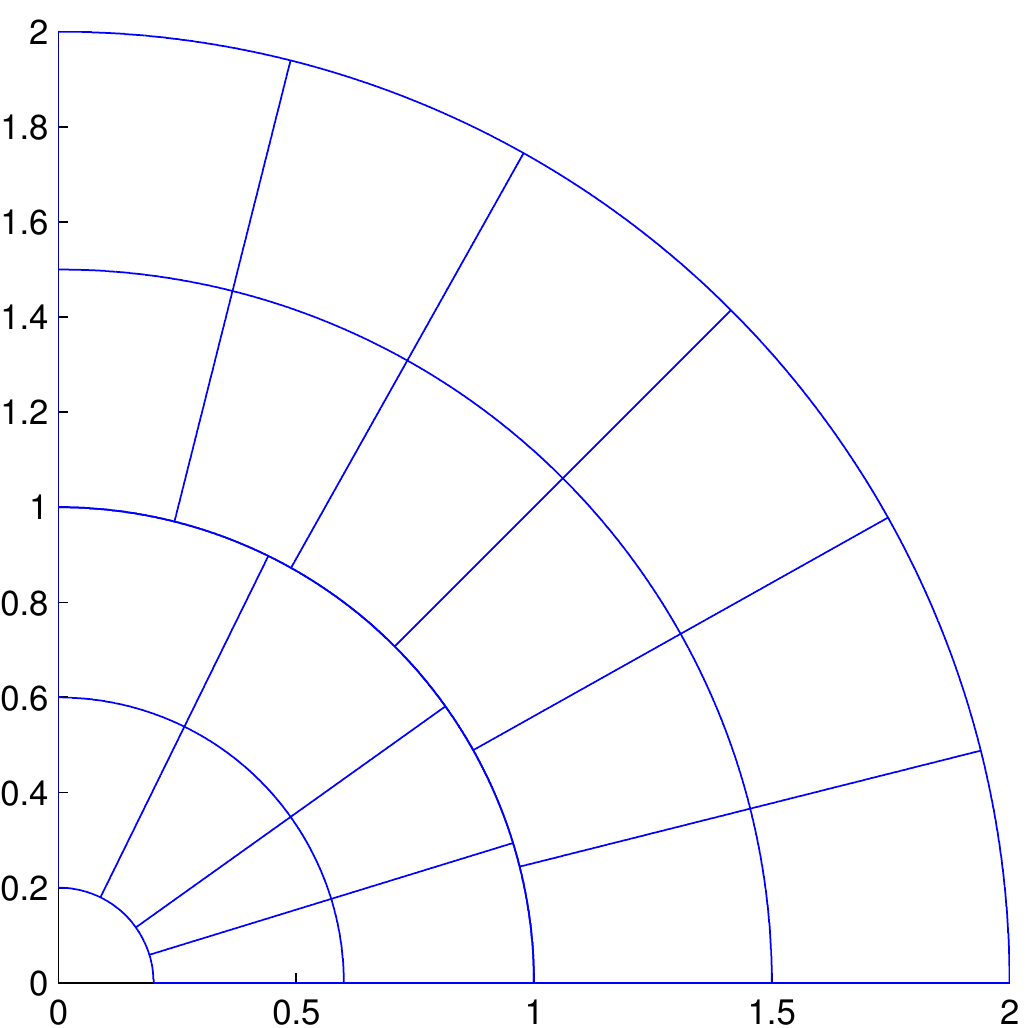}\,\,
	\includegraphics[width=0.4\textwidth]{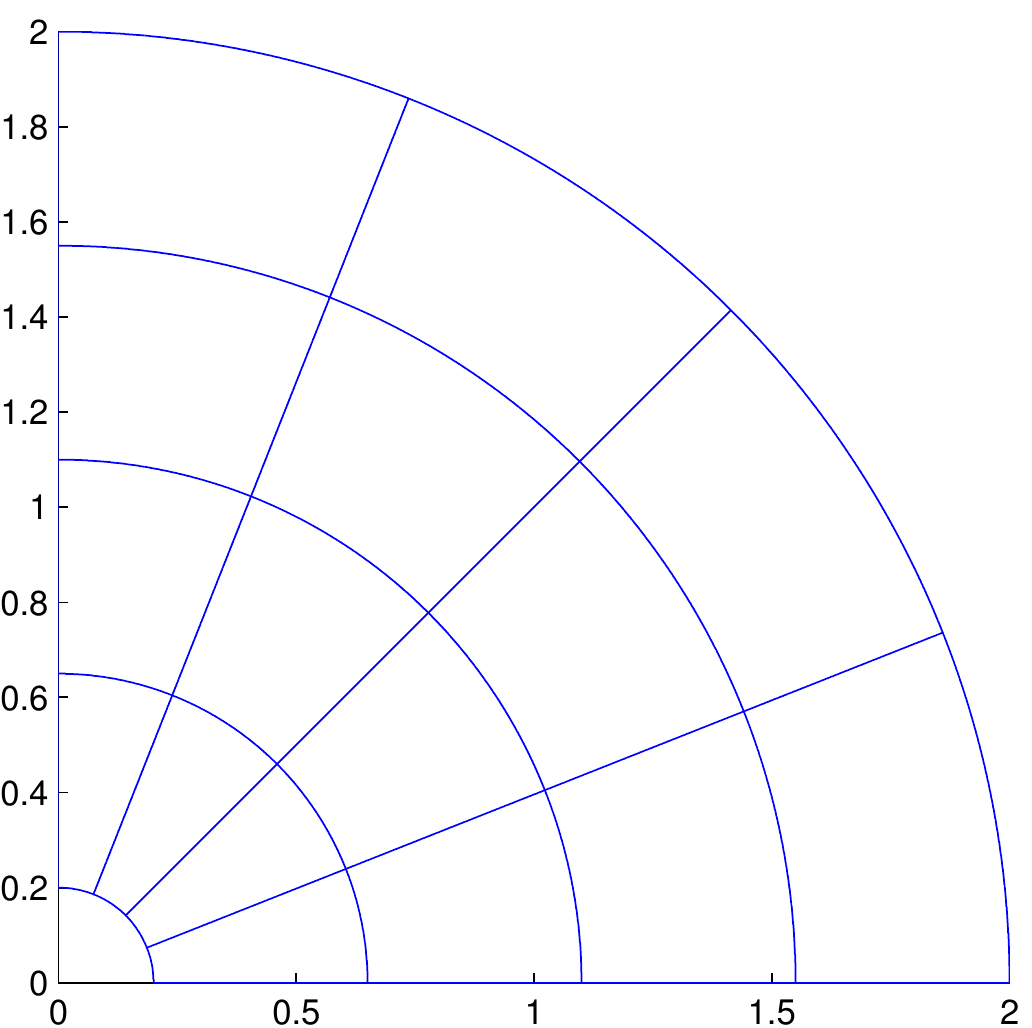}
	\caption{Problem of Subsection~\ref{numerics:subsection_annulus} - Left: a non-conforming mesh. Right: a conforming mesh. }
	\label{fig:numerics:annulus_setting_cv_curve}
\end{figure}
\begin{figure}[htbp]  
\centering
	\includegraphics[width=0.45\textwidth]{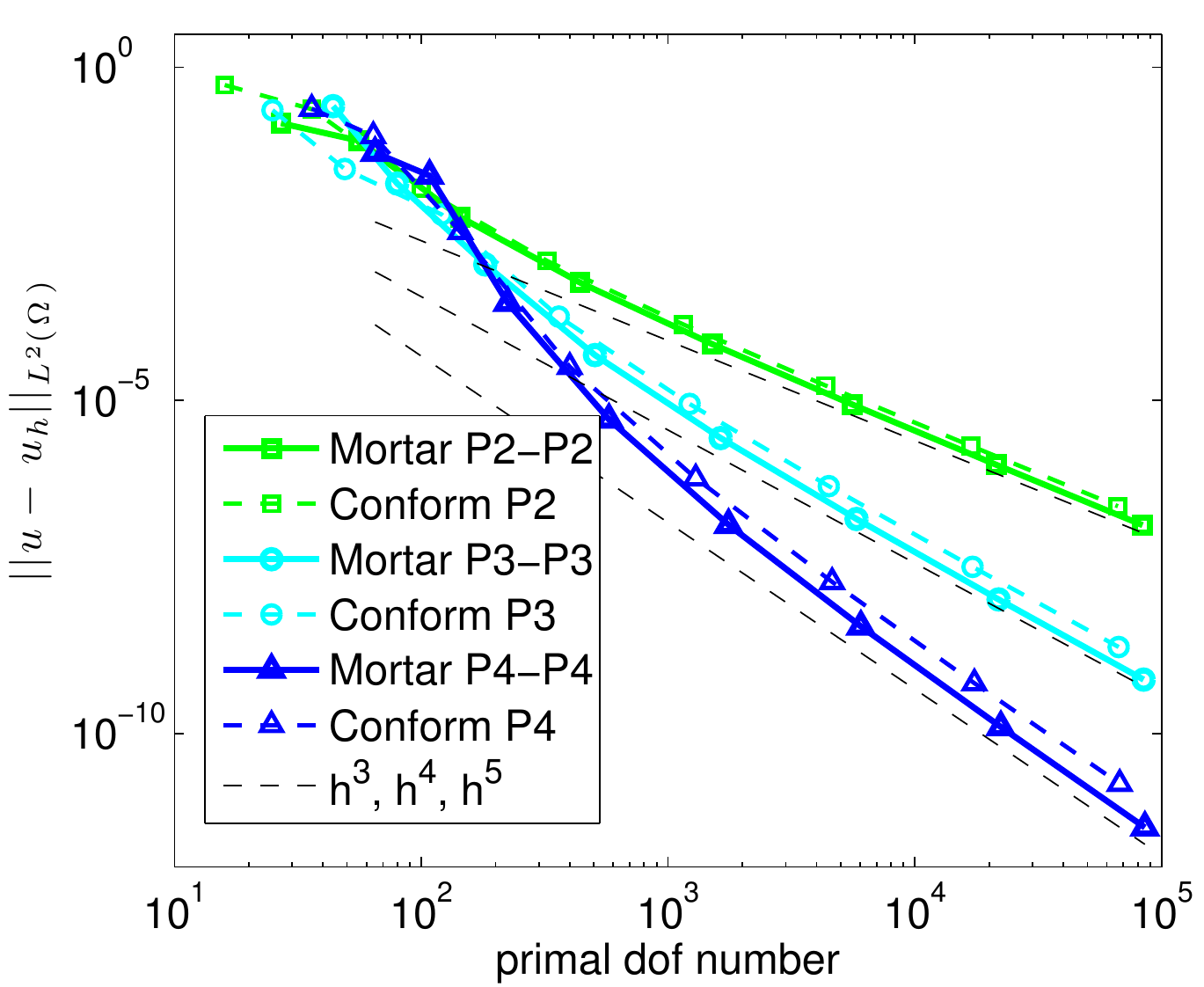}
	\includegraphics[width=0.45\textwidth]{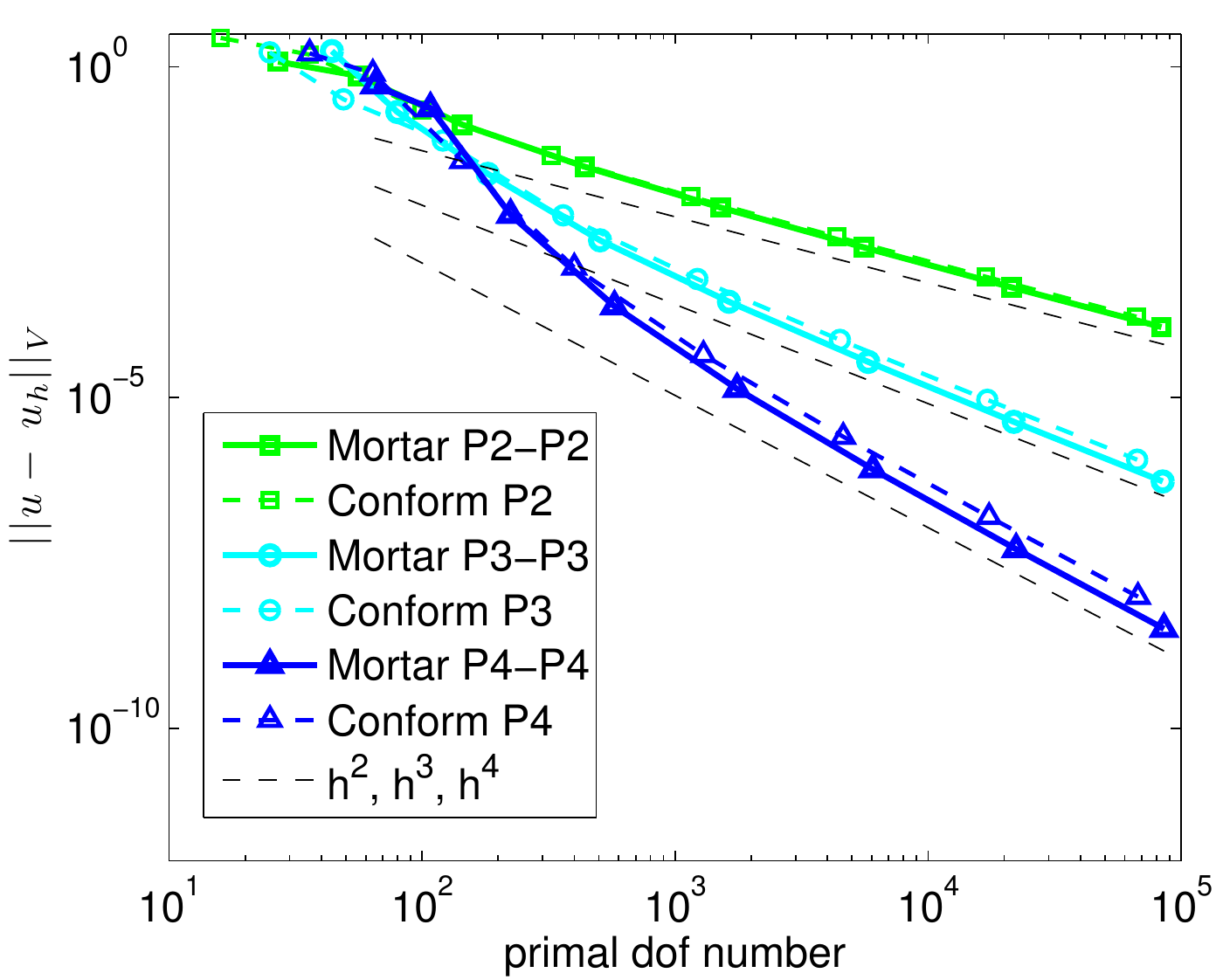}
	\caption{Problem of Subsection~\ref{numerics:subsection_annulus} - $L^2$ (left) and broken $V$ (right) primal error curves for same degree pairings.}
	\label{fig:numerics:annulus_results_cv_curve}
\end{figure}

In Figure~\ref{fig:numerics:annulus_results_cv_curve}, we show the numerically obtained error decay in the $L^2$ and the broken $V$ norm  for the primal variable and $p=2,3,4$. As expected from the theory, for an equal order $p$ pairing we observe a convergence order of $p+1$ for the $L^2$ error. We also compare the error of a matching and non-matching mesh situation and recall that in the matching case we are within the standard conforming setting. As Figure~\ref{fig:numerics:annulus_results_cv_curve} shows, no significant quantitative difference can be observed. Note that the comparison is based on results issued from similar meshes not from similar control point repartition, see Figure~\ref{fig:numerics:annulus_setting_cv_curve}. In Table~\ref{tab:annulus}, the numerically computed order of the $L^2$ error decay is given. Asymptotically, the optimal order of $p+1$ is obtained in each refinement step.
\begin{table}[htbp]
\small
\begin{tabular}{|c|cc|cc|cc|cc|}
\cline{2-7}
 \multicolumn{1}{c|}{} & \multicolumn{2}{ c| }{$P2-P2$} &  \multicolumn{2}{ c| }{$P3-P3$}  & \multicolumn{2}{ c| }{$P4-P4$} \\ 
 \hline
level         & error value & slope  
	        & error value & slope 
	        & error value & slope
\\
0 & 1.445757e-01 &   | & 2.603045e-01 & | & 5.221614e-02 & | \\ 
1 & 7.871436e-02 & 0.877 & 1.799185e-02 & 3.855 & 2.373889e-02 & 1.137 \\ 
2 & 5.651043e-03 & 3.800 & 1.100586e-03 & 4.031 & 2.897823e-04 & 6.356 \\ 
3 & 5.904159e-04 & 3.259 & 4.794994e-05 & 4.521 & 5.162404e-06 & 5.811 \\ 
4 & 7.021278e-05 & 3.072 & 2.719572e-06 & 4.140 & 1.361467e-07 & 5.245 \\ 
5 & 8.663724e-06 & 3.019 & 1.661382e-07 & 4.033 & 4.059923e-09 & 5.068 \\ 
6 & 1.079348e-06 & 3.005 & 1.033782e-08 & 4.006 & 1.253044e-10 & 5.018 \\ 
7 & 1.347999e-07 & 3.001 & 6.458495e-10 & 4.001 & 3.902800e-12 & 5.005 \\ 
\hline
\end{tabular} 
\caption{Problem of Subsection~\ref{numerics:subsection_annulus} - $|| u - u_h ||_{L^2 (\Omega)}$ and its estimated order of convergence.}
\label{tab:annulus}
\end{table}

%
\subsection{A singular scalar problem}\label{numerics:subsection_singular}
Let us now consider the Laplace equation $-\Delta u =0 $, solved on a non-convex domain with a re-entrant corner $\Omega$ decomposed into three patches, presented in Figure~\ref{fig:numerics:singular_mesh}. We need to precise for this example the mortar geometry setting. The patches are enumerated from 1 to 3 from the left to the right. We set the interface 1, as the interface between the subdomain 1 and 3, the interface 2 between 2 and 3 and the interface 3 between 1 and 2, see Figure~\ref{fig:numerics:singular_mesh}.
The singular function associated to a re-entrant corner with Dirichlet condition is given by $ r^{2/3}\sin(2/3\varphi)$, see \cite{grisvard:11}. We consider this singular case, which can be granted by setting all the boundary of $\Omega$ as a Dirichlet boundary with the value $ r^{2/3}\sin(2/3\varphi)$.

The order of the numerical method is bounded by the singularity. Standard techniques to obtain better convergence rates include the use of graded meshes, \cite{apel:96}, and $hp$-refinement,~\cite{schwab:98,buffa:14}. 
Here we do not wish to improve these rates, but to test if the proposed mortar method is disturbed by the presence of a singularity.  
 \begin{figure}[htbp] 
 	\centering
	\includegraphics[width=0.55\textwidth]{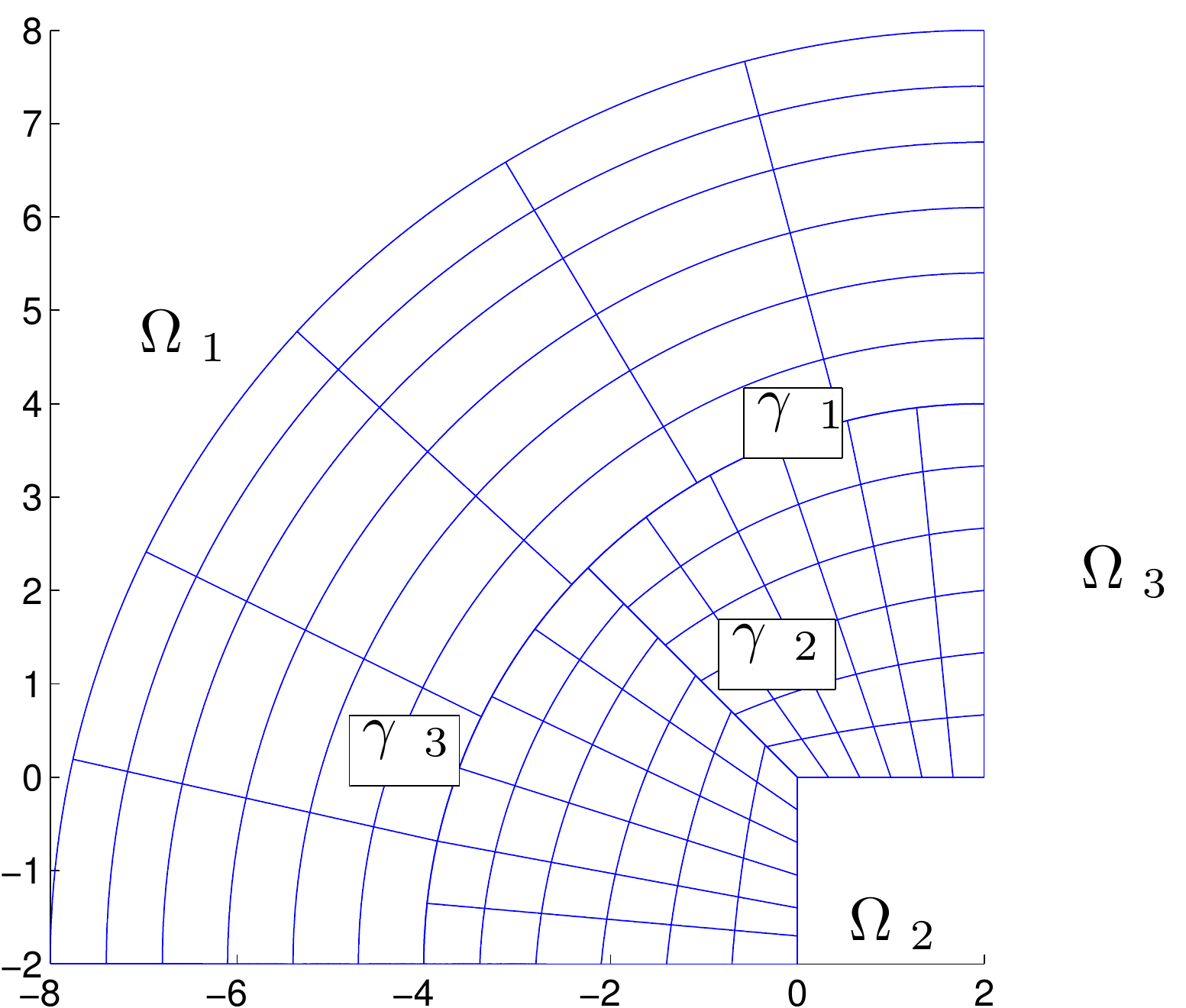}\,\,
	\caption{Problem of subsection~\ref{numerics:subsection_singular} - a non-conforming mesh.}
	\label{fig:numerics:singular_mesh} 
\end{figure}

The results are compared to the analytical solution and a numerical error study is provided. The errors are shown in  Figures~\ref{fig:numerics:singular_cv_curve_PD_new}  and~\ref{fig:numerics:singular_cv_curve_4_new}, the $L^2$ and broken $V$ errors are considered for the primal solution and the $L^2$ error for the dual solution.
 
 Considering the same degree pairing  the boundary modification is necessary and the results show the optimality of the method with respect to the regularity of the solution, see Figure~\ref{fig:numerics:singular_cv_curve_PD_new}. We note an initial bad behavior of the $L^2$ dual error on interface 2. The increase in the error might be related to the fact, that the exact Lagrange multiplier of interface 2 is zero. More precisely, the convergence rate $1/6$ for the dual variable is  a very slow rate, but induced by the regularity of the solution at this interface, as we can see that the rate on the remaining interfaces is better.  
  Moreover, we have also considered different degree pairings, and observed numerically the stability of the methods. In Figure~\ref{fig:numerics:singular_cv_curve_PD_new}, the results for the pairing $P4-P2$ and $P3-P1$ are given and show asymptotically the same convergence rates as best approximations.
 \begin{figure}[htbp] %
	\includegraphics[width=0.47\textwidth]{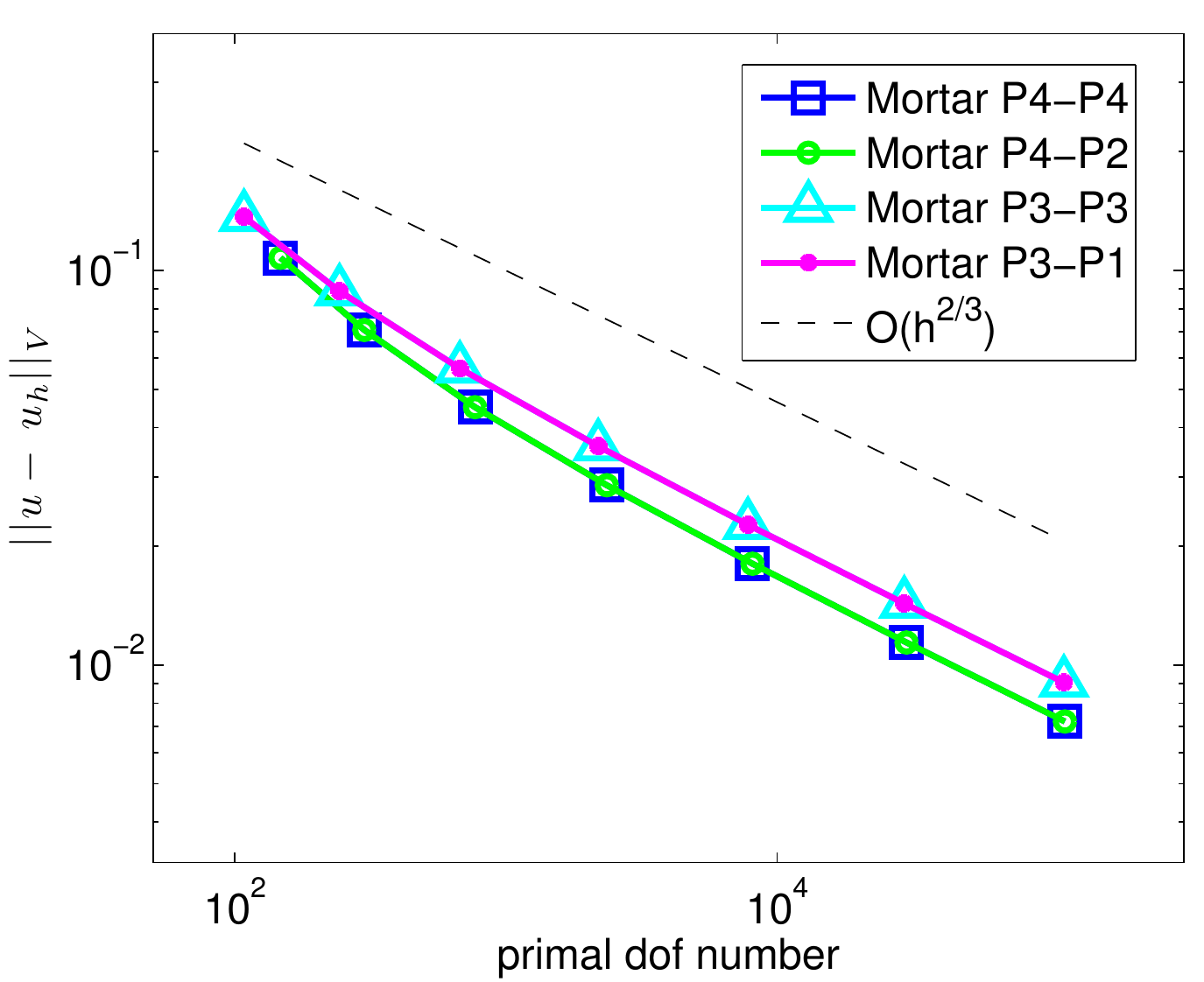}\,\,
	\includegraphics[width=0.47\textwidth]{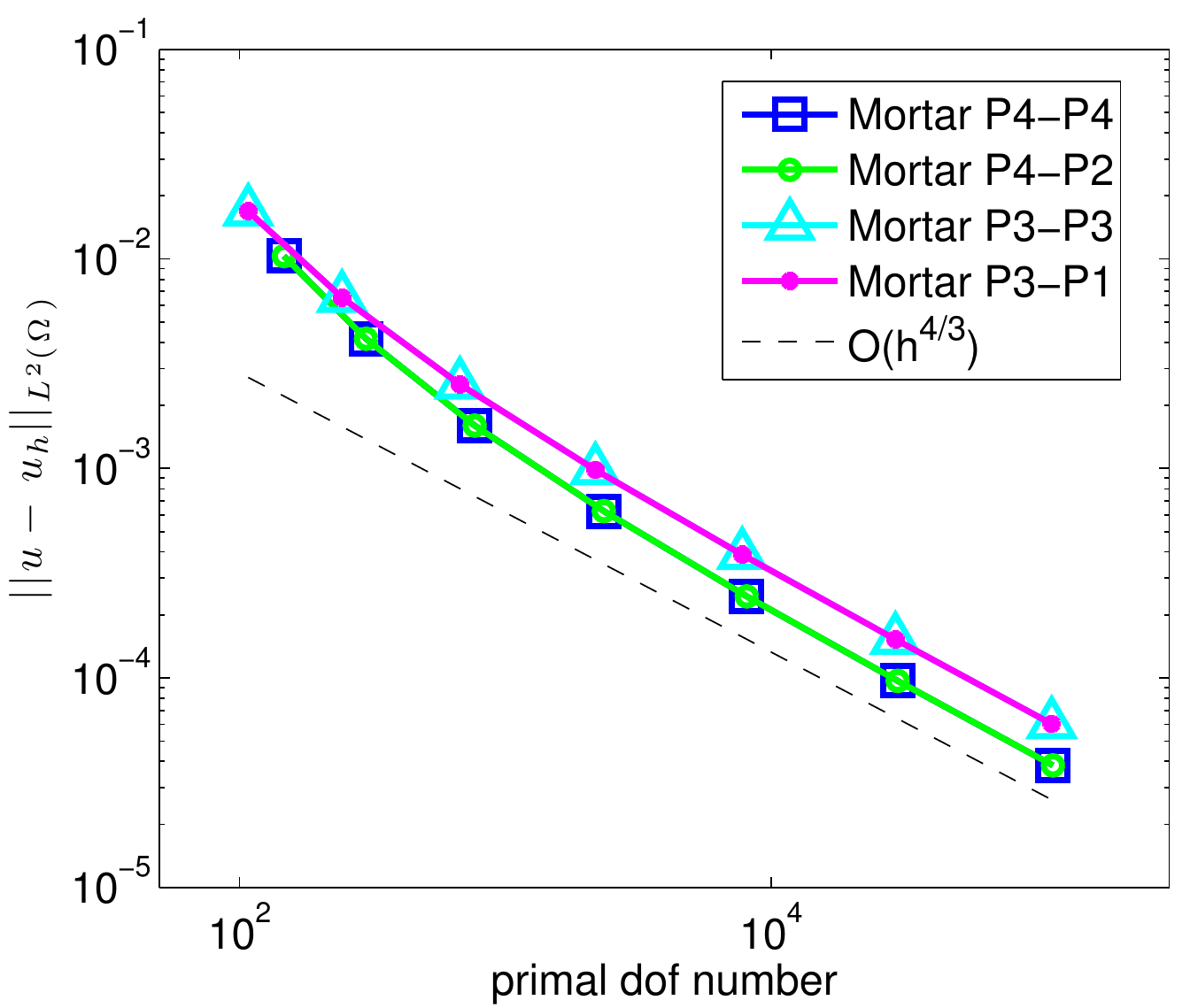}\,\,
	\includegraphics[width=0.47\textwidth]{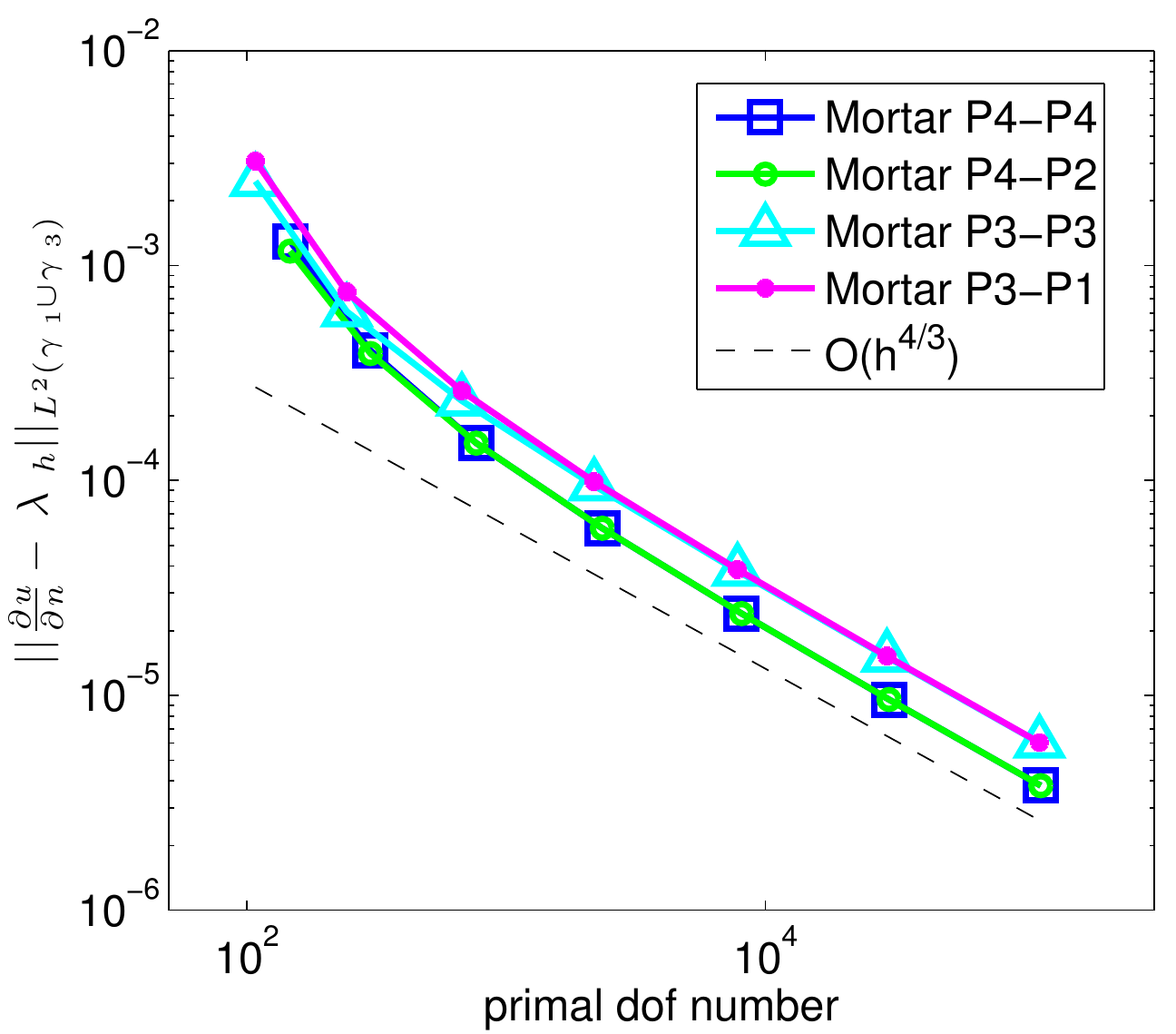}\,\,
	\includegraphics[width=0.47\textwidth]{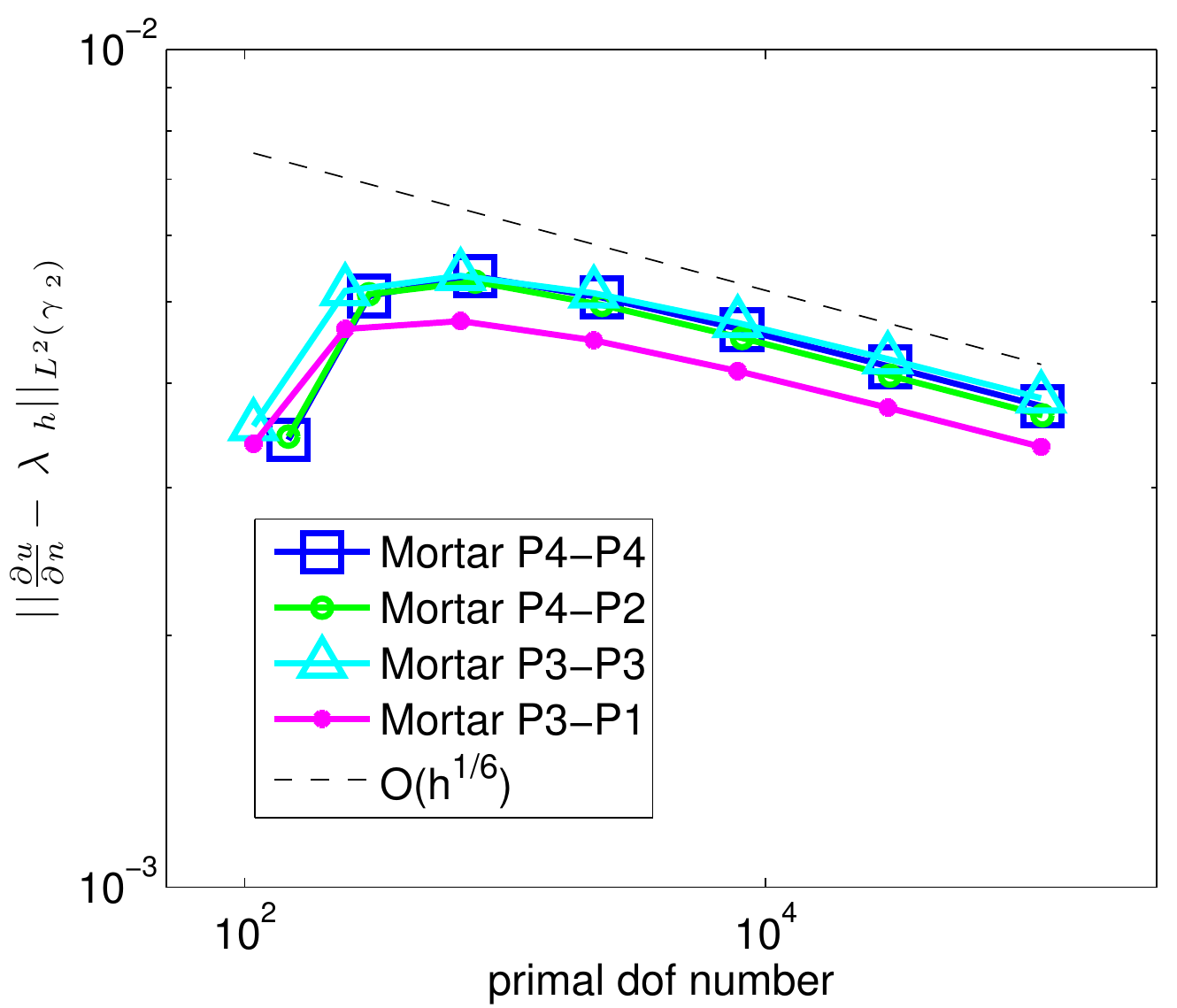}
	\caption{Problem of subsection~\ref{numerics:subsection_singular} - Error curves for several pairings. Top left: broken $V$ primal error. Top right: $L^2$ primal error. Bottom left: $L^2$ dual error at the interfaces 1 and 3. Bottom right: $L^2$ dual error at interface 2.}
	\label{fig:numerics:singular_cv_curve_PD_new} 
\end{figure}

\begin{figure}[htbp] 
	\includegraphics[width=0.47\textwidth]{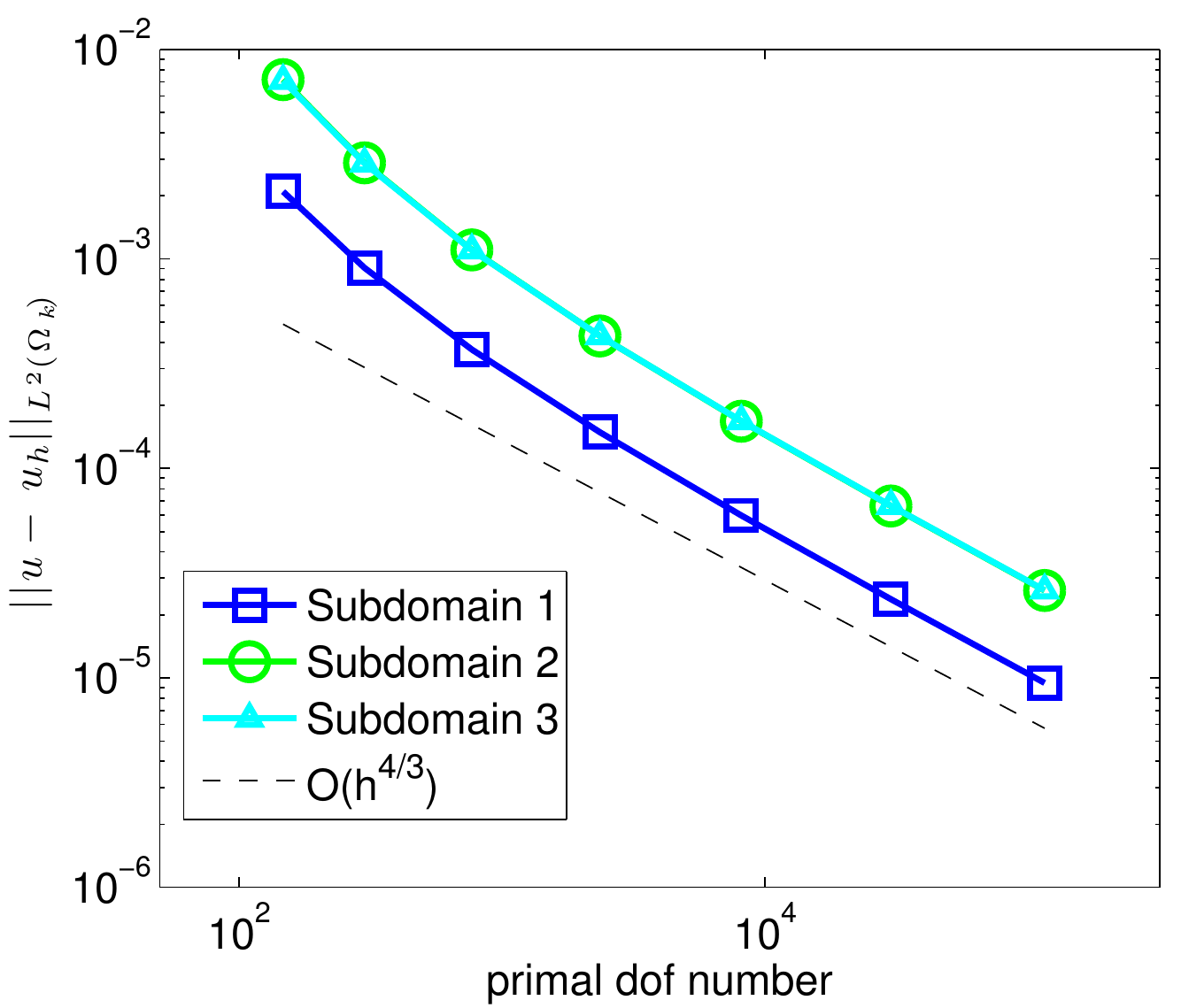}\,\,
	\includegraphics[width=0.47\textwidth]{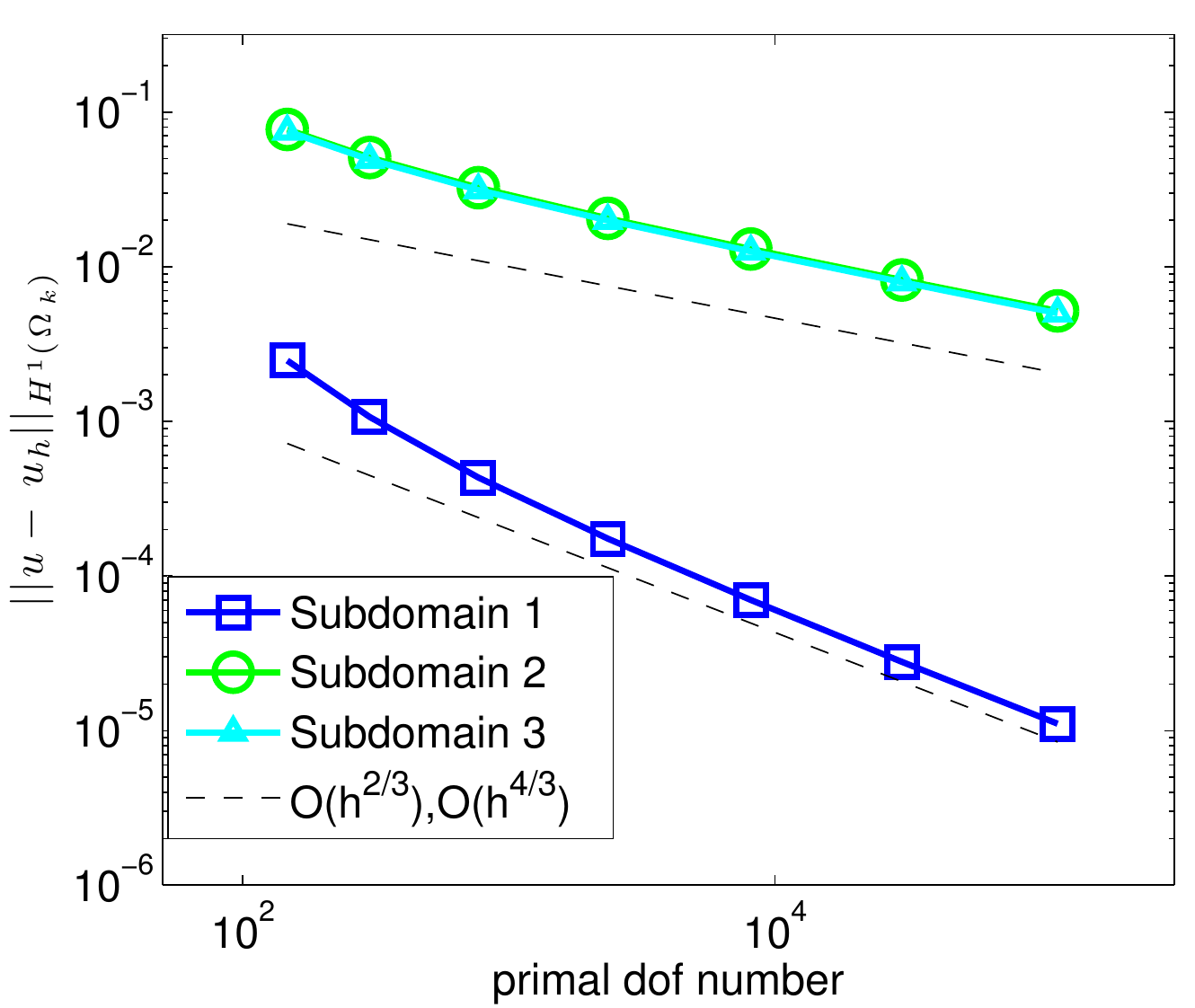}\\
	\includegraphics[width=0.47\textwidth]{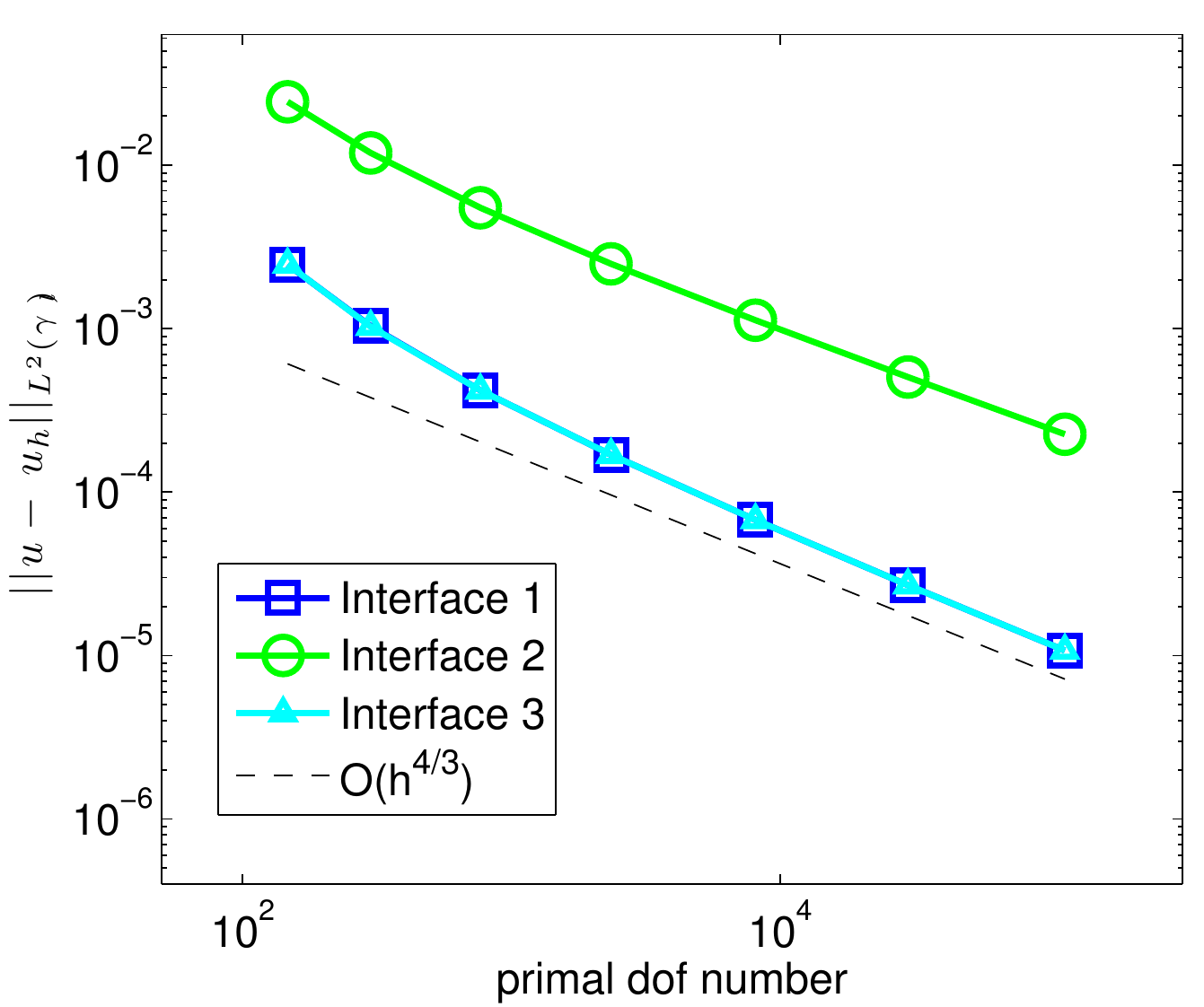}\,\,
	\includegraphics[width=0.47\textwidth]{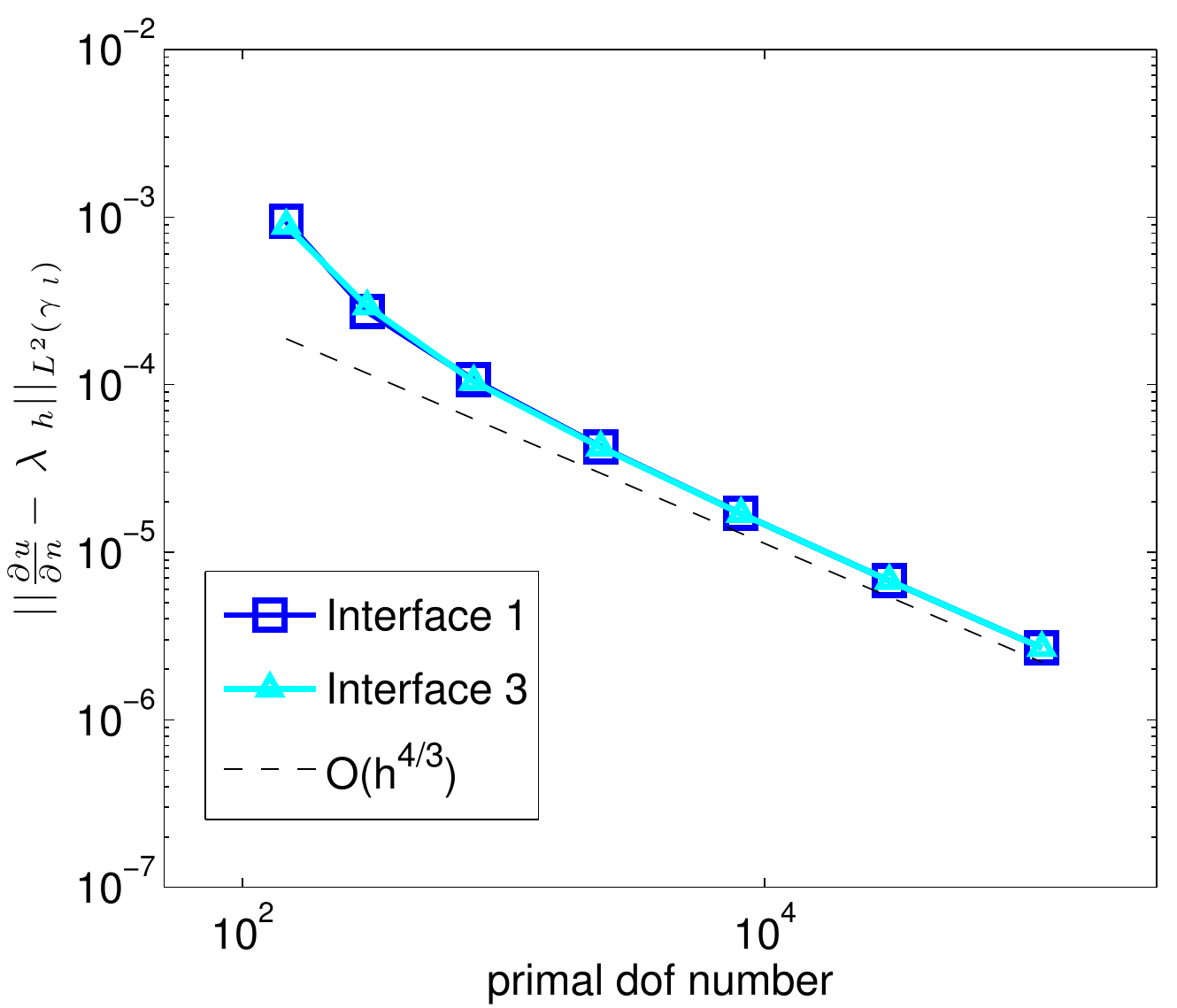}
	\caption{Problem of Subsection~\ref{numerics:subsection_singular} - Error curves for the pairing $P4-P4$. Top left: $L^2$ error on each subdomain. Top right: $H^1$ error on each subdomain. Bottom left: $L^2$ primal trace error at each interface. Bottom right: $L^2$ dual error at interface 1 and 3.}
	\label{fig:numerics:singular_cv_curve_4_new} 
\end{figure}

We also studied the error distribution over the different subdomains and interfaces, see Figure~\ref{fig:numerics:singular_cv_curve_4_new}. The results clearly show the pollution effect in the $L^2$ norm, i.e., also in the subdomain 1 far away from the singularity no better $L^2$ convergence rate can be observed. The situation is different if we consider the $H^1$ norm subdomain-wise. Here a better rate can be observed for subdomain 1 although it is significantly smaller than the best approximation rate restricted to this subdomain. This effect can be explained by local Wahlbin type error considerations in combination with the already mentioned pollution effect. Regarding the dual error, the same behavior as for the $H^1$ primal error is observed. 
This discrepancy between the interface 2 and the remaining interfaces can also be seen in the $L^2$ primal trace error. 

%
 \subsection{A scalar problem with jumping coefficients} \label{numerics:subsection_layers}
We consider the domain $\Omega = (0,2)\times (0,2.8)$ with homogeneous Dirichlet conditions applied on $\partial \Omega_D = (0,2)\times \{0, 2.8\}$ and homogeneous Neumann conditions on $\partial \Omega_N = \partial \Omega \backslash \partial \Omega_D$.

We consider three patches, with $\alpha$ being constant on each patch, see a distribution in Figure~\ref{fig:numerics:nh_setting}. The interface is a B-Spline curve of degree $3$ and exactly represented on the initial mesh.
 The external layers have the constant $\alpha = 1$, and the internal one $\alpha=1/100$ and the right hand side is $f=1$.  Due to the different values of $\alpha$, the mesh of the interior layer is chosen finer compared to the one of the other two layers. A uniform refinement starting from the initial mesh in Figure~\ref{fig:numerics:nh_setting} is performed.

In Figure~\ref{fig:numerics:nh_setting} the $L^2$ error of an equal degree pairing for $p=3$ and $p=4$ is shown. Lacking an exact solution, we compute the error by comparing to a reference solution, visible in Figure~\ref{fig:numerics:nh_results}. The reference solution is obtained by two more $h$-refinement steps starting from the finest mesh. 

We note that jumping coefficients can cause singularities in the cases, where more than two subdomains meet, although it is well-known that the case of a rectangular domain with interfaces parallel to the $x$-axis yields to a smooth solution. 

Numerically, we obtain optimal convergence for the case $p=3$, but, considering the convergence rate, there is no benefit of the degree elevation to degree $p=4$, which indicates that the solution is not sufficiently smooth. 
Further numerical investigations let us conjecture that this can have two reasons, one coming from the fact that the interface is not smooth enough to have higher regularity. In this example the interface was built from a B-Spline curve of degree $p=3$, hence the continuity on the interface is only $C^2$. This has an influence on the smoothness of the unit normal along the interface and thus on the smoothness of the solution. The other reason is to have corner singularities in the inner domain where interface meets the outer boundary. In this example, the angles were set to be $\pi/2$.

 \begin{figure}[htbp] 
\centering
	\includegraphics[width=0.32\textwidth]{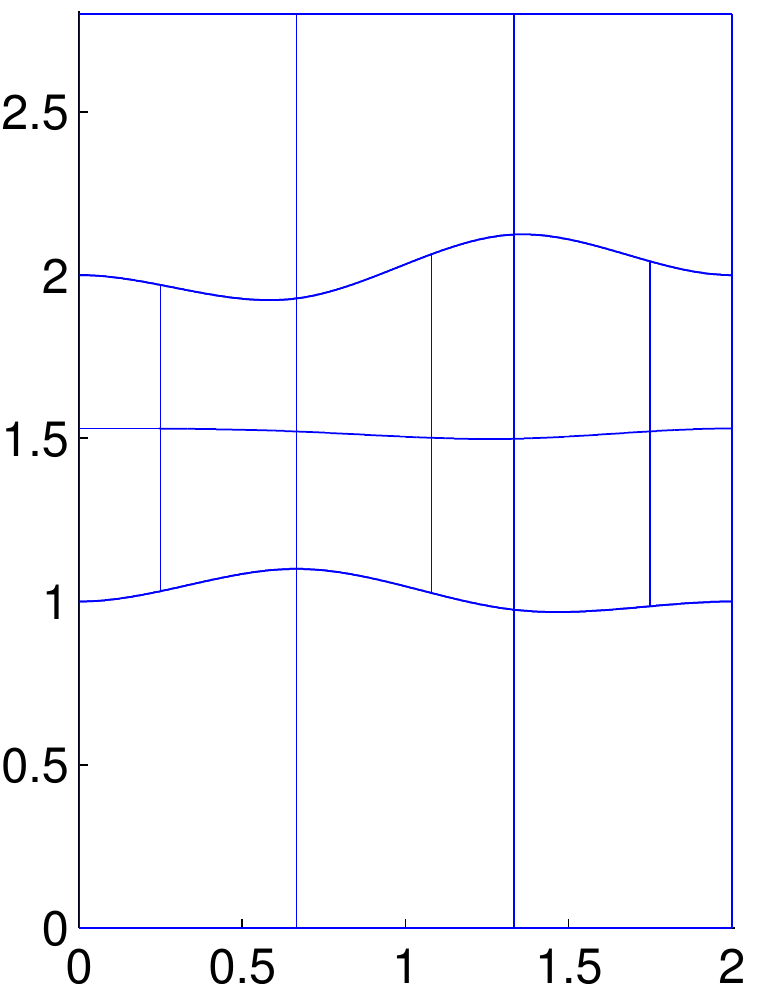} \hfill 
	\includegraphics[width=0.6\textwidth]{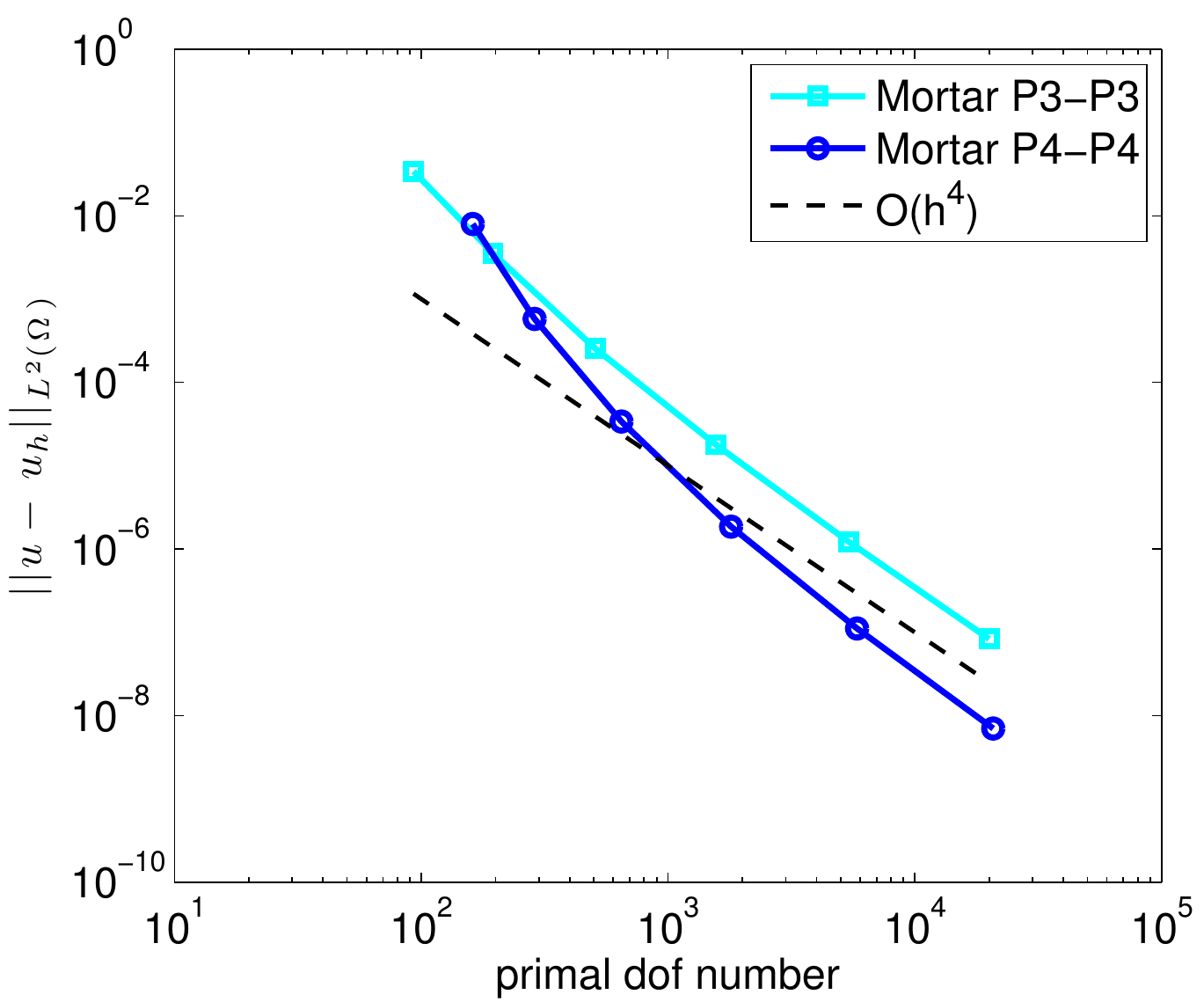}
	\caption{Problem of Subsection~\ref{numerics:subsection_layers} - Left: initial mesh. Right: primal $L^2$ error curves for two equal order pairings.}
	\label{fig:numerics:nh_setting} 
\end{figure}
\begin{figure}[htbp] 
\centering
	
	\includegraphics[width=0.7\textwidth]{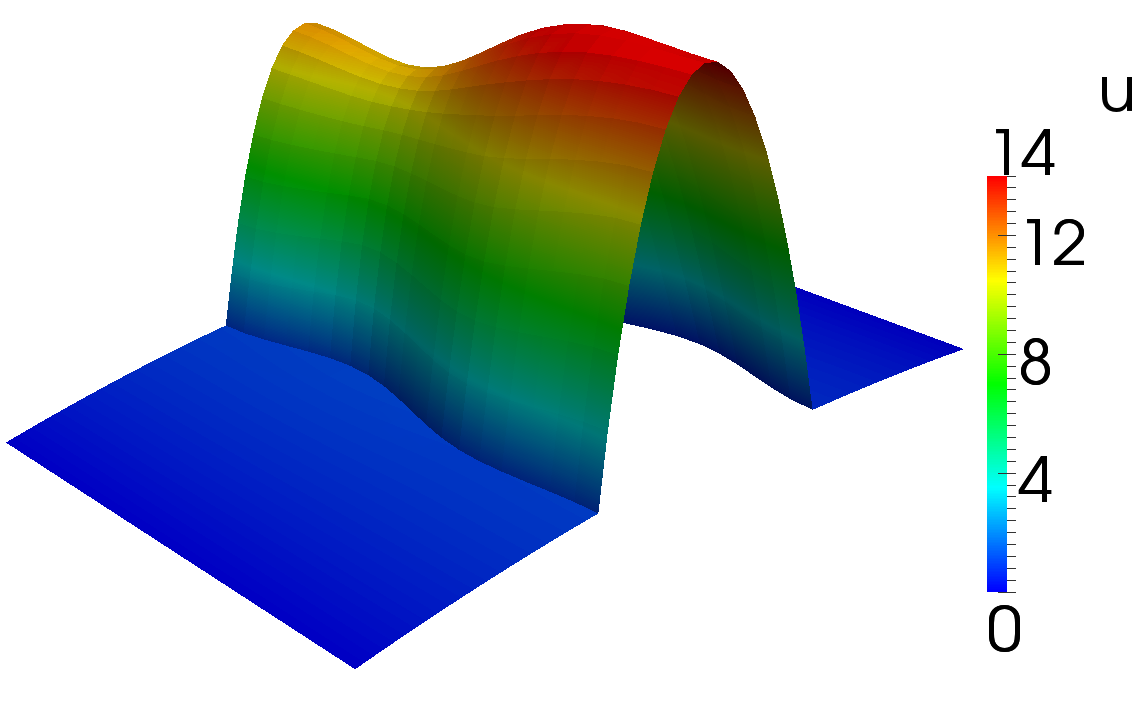}
	\caption{Problem of Subsection~\ref{numerics:subsection_layers} - Solution for the pairing $P3-P3$ on the finest mesh.}
	\label{fig:numerics:nh_results} 
\end{figure}
%
\subsection{A scalar problem on a two patch domain with a non-matching interface} \label{numerics:subsection_non_matching}
Let us consider the standard Poisson equation solved on the unit square $\Omega = (0,1)^2$, which is decomposed into two patches presented in Figure~\ref{fig:numerics:non_matching}. As the subdomains cannot exactly be represented by the chosen spline spaces for the geometry approximation, the subdomains do not match at the interface, see Figure~\ref{fig:numerics:non_matching}. And thus, due to this geometry approximation an additional variational crime is introduced in the weak problem formulation.

The internal load and the boundary conditions have been manufactured to have the analytical solution $u(x,y)= \sin(5 y) \sin(6 x)$. To measure the influence of the geometrical approximation on the mortar method accuracy, we consider the same degree pairing and note that in this case no boundary modification is required. This can be granted by setting homogeneous Neumann conditions on $\partial \Omega_N = \{0,1\}\times (0,1)$ and Dirichlet conditions on $\partial \Omega \backslash \partial \Omega_N$

Firstly, in the top row of Figure~\ref{fig:numerics:non_matching_cv_curve_new}, we show the numerically obtained error decay in the $L^2$ norm. As expected from the theory, we observe for an equal order $p$ pairing a convergence order of $p+1$ for the primal variable. We also compare the primal error of a matching and non-matching mesh situation. As Figure~\ref{fig:numerics:non_matching_cv_curve_new} shows, no significant quantitative difference can be observed in the asymptotical behavior.
Note, that the optimal primal $L^2$ rates are in accordance with the theory of finite element methods,  see~\cite{wohlmuth:09}. Moreover, the results of the bottom right picture of Figure~\ref{fig:numerics:non_matching_cv_curve_new} show even higher rates for the dual variable than expected from the theory. 

Secondly, we consider different degree pairings in order to see the accuracy of the reduced order mortar method for a problem containing an additional approximation. 
In the lower row of Figure~\ref{fig:numerics:non_matching_cv_curve_new}, the $L^2$ error of the the primal variable and of the dual variable for the pairing $P4-P2$ and $P3-P1$ is given.  We note that a lower dual degree does not deteriorate the accuracy on the primal variable. From the theoretical point of view, it is obvious that a $p/p-2$ pairing gives a priori results for the Lagrange multiplier which are of the same order as the best approximation of the dual space. However, this is not the case for the primal variable. Theorem~\ref{thm:convergence_rates} indicates that for this case a $\sqrt{h}$ is lost. This is not observed in our situation, see  Figure~\ref{fig:numerics:non_matching_cv_curve_new}. This might be a consequence of superconvergence arguments which can possibly recover an extra order of $\sqrt{h}$ on uniformly refined meshes.

\begin{figure}[htbp] 
	\centering
	\includegraphics[width=0.4\textwidth]{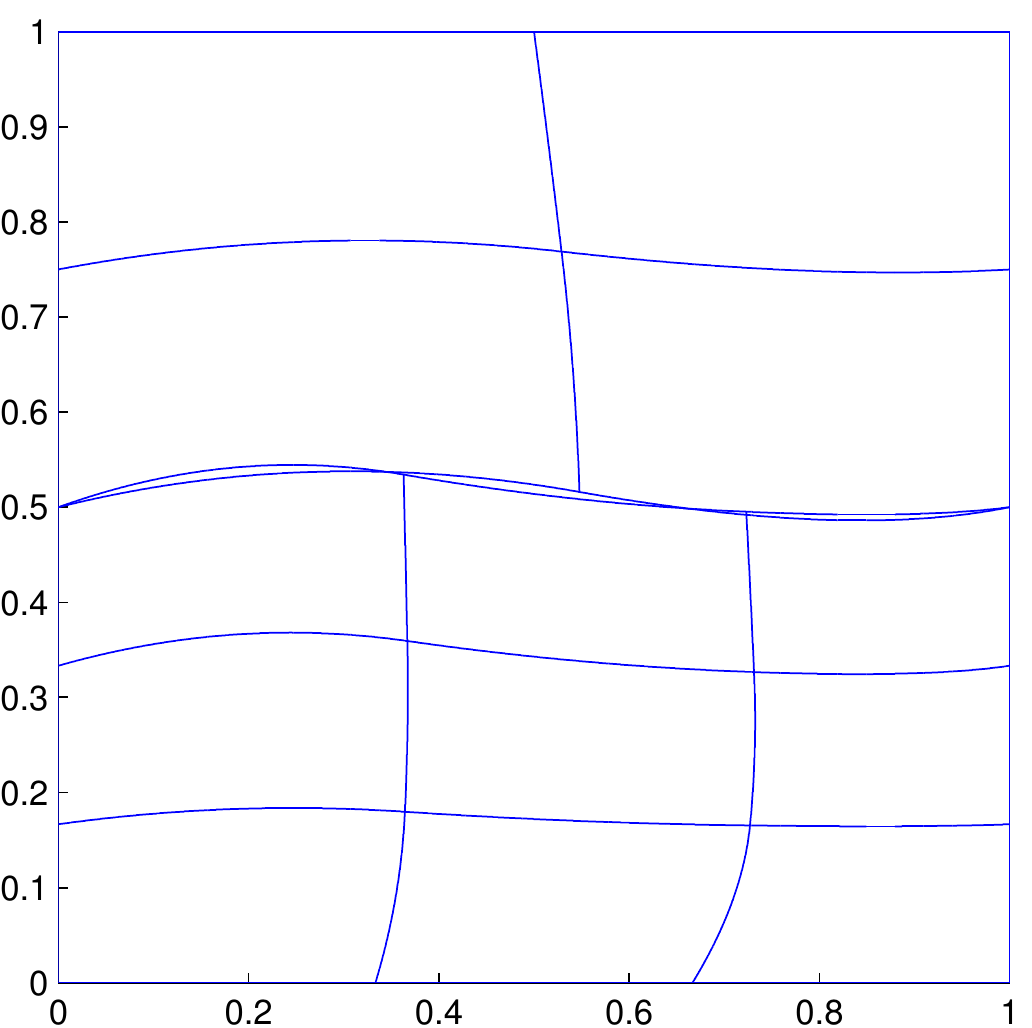}	\caption{Problem of Subsection~\ref{numerics:subsection_non_matching} - Non-conforming mesh with a non-matching interface.}
	\label{fig:numerics:non_matching} 
\end{figure}

\begin{figure}[htbp]  
	\includegraphics[width=0.57\textwidth]{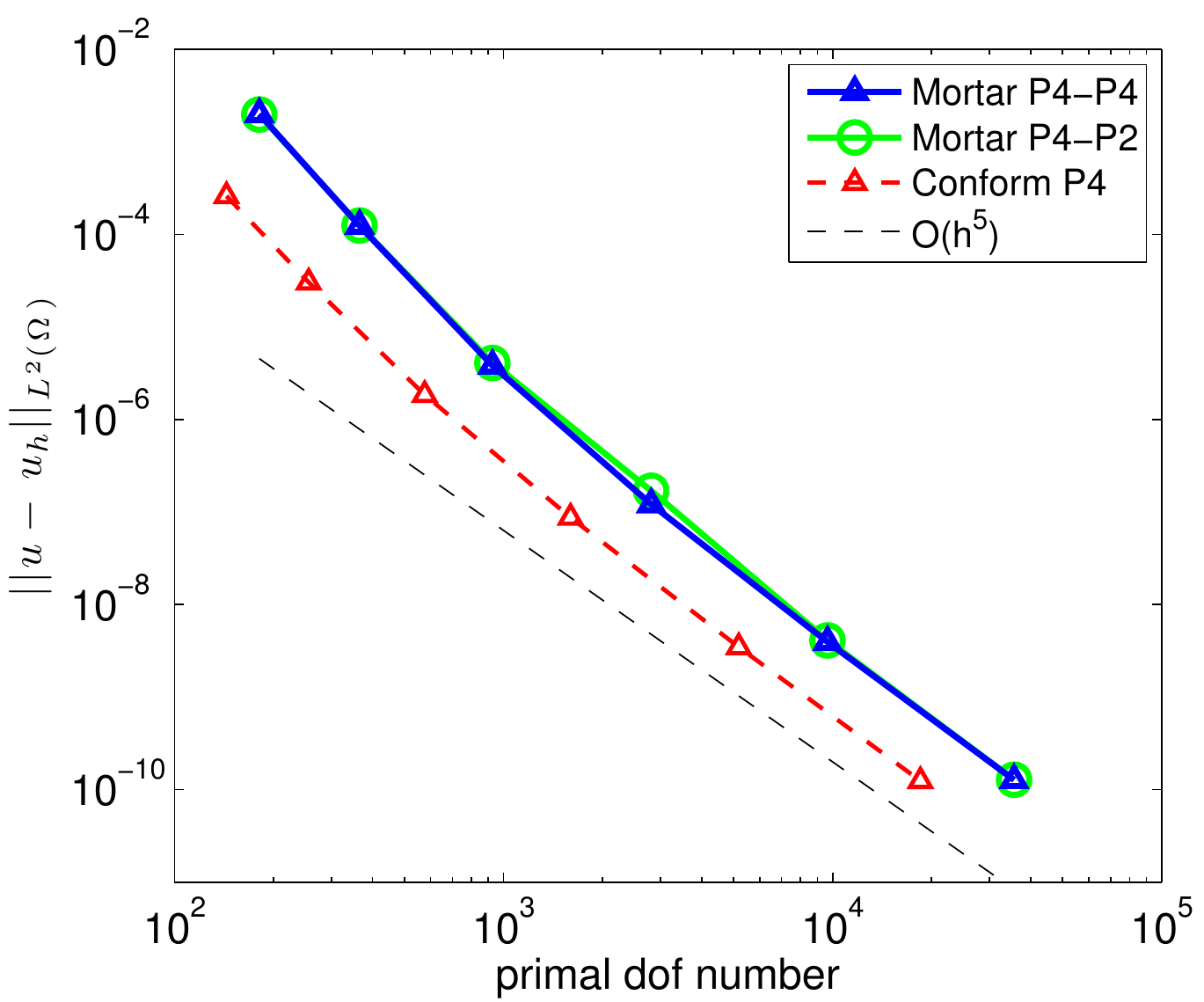}\,\,
	\includegraphics[width=0.57\textwidth]{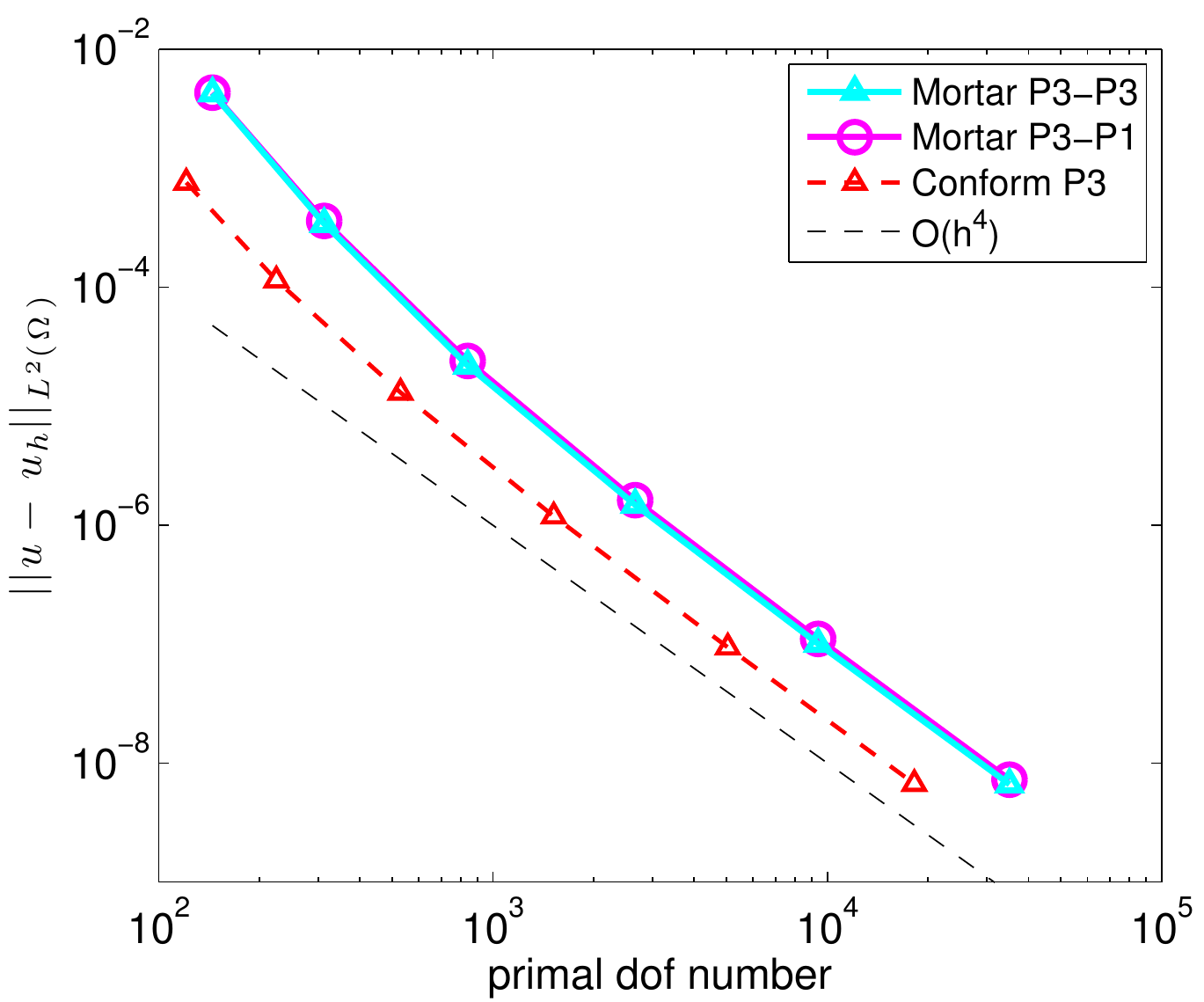}\\
	\includegraphics[width=0.57\textwidth]{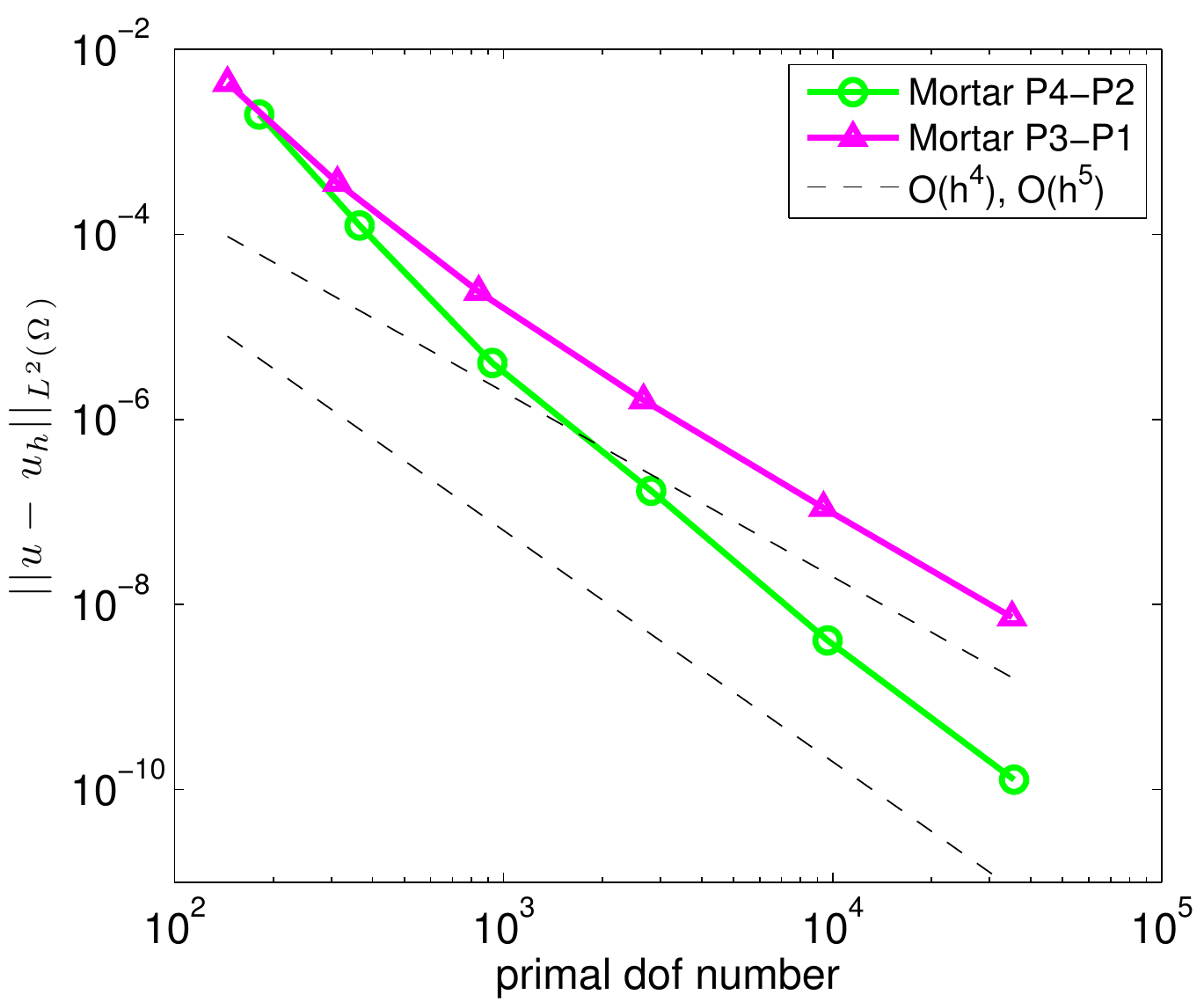}\,\,
	\includegraphics[width=0.57\textwidth]{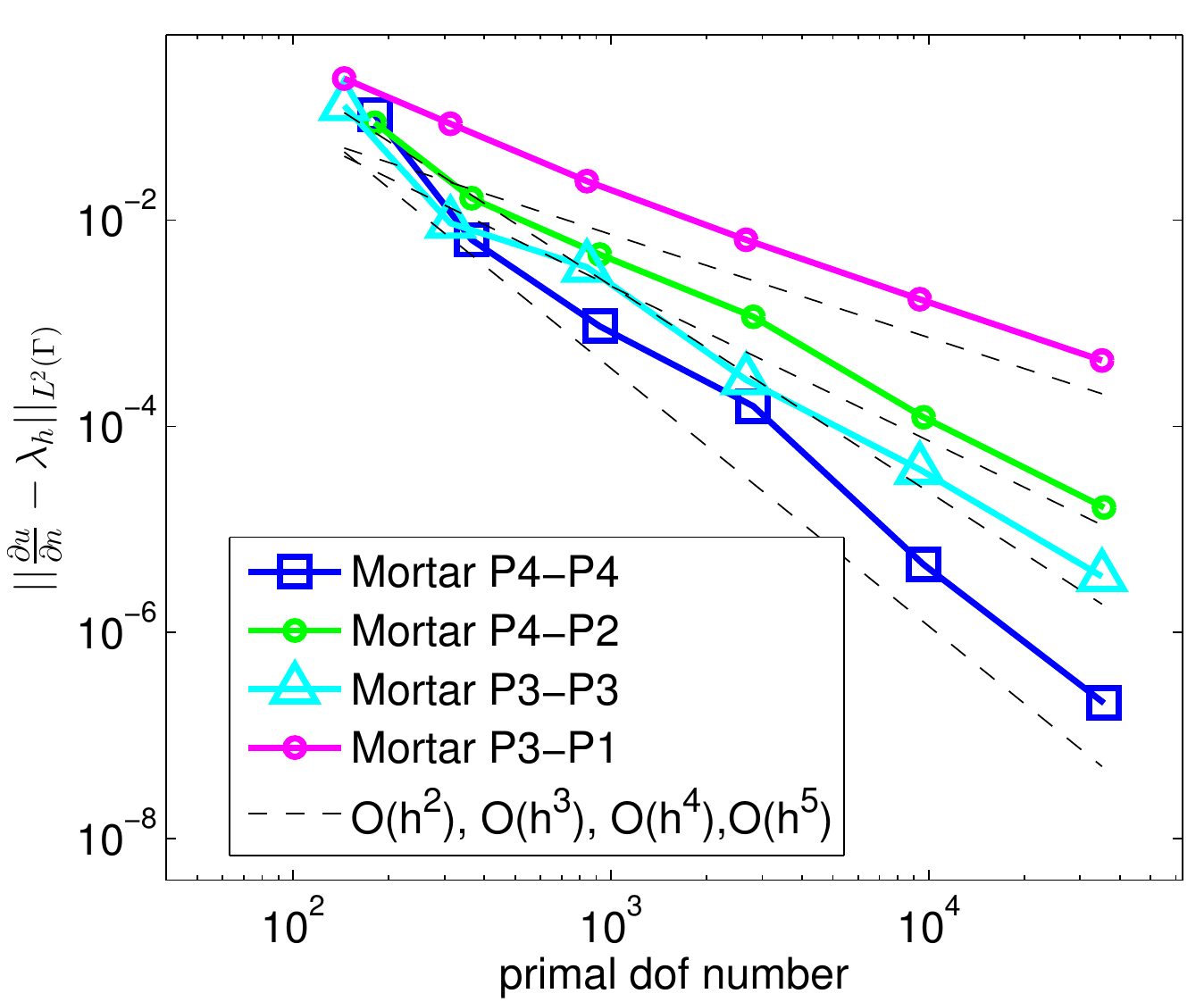}
	\caption{Problem of Subsection~\ref{numerics:subsection_non_matching} - Several $L^2$ error curves. Top left: primal error for stable pairings of primal degree $p=4$. Top right: primal error for stable pairings of primal degree $p=3$. Bottom left: direct comparison of the primal error for pairings $P4-P2$ and $P3-P1$. Bottom right: dual error for  stable pairings of primal degree $p=3$ and $p=4$.} 
	\label{fig:numerics:non_matching_cv_curve_new}
\end{figure}
To conclude, this example shows that the influence of the additional geometry error in the mortar method context is quite small.

%
\subsection{A linear elasticity problem} \label{numerics:subsection_linear_elasticity}
Let us define the mechanical equilibrium on a domain $\Omega$ as:
\begin{displaymath}
\begin{array}{llll}
	-\divergence(\underline{\underline{\sigma}})&=&\underline{f} & \textrm{in } \Omega,\\
	\underline{u}&=&\underline{u_D} &\textrm{on } {\partial \Omega}_D,\\
	\underline{\underline{\sigma}} \cdot \underline{n}&=&\underline{g}  & \textrm{on } {\partial \Omega}_N.\\
\end{array}
\end{displaymath}
In a plane linear isotropic elastic context, we have the following relations between the stress tensor $\underline{\underline{\sigma}}$, the strain tensor $\underline{\underline{\varepsilon}}$ and the displacement $\underline{u}$:  
\[
\underline{\underline{\sigma}}=\lambda \, tr(\underline{\underline{\varepsilon}}) \, \underline{\underline{\mathbb{I}}}\,+\,2\,\mu \,\underline{\underline{\varepsilon}}, \quad \underline{\underline{\varepsilon}}=\underline{\underline{\nabla^su}}= \frac{1}{2} (\underline{\underline{\nabla u}} \,+\, \underline{\underline{\nabla^T u}}).
\]
where $\divergence$, $\nabla$, $\underline{n}$, $\underline{f}$, $\underline{u_D}$, $\underline{g}$, $\lambda$ and 
$\mu$ stand respectively for the standard divergence operator, the gradient operator, the unit outward normal to $\Omega$ on $\partial \Omega$, the prescribed data values on ${\partial \Omega}_D$ and on ${\partial \Omega}_N$ and the Lam\'e coefficients.

Let us consider the problem of an infinite elastic plate with a circular hole subjected to tension loading in $x=-\infty$ and $x=+\infty$. Considering the load and the boundary condition symmetries, only a quarter of the plate is modeled. This test, which has an analytical solution,~\cite{timo:51}, is a typical benchmark in isogeometric analysis because the NURBS offer the possibility to exactly represent the geometry. However, it cannot be parametrized smoothly in a one patch setting, so it is worth to consider it within a domain decomposition approach such as the mortar method.

\begin{figure}[htbp]  
	\includegraphics[width=0.315\textwidth]{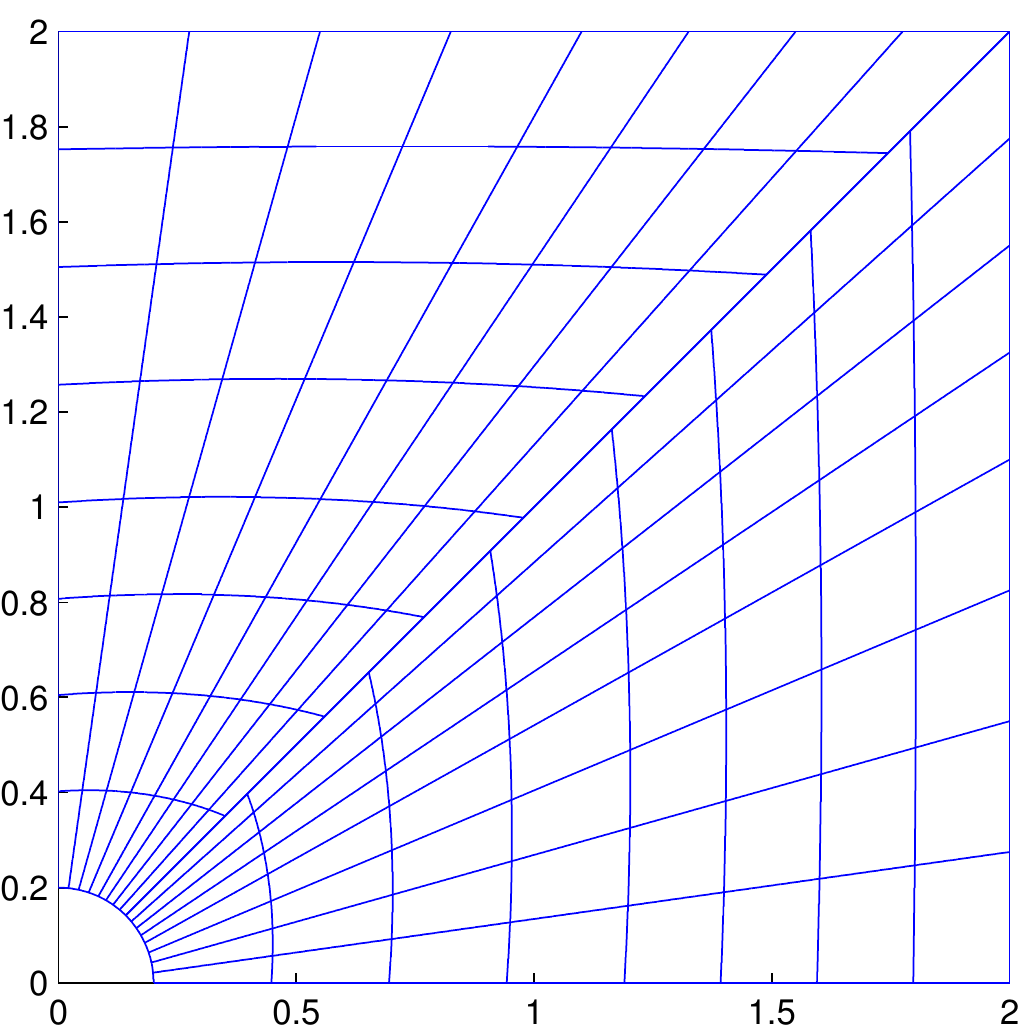}
	\includegraphics[width=0.31\textwidth]{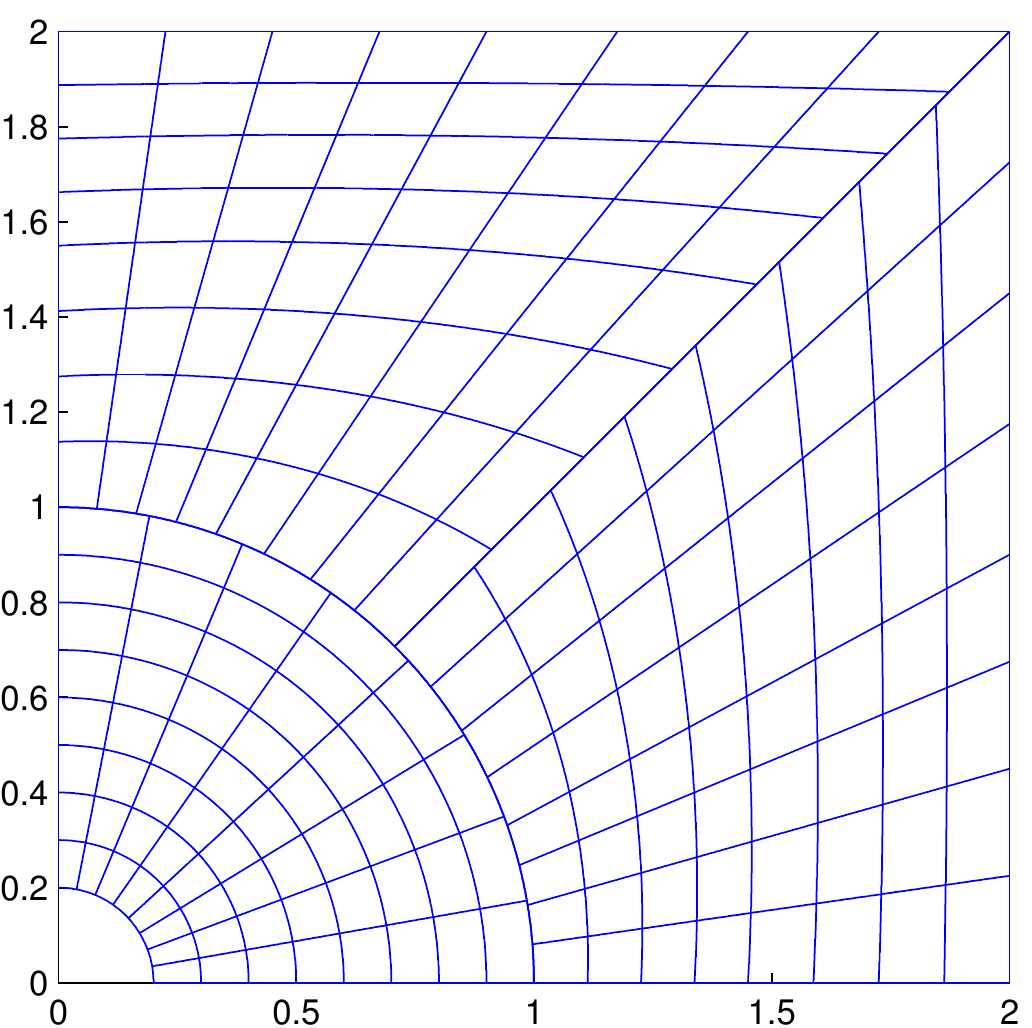}
	\includegraphics[width=0.31\textwidth]{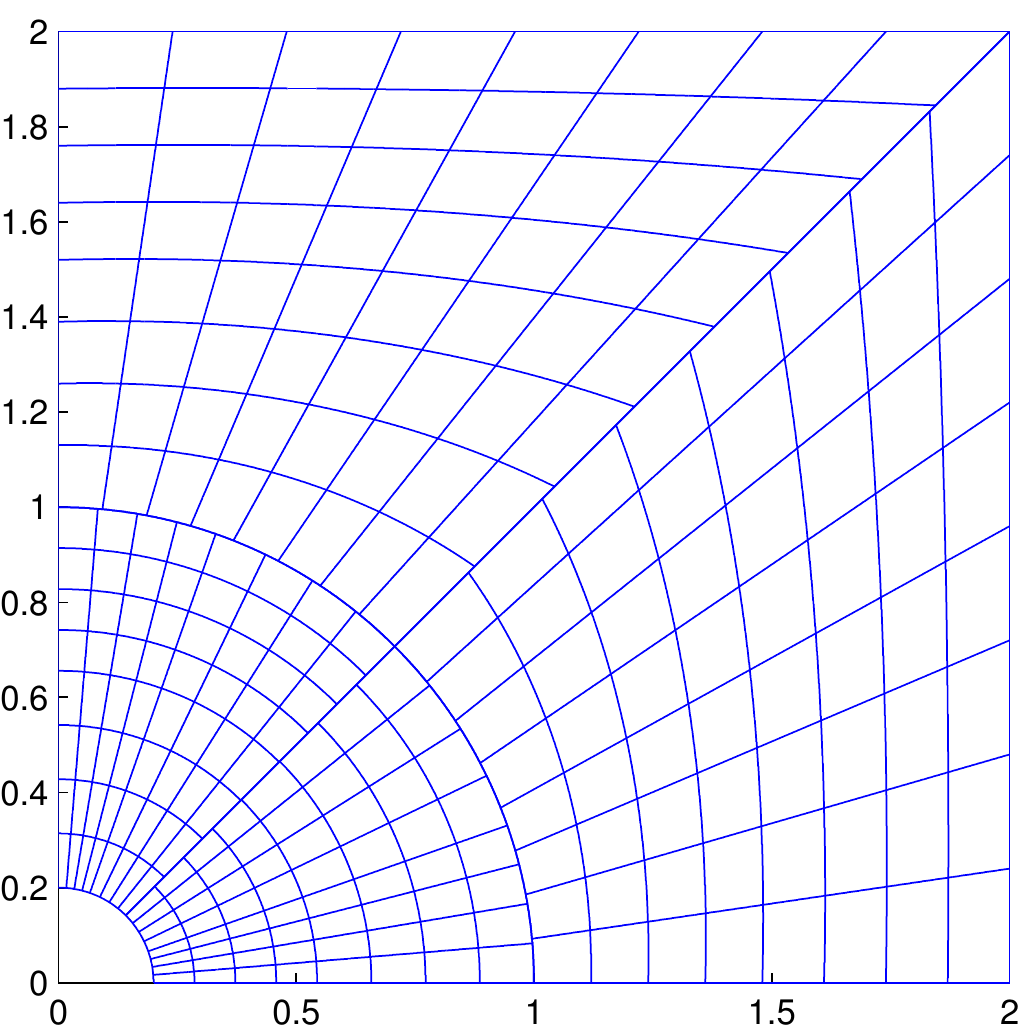}
	\caption{Problem of Subsection~\ref{numerics:subsection_linear_elasticity} - Different parametrizations of the infinite plate with a hole. From left to right: 2, 3 and 4 subdomains.}
	\label{fig:numerics:hole_setting}
\end{figure} 

\begin{figure}[htbp]  
	\includegraphics[width=0.485\textwidth]{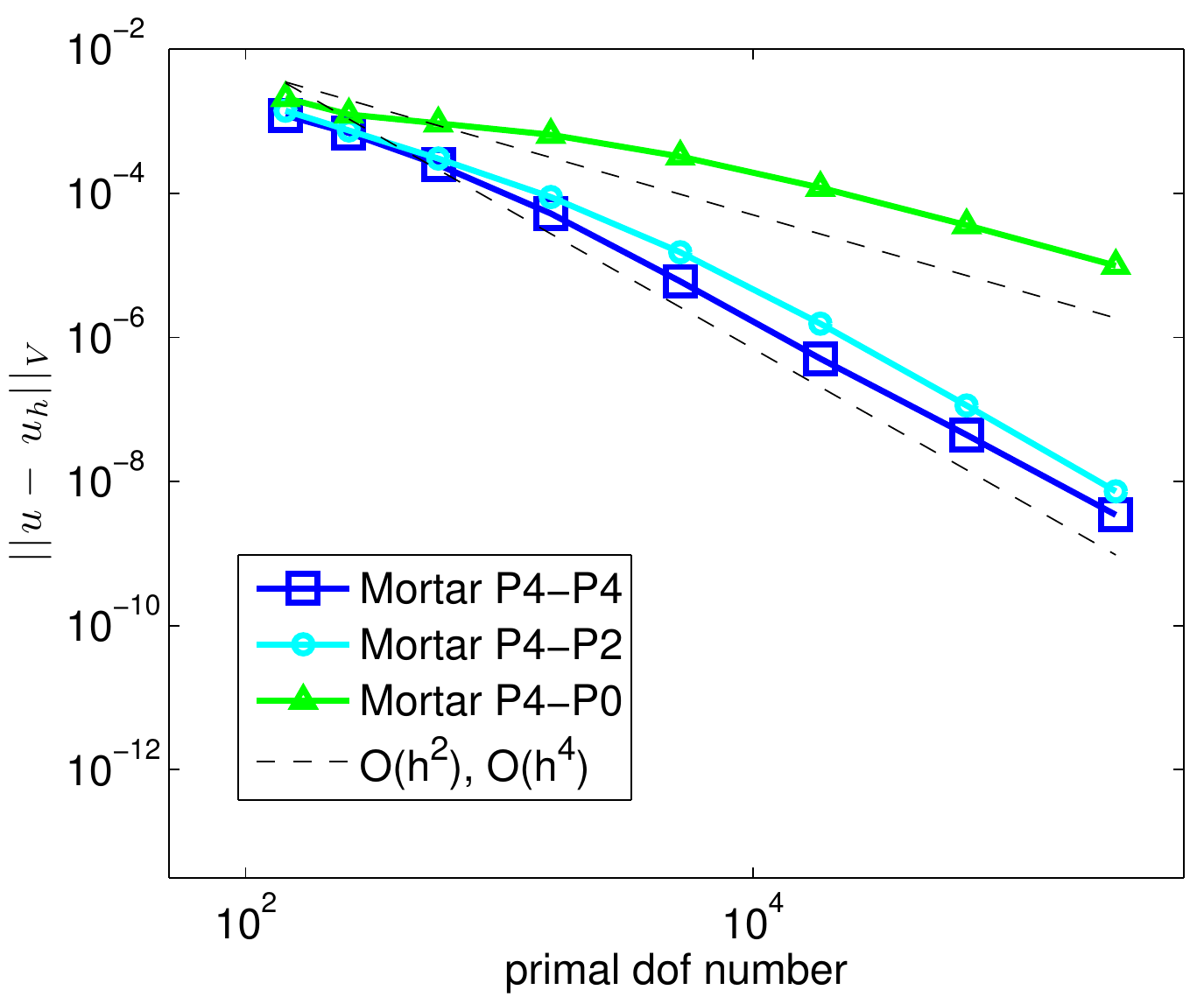}\,\,
	\includegraphics[width=0.47\textwidth]{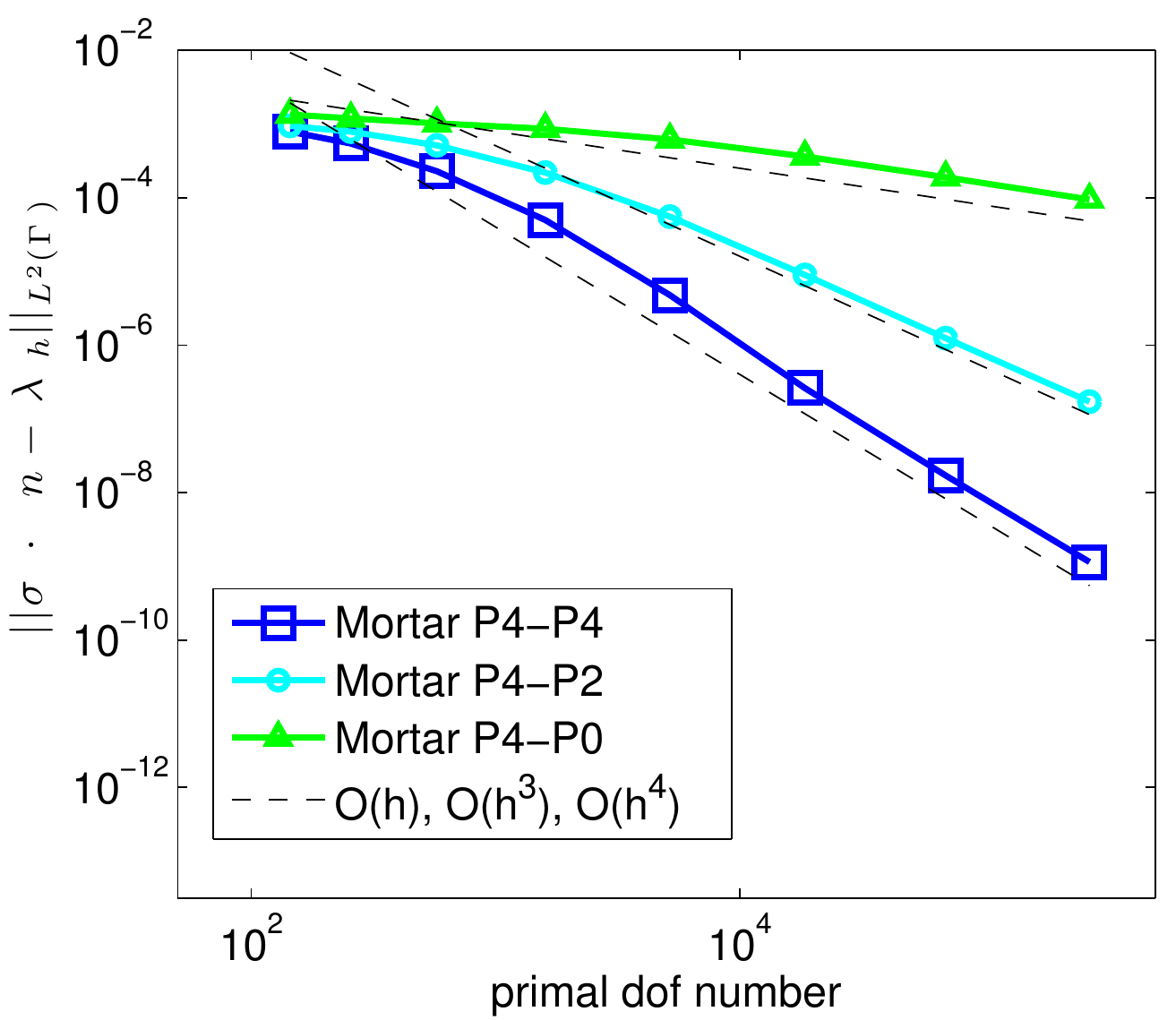}\\
         \includegraphics[width=0.485\textwidth]{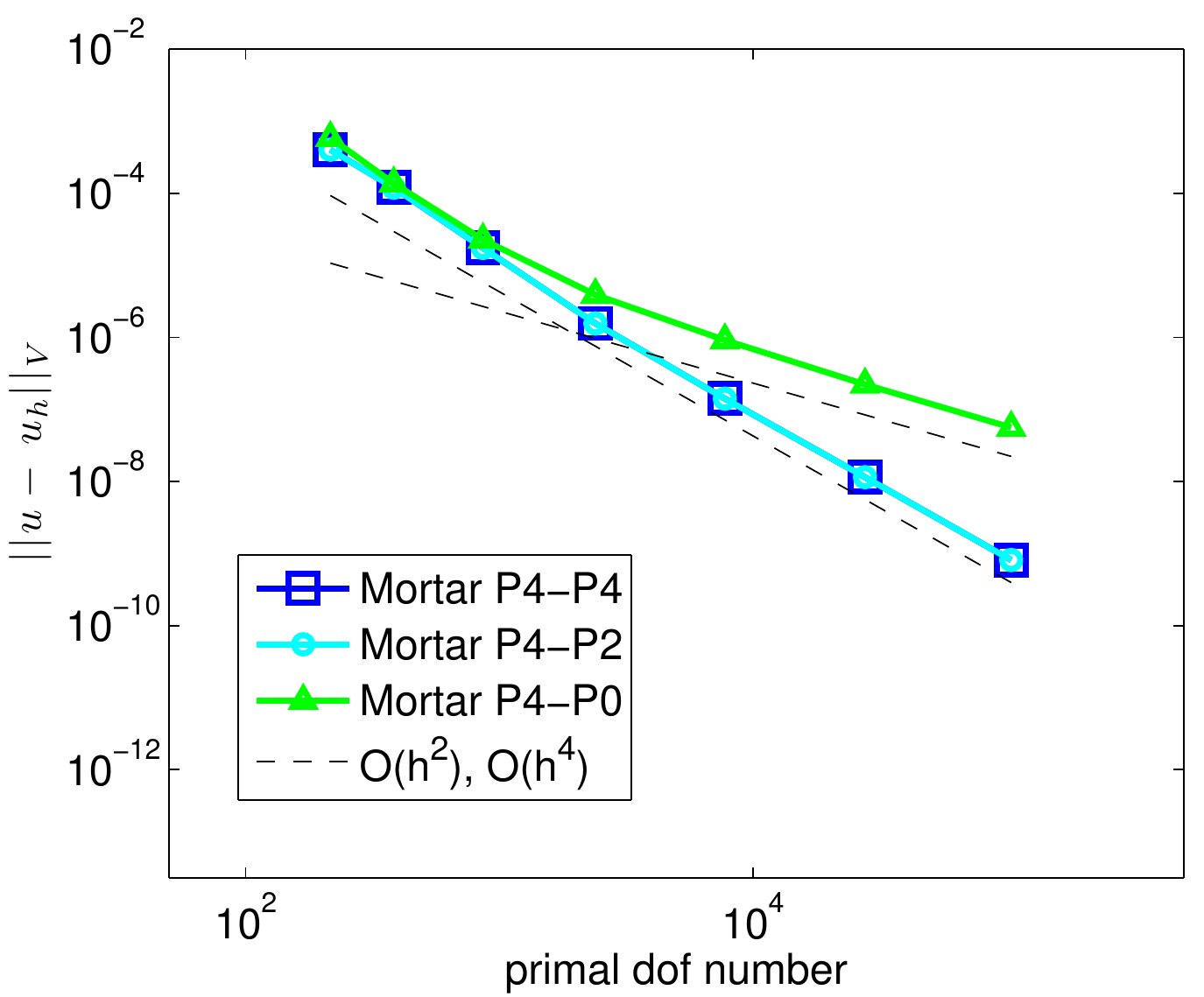}\,\,
	\includegraphics[width=0.47\textwidth]{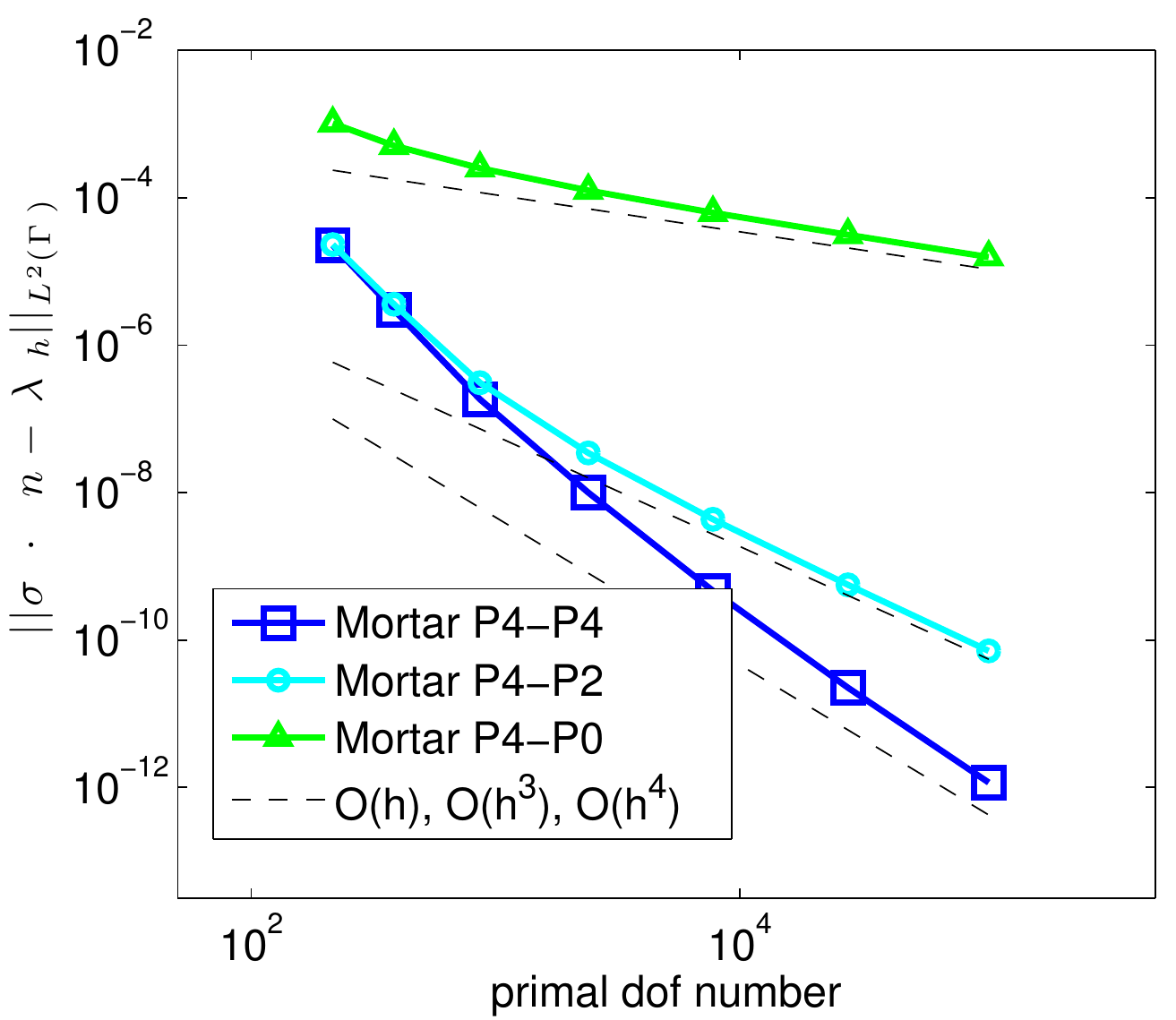}\\
	\includegraphics[width=0.485\textwidth]{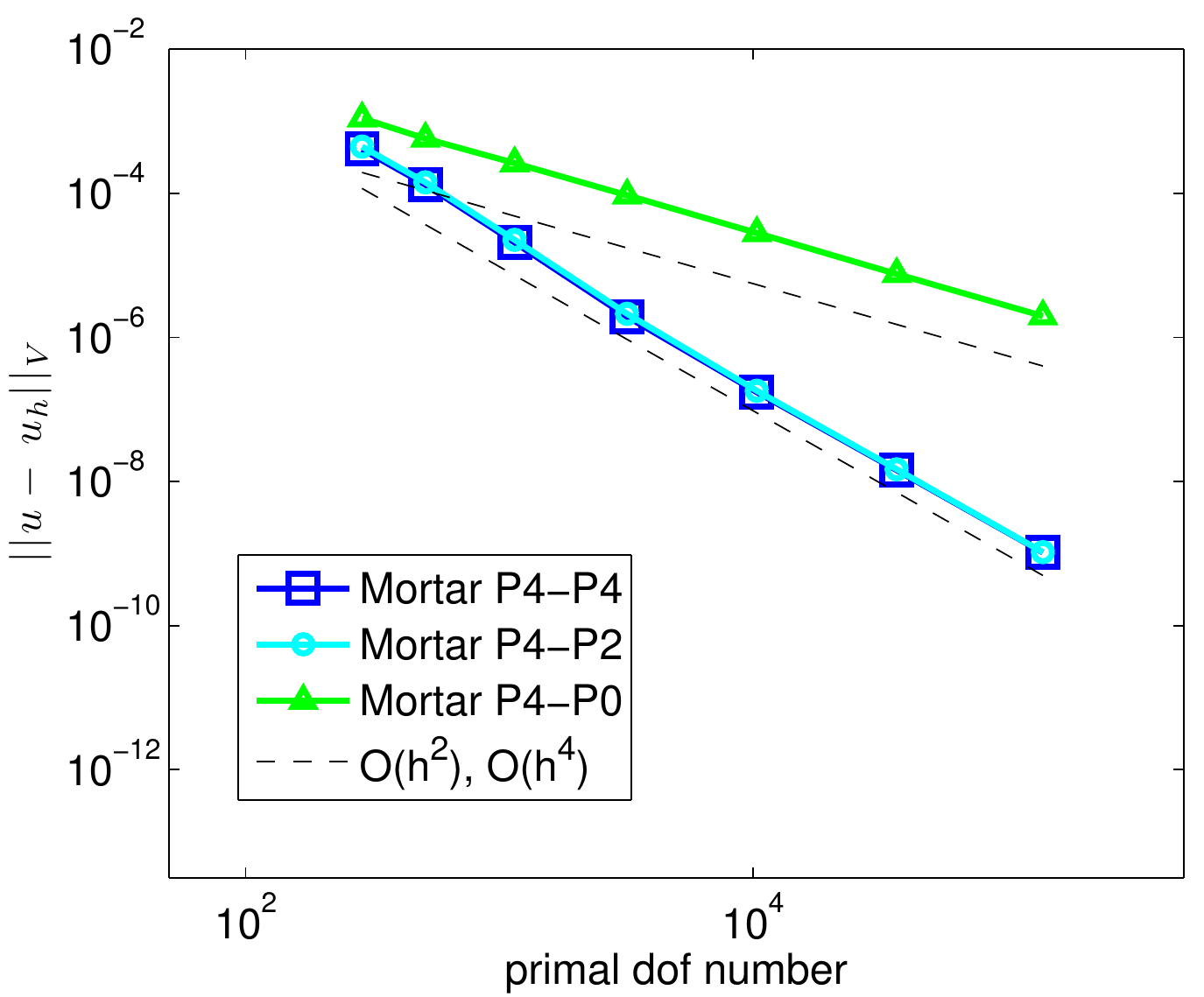}\,\,
	\includegraphics[width=0.47\textwidth]{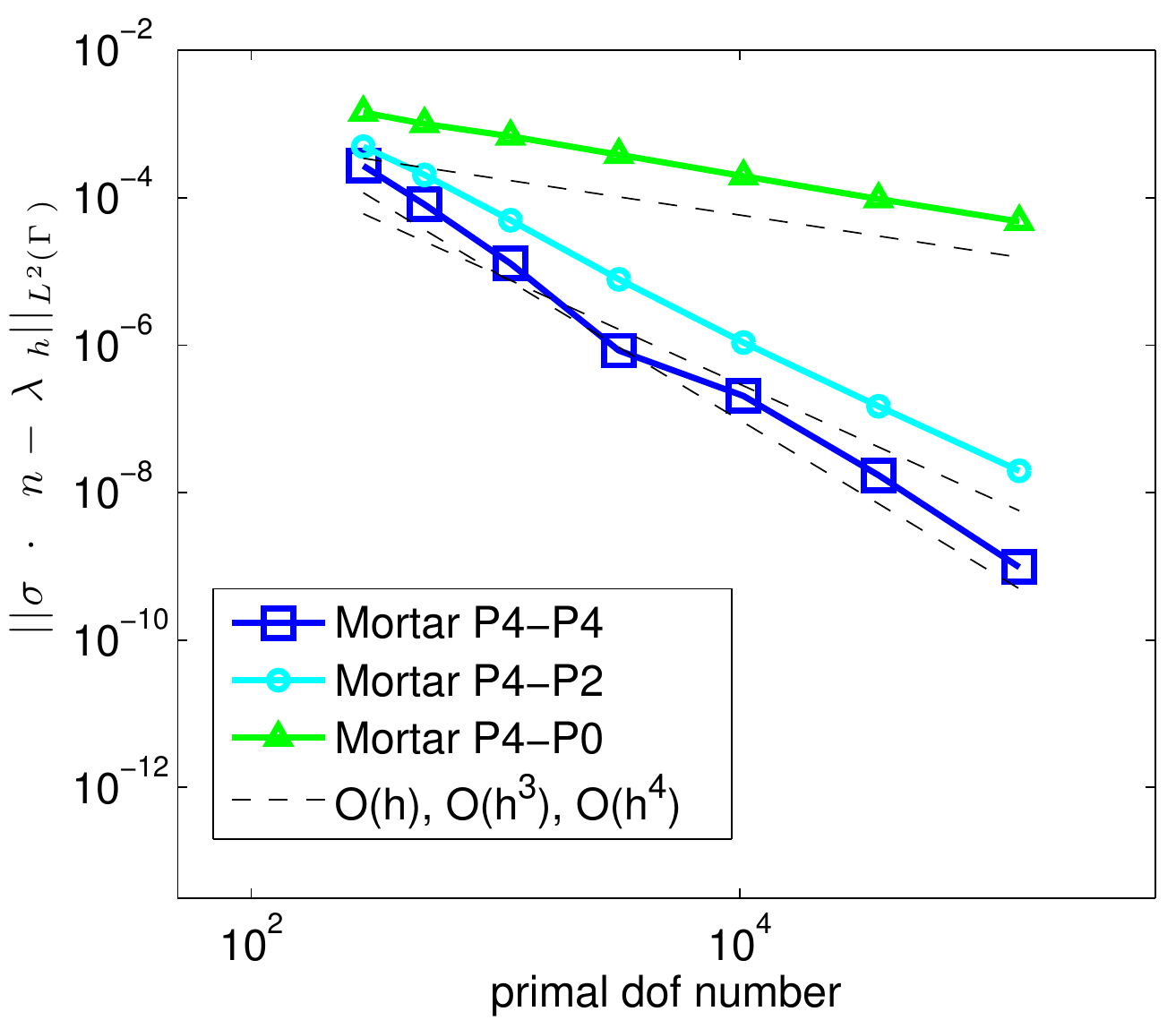}
	\caption{Problem of Subsection~\ref{numerics:subsection_linear_elasticity} - Left: broken $V$ primal error curves. Right: $L^2$ dual error curves. Respectively from the top to the bottom, for the 2, 3 and 4 patch parametrizations given on Figure~\ref{fig:numerics:hole_setting}, for several degree pairings.} 
	\label{fig:numerics:hole_errors_variousP4}
\end{figure}

We consider a domain $\Omega = \{ (x,y)\in (0,2)^2: x^2 + y^2 > 0.04\}$, shown in Figure~\ref{fig:numerics:hole_setting}, apply the exact pressure on $\partial \Omega_N= \{2\}\times(0.2,2) \cup (0.2,2)\times\{2\}$ and the symmetry condition on $\partial \Omega_{D_1} =  \{0\}\times(0.2,2)$ and $\partial \Omega_{D_2} =(0.2,2)\times\{0\}$.

Let us consider three different parametrizations of this test. First, two geometrically conforming cases which are constituted by 2 and 4 patches, respectively (see the left and the right pictures of Figure \ref{fig:numerics:hole_setting}). Only in the four patches situation, we have cross points where the boundary modification of the dual space is required. Secondly, let us consider a slave geometrical conforming case constitutes by 3 patches (see  middle  of  Figure  \ref{fig:numerics:hole_setting}) for which the boundary modification is necessary considering the same degree pairing. In each case, the results are compared to the analytical solution, and a numerical convergence study is presented.

As it is visible in the left column of Figure~\ref{fig:numerics:hole_errors_variousP4} for the broken $V$ error of the primal variable, the mortar methods considering the same degree pairing with the correct boundary modifications remain optimal in all the cases.
Moreover, we have also considered different degree pairings, and observed numerically the optimality of the methods. We note that even if we were expecting from the theory a reduced order regarding the convergence of the primal variable in broken $V$ norm of the pairing $P4-P2$, we numerically obtain for some parametrization a better order.
Additionally, in the right column of Figure~\ref{fig:numerics:hole_errors_variousP4}, the $L^2$ error of the dual variable is given for the primal degree $p=4$ and its corresponding stable reduced degrees. As already observed several times, we obtain the best approximation rates for the different degree pairings.

\section{Conclusion}
In this article an isogeometric mortar formulation was presented and investigated from a mathematical and a practical point of view. For a given primal order $p$, dual spaces of degree $p$, $p-1$ and $p-2$ were considered.
While the pairing $p/p-1$ was proven unstable, the others satisfied this condition, noting that the stability is achieved for the same degree pairing because of a boundary modification. For a given primal space, the proposed mortar methods are such that the equal order pairing guarantees optimal results, while for the pairing $p/p-2$ the convergence order can be reduced by at most $1/2$. However, we note that a boundary modification always yields additional effort for the implementation and the data structure.

Numerical examples showed that the mortar method can also handle further difficulties arising from geometry approximations and is not perturbed by singularities. Also in several cases the obtained convergence order was superior to the theoretical results.

The application of mortar methods in the isogeometric analysis in not restricted to linear problems. Since isogeometric discretizations have recently given promising results in contact problems, the 
application of the stated mortar spaces tailored to contact problems is a subject of a ongoing research.

\section*{Acknowledgments}
The first author has been supported by Michelin under the contract A10-4087, this support is gratefully acknowledged.
The third and the fourth authors have been supported by the International Research Training Group IGDK 1754, funded by the German Research Foundation (DFG) and the Austrian Science Fund (FWF). This support is gratefully acknowledged.

\bibliography{./bib/bibliography}

\end{document}